\documentclass{article} 
\usepackage{ssArxiv}

\usepackage{microtype}
\usepackage{graphicx}
\usepackage{subcaption}  %
\usepackage{booktabs}
\usepackage{color}
\usepackage{setspace} %
\usepackage{algorithm}
\usepackage{algpseudocode}
\usepackage{hyperref}
\usepackage{natbib}
\usepackage{wrapfig}
\usepackage{multirow}
\usepackage{siunitx}
\usepackage{multirow}
\usepackage{array}
\usepackage{amsmath,amsbsy,amsgen,amscd,amssymb,amsthm,amsfonts, mathptmx}
\usepackage{mathtools}
\usepackage{bm}
\usepackage{enumitem}
\usepackage[nameinlink,noabbrev]{cleveref}
\usepackage{url} 
\usepackage{tikz}
\usetikzlibrary{
    arrows.meta,
    shapes.geometric,
    positioning, calc
}
\usepackage{hyphenat}

\newtheorem{assumption}{Assumption}
\newtheorem{lemma}{Lemma}

\newtheorem{corollary}{Corollary}

\newtheorem{theorem}{Theorem}

\newtheorem{definition}{Definition}

\title{Inexact PDHG for Large-Scale Semidefinite Programming}

\title{Inexact PDHG for Large-Scale Semidefinite Programming}
 \author{\name{Karthik Prakhya} 
 \email{karthik.prakhya@umu.se} \\
 \addr{Ume\r{a} University, Sweden} \\[0.5em]
 \name{Zimeng Wang} 
 \email{zimeng.wang@umu.se} \\
 \addr{Ume\r{a} University, Sweden} \\[0.5em]
 }

\crefname{algorithm}{Algorithm}{Algorithms}
\Crefname{algorithm}{Algorithm}{Algorithms}

\usepackage[textsize=tiny]{todonotes}

\DeclareFontFamily{OT1}{pzc}{}
\DeclareFontShape{OT1}{pzc}{m}{it}{ <-> s*[1.2] pzcmi7t }{}
\DeclareMathAlphabet{\mathpzc}{OT1}{pzc}{m}{it}

\DeclareMathOperator*{\minimize}{\mathrm{minimize}}
\DeclareMathOperator*{\argmin}{\mathrm{arg\,min}}
\DeclareMathOperator*{\argmax}{\mathrm{arg\,max}}
\DeclareMathOperator{\subjto}{\mathrm{subject~to}}
\newcommand{\trans}{^{\mathstrut\scriptscriptstyle{\top}}}
\renewcommand{\vec}[1]{\mathbf{#1}}

\newcommand{\linop}[1]{\mathpzc{#1}}
\newcommand{\set}[1]{\mathpzc{#1}}
\newcommand{\reals}{\mathbb{R}}
\newcommand{\ip}[2]{\left\langle#1,#2\right\rangle}
\newcommand{\norm}[1]{\|#1\|}

\DeclareMathOperator{\prox}{\mathrm{prox}}
\DeclareMathOperator{\proj}{\mathrm{proj}}
\DeclareMathOperator{\grad}{\nabla}

\newcommand{\pdhg}{(\ref{alg:pdhg})}
\newcommand{\approxpdhg}{(\ref{alg:approx_pdhg})} %

\usepackage{amsmath,amsfonts,bm}

\def\1{\bm{1}}

\DeclareMathAlphabet{\mathsfit}{\encodingdefault}{\sfdefault}{m}{sl}
\SetMathAlphabet{\mathsfit}{bold}{\encodingdefault}{\sfdefault}{bx}{n}

\usepackage{hyperref}
\usepackage{url}

\begin{document}
\maketitle

\begin{abstract}
Motivated by their strong empirical performance on moderate-sized problems, first-order methods, such as the primal-dual hybrid gradient (PDHG) method, have recently been explored for solving large-scale optimization problems, including semidefinite programming problems. 
In these settings, exact evaluations of prox-operators or projections maybe computationally infeasible, so they are typically approximated. 
Recently, computationally cheap and projection-free conditional gradient methods \citep{jaggi2013revisiting} have become an attractive alternative for this reason.
However, the number of inner iterations for conditional gradient methods remains a computational bottleneck given their usual sublinear convergence rates.
In this work, we exploit advanced variants of conditional gradient method such as Spectrahedron Frank-Wolfe \citep{garber2026randomized} that achieve linear convergence in expectation for trace-constrained semidefinite programming once close enough to the optimal face for the projection subproblem onto the spectrahedron.
In this paper, we present various methods to make the inner projection subproblem onto the spectrahedron tractable including (1) a method of setting error tolerances for inexact projections to avoid unnecessary inner iterations that provably guarantees the convergence of PDHG  and (2) a heuristic spectral warm start method based on the Lanczos-algorithm to initialize Spectrahedron Frank-Wolfe near the optimal face to avoid the initially sublinear convergence of Spectrahedron Frank-Wolfe. 
We show that the overall PDHG algorithm converges in the presence of errors with an $\mathcal{O}(1/t)$ rate and estimate the total number of inner iterations required to converge the optimal solution.
We then do an empirical study that shows that our method of combining spectral warm starts with Spectrahedron Frank-Wolfe to do projection onto the spectrahedron outperforms a method that uses Lanczos to do inexact projection onto the spectrahedron. 
We perform this evaluation on several semidefinite programming (SDP) instances for the MaxCut problem of increasing size in terms of the number of edges and vertices.
\end{abstract} %
\section{Introduction}\label{sec:intro}
In this paper, we are interested in the following saddle-point problem:
\begin{equation}\label{prob:main-saddle-point}
    \tag{P-Saddle-Point}
    \min_{\vec{x} \in \set{X}} \max_{\vec{y} \in \set{Y}} ~\ip{\linop{A}\vec{x}}{\vec{y}} + f(\vec{x}) + g(\vec{x}) - h^\ast(\vec{y}).
\end{equation}
where $\set{X}$ and $\set{Y}$ are two finite-dimensional Hilbert spaces, with dual spaces $\set{X}^\ast$ and $\set{Y}^\ast$, $\linop{A}: \set{X} \to \set{Y}$ is a bounded linear operator with its adjoint denoted by $\linop{A}\trans: \set{Y} \to \set{X}$, $f : \set{X} \to \reals$ is a convex function with Lipschitz continuous gradient with some constant $L$ and both $g : \set{X} \to \reals \cup \{+\infty\}$ and $h :\set{Y} \to \reals \cup \{+\infty\}$ are closed, convex and proper (CCP) functions.

To ensure that strong duality holds, we assume that Slater's condition hold here and that there exists some $\vec{x}_0 \in \mathrm{ri}(\mathrm{dom}(f)) \cap \mathrm{ri}(\mathrm{dom}(g))$ such that $\linop{A}\vec{x}_0 \in \mathrm{ri}(\mathrm{dom}(h))$ (see Proposition 15.24 and Fact 15.25 in \citep{bauschke2020correction}).

Primal-dual first-order methods like Primal Dual Hybrid Gradient (PDHG) (also called Chambolle-Pock algorithm) \citep{chambolle2011first} can solve these types of problems with proximal minimizations to handle potentially non-smooth $g$ and $h$ terms.
We work with the following form of the PDHG algorithm:
\begin{equation}
    \tag{PDHG}
    \begin{aligned}
        \vec{x}_{t+1} & = \prox_{\tau g} \big(\vec{x}_t - \tau \grad f (\vec{x}_t) - \tau \linop{A}^\top \vec{y}_t \big) \\[0.5em]
        \vec{y}_{t+1} & = \prox_{\sigma h^\ast} \big(\vec{y}_t + \sigma \linop{A}(2\vec{x}_{t+1} - \vec{x}_t)\big ). \\[0.5em]
    \end{aligned}
    \label{alg:pdhg}
\end{equation}
However, the proximal minimizations in \eqref{alg:pdhg} sometimes lack closed-form solution and approximate solutions to these subproblems can be computationally intensive to obtain especially for large-scale problems, so we propose an efficient strategies for these proximal minimizations that minimize the total number of inner iterations.
Our results have easy extension to other similar and related primal-dual methods based on proximal minimizations such as Alternating Direction Method of Multipliers (ADMM) \citep{boyd2011distributed} and Three Operator Splitting (TOS) \citep{davis2017three}.

If we take the supremum of the saddle point problem with respect to $\vec{y} \in \set{Y}$, we arrive at the primal formulation, which is the following convex composite minimization problem:
\begin{equation}
\tag{P}
\label{prob:main-primal}
\minimize_{\vec{x} \in \set{X}} \quad f(\vec{x}) + g(\vec{x}) + h(\linop{A}\vec{x}).
\end{equation}
On the other hand, if we take the infimum of the objective of \eqref{prob:main-saddle-point} with respect to $\vec{x} \in \set{X}$, we obtain the dual formulation of \eqref{prob:main-primal}.

Convex composite optimization problems of the form \eqref{prob:main-primal} have many practical applications ranging from operations research \citep{goemans1995improved}, training of machine learning models \citep{peng2007approximating, jacob2009group, chambolle2016ergodic}, robustness characterization of neural networks and other machine learning models \citep{fazlyab2019efficient}, signal processing \citep{rudin1992nonlinear, candes2013phaselift, chambolle2011first, chambolle2016introduction, rasch2020inexact, barbero2018modular, kim2009ell_1, weiss2009efficient, zhang2011unified, zhu2008efficient}, numerical linear algebra \citep{boumal2020deterministic}, min-max games \citep{chambolle2016ergodic} and portfolio optimization \citep{ye2012robust, boyd2017multi}.

Training of machine learning models is enhanced by use of regularizers to promote some desired structure in the learned model such as reduction of overfitting ($\ell_2$ regularization) or sparsity in the learned model ($\ell_1$ regularization) in the learned model and these additional regularizers can be combined with original training objective to produce a convex composite optimization problem \citep{peng2007approximating, jacob2009group, chambolle2016ergodic}.
In addition, many signal and image processing problems such as image denoising, image deconvolution, image inpainting, motion estimation and multi-label
image segmentation lend themselves to the convex composite optimization template \citep{chambolle2011first, chambolle2016introduction, rasch2020inexact}.
As in machine learning optimization problems, additional function terms such as those based on total variation in the convex composite optimization problems in signal processing enforce desired structure on the recovered image or motion field such as reduction of certain types of noise such as salt-pepper noise, reduction of effects due to illumination and shadow or a desired level of smoothness.
For high-dimensional problems in machine learning and signal processing, proximal minimization subproblems can be computationally intensive and/or time-consuming.

We can obtain a special subset of convex minimization problems by setting $g$ and $h$ as indicator functions over the closed convex sets $\set{H} \subseteq \set{X}$ and $\set{K} \subseteq \set{Y}$ in \eqref{prob:main-primal}:
\begin{equation}
\tag{P'}
\label{prob:main-2}
\minimize_{\vec{x} \in \set{X}} \quad f(\vec{x}) \quad \subjto \quad \vec{x} \in \set{H} \quad \text{and} \quad \linop{A}\vec{x} \in \set{K}. 
\end{equation}
For this problem template, of particular interest to this paper are positive semidefinite programs (SDPs) with affine constraints and bounded trace constraints.
These can be represented using \eqref{prob:main-2} by assuming that $\set{H}$ is the intersection of the bounded trace constraint and positive semi-definite cone with the affine constraints encapsulated by the linear operator $\linop{A}$.
Then the proximal minimization corresponding to $g$ reduces to a projection onto the intersection of the bounded trace constraint and positive semi-definite cone, which can be done exactly using a combination of singular value decomposition and simplex projection or approximately using the projection-free conditional gradient method (CGM, a.k.a Frank-Wolfe algorithm) \citep{frank1956algorithm, jaggi2013revisiting}, which is one of the most scalable methods in literature for solving convex optimization problems over a convex and compact set \citep{hazan2012projection, yurtsever2017sketchy, yurtsever2018conditional, yurtsever2019conditional, yurtsever2021scalable, pham2023scalable}. 
The main computational efficiency of CGM comes from the fact that the exact projection is replaced with a computationally less expensive linear minimization oracle (LMO):
\begin{equation*}
    \mbox{lmo}_{\set{X}}(\vec{v}) = \arg\min_{\vec{x} \in \set{H}} \ip{\vec{v}}{\vec{x}}, \quad \vec{v} \in \set{X}
\end{equation*}
that computes the step direction by considering a linearization of the objective near the current iterate.
For bounded trace SDPs, the LMO amounts to solving a cheaper partial eigenvalue decomposition as opposed to a full eigenvalue decomposition.
For large SDPs, the number of LMO calls can still be a computational bottleneck, so there is a need for efficient termination criteria for inner projection subproblems.

Interior point methods (IPMs) \citep{nesterov1994interior, kojima1991primal, toh1999sdpt3, vandenberghe1996semidefinite} can be used to solve such bounded trace SDPs in polynomial time. 
Despite their effectiveness for solving moderate-sized instances, interior-point methods have become intractable to implement for large-scale problems. 
The primary computational bottleneck for IPMs is their reliance on Cholesky factorization to solve a large sparse linear system arising from applying Newton method on the Karush-Kuhn Tucker (KKT) conditions. 
The computational cost of Cholesky factorizations can grow superlinearly with problem size, which is measured by the number of decision variables and constraints. 
For this reason, there is an interest in applying first-order methods such as PDHG that avoid such large matrix factorizations \citep{lin2025pdcs} and can potentially exploit projection-free based methods like CGM to do projections onto the positive semidefinite cone and incorporates several enhancements along with GPU acceleration including adaptive reflected Halpern restarts, adaptive step-size selection, adaptive weight adjustment and diagonal rescaling to work with many different types of conic constraints.
However, \citep{lin2025pdcs} doesn't incorporate inexact projections onto the spectrahedron and instead uses exact projection onto the positive semi-definite cone for general semidefinite programming problems based on full eigendecomposition.

\textbf{Organization}: The rest of the paper is organized as follows. \Cref{sec:related-works} focuses on related works in the development of exact and inexact primal-dual methods with focus on PDHG.
\Cref{sec:preliminaries} goes over convex analysis preliminaries and the notion of inexact proximal point that is essential to defining our inexact PDHG method.
\Cref{sec:inexact-pdhg} defines our inexact PDHG method for the saddle point problem defined in \eqref{prob:main-saddle-point} and establishes convergence guarantees.
\Cref{sec:inexact-pdhg-application-p-prime} focuses on the application of our inexact PDHG method to a special case of \eqref{prob:main-saddle-point} defined by \eqref{prob:main-2} of which trace-bounded positive semidefinite programs are a special case, describes how the Specrahedron Frank-Wolfe algorithm can be used to solve the resulting projection subproblem and provides an upper bound on the total number of inner iterations required by the outer PDHG method for problem described in \eqref{prob:main-2}.
\Cref{sec:experiments} focus on experiments with trace-bounded positive semidefinite programs with Spectrahedron Frank-Wolfe and spectral warm starts for the projection subproblem.

\section{Related Works}\label{sec:related-works}
\subsection{Development of Exact PDHG Methods}
To solve problems of the form \eqref{prob:main-saddle-point}, primal-dual methods of the following form \citep{chambolle2011first, chambolle2016ergodic}:
\begin{equation}
    \tag{General-PDHG}
    \begin{aligned}
        \vec{x}_{t+1} & = \prox_{\tau g} \big(\vec{x}_t - \tau \grad f (\vec{x}_t) - \tau \linop{A}^\top \vec{y}_t \big) \\[0.5em]
        \widetilde{\vec{x}}_{t + 1} &= \vec{x}_{t+1} + \theta (\vec{x}_{t+1} - \vec{x}_t) \\[0.5em] 
        \vec{y}_{t+1} & = \prox_{\sigma h^\ast} \big(\vec{y}_t + \sigma \linop{A}\widetilde{\vec{x}}_{t + 1}  \big ), \\[0.5em]
    \end{aligned}
    \label{alg:gen-pdhg}
\end{equation}
where $\tau, \theta, \sigma>0$ are hyperparameters, 
have been utilized for a long time starting with the seminal work by \citep{arrow1958studies} that put forward the Arrow-Hurwicz method. 
The use of the PDHG method was made popular in the signal processing domain by the work of \citep{zhu2008efficient}, who used this method for TV denoising and deblurring models. 
\citep{chambolle2011first} extended this work to the general saddle point formulation with two function terms (same as \eqref{prob:main-saddle-point} with $f = 0$), which was subsequently extended to the problem template in \eqref{prob:main-saddle-point} in \citep{chambolle2016ergodic}.
One line of works investigate choices for the PDHG hyperparameters that enable convergence \citep{he2012convergence, he2014convergence, esser2009general}. 
\citep{goldstein2015adaptive} focus on adaptive choices for the PDHG hyperparameters.
Many operator-splitting and primal-dual methods including P3DO \citep{yan2018new}, Condat-V{\~u} algorithm \citep{condat2013primal, vu2013splitting} and PDDY \citep{salim2022dualize} can be viewed as variations of PDHG, we refer to this paper for a detailed discussion on these connections \citep{salim2022dualize}.
\citep{upadhyaya2025autolyap, upadhyaya2025automated} provide Lyapunov-type analysis of \eqref{alg:gen-pdhg} with the goal of extending the region of permissible parameter choices for $\tau$ and $\sigma$.
\citep{banert2026chambolle} show that the Chambolle-Pock method coverges under weaker condition of $\tau\sigma\|\linop{A}\|^2 < \frac{4}{1 + 2 \vartheta}$, where $\vartheta > \frac{1}{2}$.
For simplicity, we assume that $\theta = 1$ to get the following formulation of PDHG described in \eqref{alg:pdhg} that we use for our analysis and experimental evaluation.

\subsection{PDHG for Conic Programs}
One line of work uses PDHG for solving general conic programs. 
PDHG is already an algorithm of choice to solve linear programming problems (LPs) \citep{applegate2021practical, applegate2023faster} and more recently SDPs \citep{souto2022exploiting, lin2025pdcs}.
While \citep{lin2025pdcs} solves the projection onto positive-semidefinite cone exactly, \citep{souto2022exploiting} exploits the low-rank matrix approximation to do inexact projection onto the positive-semidefinite cone.
We also use inexact projection onto the same, but unlike \citep{souto2022exploiting}, we plan to use Frank-Wolfe algorithm to do the approximate projection and do not require low-rank assumptions.
PDHG has been used to directly solve the saddle-point formulation of LP \citep{applegate2023faster}. 
There are several efficient large-scale LP solvers developed around PDHG that utilize a variety of heuristics for efficiency such as PDLP for CPUs \citep{applegate2021practical}, cuPDLP \citep{lu2023cupdlp} for GPUs and a C-language implementation of cuPDLP \citep{lu2023cupdlp-c}. 
The two GPU-based implementations of PDLP \citep{lu2023cupdlp, lu2023cupdlp-c} outperform many classical LP solver on many problem instances. 
As a result, PDHG has been integrated into several open-source and commercial LP solvers including COPT \citep{lu2023cupdlp-c}, Xpress \citep{biele2024xpress}, Gurobi \citep{rothberg2024giantLPs}, Google OR tools \citep{applegate2021practical}, HiGHS \citep{lu2023cupdlp-c} and NVIDIA cuOpt \citep{fender2024optimization}.
\citep{jiang2026local} show that PDHG converges linearly for SDPs if strict complementarity and primal-dual nondegeneracy assumptions are satisfied.
\citep{wang2024tuning} present a tuning-free PDHG-based algorithm for large-scale semidefinite programming and they use inexact projection onto the semidefinite cone by taking $\log(n)$ eigenpairs estimated using an iterative solver.
In contrast to \citep{wang2024tuning}, our approach to projection onto the spectrahedron is rank-adaptive in the sense that number of eigenpairs is estimated as part of the spectral warm start procedure and is subsequently adjusted using Spectrahedron Frank-Wolfe, which results in the number of eigenpairs growing or shrinking as PDHG progresses as opposed to being fixed.

\subsection{Inexact Proximal and Primal-Dual Methods}
While there has been extensive earlier work on operator-splitting and primal-dual methods to solve \eqref{prob:main-saddle-point} with exact proximal minimizations, the literature on inexact primal-dual methods of the form \eqref{alg:gen-pdhg} is in its infancy.
The simplest type of error model for proximal minimizations is what \citep{rasch2020inexact} refers to as a type-0 error, which simply bounds the Euclidean distance between the exact proximal point and its approximate counterpart.
\citep{rasch2020inexact} also presents other notions of inexactness of proximal minimizations by assuming that the objective of the proximal operator is minimized up to within some small constant $\epsilon$. 
Given that the proximal operator sometimes lacks closed form solution or can only be approximated such as TV-regularized inverse problems in image processing, there has been a lot of previous work in inexact proximal methods. 
One line of work focuses on convergence of proximal point methods under error assumptions: \citep{cominetti1997coupling, guler1992new, he2012accelerated, lemaire1992convergence, salzo2012inexact} with \citep{lemaire1992convergence, salzo2012inexact} proving convergence rates. 
\citep{combettes2005signal} showed convergence of an inexact forward-backward splitting algorithm. 
For the accelerated versions of these methods including the FISTA method, the earlier works can be divided into ones focused on proofs of convergence \citep{friedlander2012hybrid, luo1993error, patriksson1998cost, zaslavski2010convergence, bello2020inexact}, convergence up to a certain level of accuracy \citep{d2008smooth, devolder2014first, nedic2001convergence} and error-dependent convergence rates \citep{aujol2015stability, schmidt2011convergence, villa2013accelerated}.
\citep{bonettini2016variable} extends the analysis of \citep{villa2013accelerated} to nonconvex problems by utilizing variable metric strategies and linesearch methods.
\citep{nemirovski2004prox} present an optimization algorithm for saddle point problems that converges at the same optimal rate as the exact case as soon as a particular inequality is satisfed for each iteration even with a controlled error term.
\citep{schmidt2011convergence} give convergence rates for inexact proximal gradient descent algorithm.

Previous approaches to inexact primal-dual methods have shown convergence or convergence rates on absolute summable criteria, meaning the error terms are absolutely summable or based on summable relative error criteria, which uses the distance between the iterates to obtain a looser criteria for convergence. 
\citep{condat2013primal} shows that the Condat-V{\~u} algorithm converges given type-0 error for proximal minimizations as well as an error for the gradient evaluation given absolute summable criteria but didn't show any convergence rates.
\citep{rasch2020inexact} study several types of error models mentioned above for proximal minimizations and show that an inexact primal-dual method of the form \eqref{alg:gen-pdhg} with inexact proximal minimizations can converge at the same rate as exact version of the same method.
\citep{tian2016linearized} deals with inexact case by approximating one of the proximal minimizations by a linearized objective function.
\citep{eckstein1992douglas, yuan2005improvement, xie2017inexact, eckstein2017approximate, eckstein2018relative} develop inexact ADMM methods and present convergence guarantees for the same under either absolutely summable error and/or relative error criterion.
\citep{jiang2021approximate, jiang2021first} also look at inexact primal-dual methods but they consider absolute summable error criteria or just solve one of the subproblems inexactly. 
\citep{marques2024general} develop a general framework for inexact operator splitting algorithms with relative error criterion and apply it to both the PDHG and TOS algorithms.
\citep{wu2022general} also explore relative summable error criterion to solve saddle point problems of the form \ref{prob:main-saddle-point} 
with $f = 0$ and using Type-3 error assumption for proximal mapping (see \citep{rasch2020inexact} for definition) with focus on applications to image restoration problems.
Our work also uses summable relative error criteria in the sense that error tolerances are set using distance between consecutive iterates but unlike previous work on inexact primal-dual methods, we focus on Type-2 error assumption for proximal evaluation (see \citep{rasch2020inexact} for definition) that is closely related to LMO error for conditional gradient methods and focus on use of summable relative error criteria to make it feasible to solve large-scale semidefinite programming problems using conditional gradient methods and inexact PDHG.

\section{Preliminaries}\label{sec:preliminaries}

\noindent\textbf{Notation and terminology:}\quad For notational simplicity, we work on the $\mathbb{R}^p / \mathbb{R}^n$ spaces with the Euclidean norms. 
Throughout the paper, $\|\cdot\|$ represents the Euclidean norm for vectors and the Frobenius norm for matrices. 
For a convex function $f$ and a point $\mathbf{x} \in \set{X}$, we use $\nabla f(\mathbf{x})$ for its gradient at point $x$ and $\mathbf{g} \in \partial f(\mathbf{x})$ for its subgradient, and $f^\ast(\vec{y})$ for its Fenchel conjugate. 
We denote the set of symmetric matrices of dimension $n$ by $\set{S}^n$, the positive semidefinite cone of dimension $n$ by $\set{S}_+^n$ and the set of binary vectors of dimension $n$ as $\mathbb{Z}_2^n$. 

\subsection{Convex Analysis Preliminaries}
In this section, we first discuss the preliminaries of convex analysis that are useful throughout the paper and then present the error model that we use for the PDHG algorithm, which corresponds to a Type-2 proximal error \citep{rasch2020inexact}. 
We start with the basic definitions of convexity, strong convexity, smoothness and Fenchel conjugate  here.
\begin{definition}[Convex Function, \citep{bauschke2020correction}] A function $f: \set{X} \rightarrow [-\infty , +\infty]$ is convex if and only if it satisfies the following inequality:
\begin{equation*}
    f(\alpha \vec{x} + (1 - \alpha)\vec{y}) \leq \alpha f(\vec{x}) + (1 - \alpha) f(\vec{y}),\quad \forall\vec{x}, \vec{y} \in \operatorname{dom}(f), \alpha \in [0, 1].
\end{equation*}
\end{definition}
\begin{definition}[Strong convexity] \label{def:strong-convexity}
A proper, closed, and convex function $\phi \colon \set{Y} \to \mathbb{R}$ is said to be $\beta$-strongly convex with respect to a norm $\|\cdot\|$ if, for all $\vec{x}, \vec{y} \in \set{Y}$ and any subgradient $\bm{g} \in \partial \phi(\vec{x})$,
\begin{equation*}
    \phi(\vec{y}) \geq \phi(\vec{x}) + \langle \bm{g}, \vec{y} - \vec{x} \rangle + \frac{\beta}{2}\|\vec{y} - \vec{x}\|^2.
\end{equation*}
\end{definition}

\begin{definition}[Smoothness] \label{def:smoothness}
A differentiable, proper, closed, and convex function $\psi: \set{Y} \to \mathbb{R}$ is said to be $\gamma$-smooth with respect to a norm $\|\cdot\|$ if, for all $\vec{x}, \vec{y} \in \set{Y}$,
\begin{equation*}
    \psi(\vec{y}) \leq \psi(\vec{x}) + \langle \nabla \psi(\vec{x}), \vec{y} - \vec{x} \rangle + \frac{\gamma}{2} \|\vec{y} - \vec{x}\|^2.
\end{equation*}
\end{definition}

\begin{definition}[Fenchel Conjugate \citep{bauschke2020correction}]
Let $\phi \colon \set{Y} \to [-\infty, +\infty]$ be a proper function. Then, its Fenchel conjugate $\phi^* \colon \set{Y}^\ast \to [-\infty, +\infty]$ is defined as
\begin{equation*}
    \phi^*(\vec{y}) := \sup_{\bm{x} \in \set{Y}} \left\{ \langle \vec{y}, \bm{x} \rangle - \phi(\bm{x}) \right\}.
\end{equation*}
\end{definition}
We now define the the concept of proximal mapping and then discuss how to relax it to incorporate an approximation error. 
\begin{definition}[Proximal Mapping] \label{def:prox-mapping}
For a proper, convex and lower semicontinuous function $g: \set{X} \rightarrow (-\infty, +\infty]$, $\tau > 0$ and a point $\vec{y} \in \set{X}$, the proximal map $\prox_{\tau g}$ is defined by 
\begin{equation*}
    \prox_{\tau g}(\vec{y}) := \smash{\arg\min}_{\vec{x} \in \set{X}} \left\{g(\vec{x}) + \frac{1}{2\tau}\norm{\vec{x} - \vec{y}}^2 \right \}
\end{equation*}
\end{definition}
\noindent Given the definition of a proximal mapping, we can use the corresponding first-order optimality condition to present an equivalent characterization of proximal mapping:
\begin{lemma}\label{lem:prox-lemma}
For any point $\vec{y}, \vec{z} \in \set{X}$, we have
\begin{equation*}
    \vec{z} = \prox_{\tau g}(\vec{y}) \iff \frac{\vec{y} - \vec{z}}{\tau} \in \partial g(\vec{z}).
\end{equation*}
\end{lemma}
\begin{proof}
The proximal map defined by \Cref{def:prox-mapping} has to satisfy the following first-order optimality condition: $\bm{0} \in \partial g(\vec{z}) + \frac{\vec{z} - \vec{y}}{\tau}$, which can be rearranged to get the desired result.
\end{proof}
\noindent This characterization can then be relaxed by considering an $\epsilon$-subdifferential or inexact subdifferential of $g$ at $\vec{z} \in \set{X}$:
\begin{definition}[$\epsilon$-subdifferential \citep{rasch2020inexact}] \label{def:inexact-subdiff}
Suppose that $g: \set{X} \rightarrow [-\infty, +\infty]$ is a proper function.
The $\epsilon$-\textbf{subdifferential} $\partial_\epsilon g(\vec{z})$ at point $\vec{z} \in \mathrm{dom}(g)$ and for some $\epsilon > 0$ is given by 
\begin{equation*}
    \partial_\epsilon g(\vec{z}) = \{\vec{g} \in \set{X}^\star~:~ g(\vec{x}) \geq g(\vec{z}) + \ip{\vec{g}}{\vec{x} - \vec{ 
       z}} - \epsilon,~ \forall~\mathbf{x} \in \mathrm{dom}(g)\}.
\end{equation*}
\end{definition}
\noindent Based on this, it is possible to characterize $\epsilon$-optimality with respect to $g$ as follows:
\begin{equation*}
    \vec{0} \in \partial_\epsilon g(\vec{z}) \iff g(\vec{z}) \leq \inf g + \epsilon.
\end{equation*}
Given this, we can define the inexact proximal point as follows:
\begin{definition}[Inexact Proximal Point \citep{rasch2020inexact}] \label{def:inexact-proximal-point} Given some $\epsilon > 0$ and some proper, convex and lower semi-continuous function $g$, we say that $\vec{z} \in \set{X}$ is an \textbf{inexact proximal point} with precision $\epsilon$ denoted by $\vec{z} := \smash{\prox_{\tau g}^{(\epsilon)}}(\vec{y})$ if
\begin{equation*}
     \frac{\vec{y} - \vec{z}}{\tau} \in \partial_\epsilon g(\vec{z}).
\end{equation*}
\end{definition}
Finally, we require the following three-point descent lemma from convex analysis: 
\begin{lemma}[Three-points descent lemma]
\label{lem:three-points-descent}
Suppose $f:\set{X} \to \reals$ is $L$-smooth and convex. Then, for any $\vec{x}$, $\vec{y}$, $\vec{z}\in \set{X}$, the following bound holds:
\begin{equation*}
    f(\vec{x}) \leq f(\vec{y}) + \ip{\grad f(\vec{z})}{\vec{x} - \vec{y}} + \frac{L}{2}\norm{\vec{x}-\vec{z}}^2
\end{equation*}
\end{lemma}
\begin{proof}
By $L$-smoothness of $f$, for any $\vec{u} \in \set{X}$ we have
\begin{equation*}
    f(\vec{u}) \le f(\vec{z}) + \langle \nabla f(\vec{z}), \vec{u} - \vec{z} \rangle + \frac{L}{2}\|\vec{u} - \vec{z}\|^2 .
\end{equation*}
Applying this with $\vec{u} = \vec{x}$ yields
\begin{equation*}
    f(\vec{x}) \leq f(\vec{z}) + \langle \nabla f(\vec{z}), \vec{x} - \vec{z} \rangle + \frac{L}{2}\|\vec{x} - \vec{z}\|^2 .
\end{equation*}
On the other hand, convexity of $f$ implies that for any $\vec{u}, \vec{v}$, the following holds:
\begin{equation*}
    f(\vec{u}) \geq f(\vec{v}) + \langle \nabla f(\vec{v}), \vec{u} - \vec{v} \rangle .
\end{equation*}
Applying this with $\vec{u} = \vec{y}$ and $\vec{v} = \vec{z}$ gives
\begin{equation*}
    f(\vec{y}) \geq f(\vec{z}) + \langle \nabla f(\vec{z}), \vec{y} - \vec{z} \rangle
    \quad\Longrightarrow\quad
    f(\vec{z}) \le f(\vec{y}) - \langle \nabla f(\vec{z}), \vec{y} - \vec{z} \rangle .
\end{equation*}
Substituting this inequality for $f(\vec{z})$ into the bound on $f(\vec{x})$, we obtain
\begin{align*}
    f(\vec{x})
    &\leq f(\vec{y}) - \langle \nabla f(\vec{z}), \vec{y} - \vec{z} \rangle
        + \langle \nabla f(\vec{z}), \vec{x} - \vec{z} \rangle
        + \frac{L}{2}\|\vec{x} - \vec{z}\|^2 \\[0.5em]
    &= f(\vec{y}) + \langle \nabla f(\vec{z}), \vec{x} - \vec{y} \rangle
        + \frac{L}{2}\|\vec{x} - \vec{z}\|^2 .
\end{align*}
This completes the proof.
\end{proof}
\subsection{Convergence in Normed Spaces}
Here, we will review some results from functional analysis on the subject of convergence in normed spaces, which will prove useful in the convergence analysis of \pdhg and its inexact variant.
More specifically, we try to find an equivalent norm such that quantities of interest can be expressed more compactly and the convergence analysis can be done more easily as a result. 
So we start by making precise the notion of an equivalent norm.
This definition is based on the idea that equivalent norms on a vector space $\set{V}$ induce the same topology on $\set{V}$. 
Now we present an important result that states that all norms defined on finite dimensional real vector spaces are equivalent:
\begin{lemma}[Equivalent Norms, Theorem 2.4-5, \citep{kreyszig1978functional}] \label{lem:equiv-norms-finite-spaces} For any finite-dimensional real vector space $\set{V}$ and any two norms $\norm{\cdot}_a$ and $\norm{\cdot}_b$ defined on $\set{V}$, $\norm{\cdot}_a$ is equivalent to $\norm{\cdot}_b$, i.e., $c_0\norm{\vec{x}}_b \leq \norm{\vec{x}}_a  \leq c_1\norm{\vec{x}}_b$ for some $0<c_0\le c_1$. 
\end{lemma}
Now we present a useful result that basically states that equivalent norms have similar convergence properties:
\begin{lemma}\label{lem:convergence-equiv-norms-general} Let $\set{V}$ be a real vector space and suppose that $\norm{\cdot}_a$ and $\norm{\cdot}_b$ are defined on $\set{V}$ and that they are equivalent, then for any sequence $\{\vec{u}_t\}_{t \geq 0} \subset \set{V}$ and $\vec{u}^\star \in \set{V}$
\begin{equation*}
    \norm{\vec{u}_t - \vec{u}^\star}_a \rightarrow 0 \iff \norm{\vec{u}_t - \vec{u}^\star}_b \rightarrow 0.
\end{equation*}
\end{lemma}
\begin{proof} Given that $\norm{\cdot}_a$ and $\norm{\cdot}_b$ are equivalent, they must satisfy the following relation:
\begin{equation*}
    c_0\norm{\vec{x}}_b \leq \norm{\vec{x}}_a  \leq c_1\norm{\vec{x}}_b, \quad \forall \vec{x} \in \set{V}.
\end{equation*}
Now suppose that $\norm{\vec{u}_t - \vec{u}^\star}_a \rightarrow 0$. Let $\epsilon > 0$. Then there exists an $N \in \set{N}$ such that for all $t \geq N$, we have that $\norm{\vec{u}_t - \vec{u}^\star}_a \leq c_0 \epsilon$, which means that
\begin{equation*}
    \norm{\vec{u}_t - \vec{u}^\star}_b \leq \frac{1}{c_0} \norm{\vec{u}_t - \vec{u}^\star}_a \leq \frac{1}{c_0} (c_0 \epsilon) = \epsilon
\end{equation*}
for all $t \geq N$.
Therefore, $\norm{\vec{u}_t - \vec{u}^\star}_b \rightarrow 0$.
Using the relation $\|\mathbf{u}_t - \mathbf{u}^\star\|_a \leq c_1\|\mathbf{u}_t - \mathbf{u}^\star\|_b$, we can similarly show that $\norm{\vec{u}_t - \vec{u}^\star}_b \rightarrow 0$ implies that $\norm{\vec{u}_t - \vec{u}^\star}_a \rightarrow 0$.
\end{proof}

\subsection{Saddle-Point Problem}\label{sec:saddle-point-prelim}
We can define the Lagrangian for \eqref{prob:main-saddle-point} as  $\mathcal{L}(\vec{x}, \vec{y}) := \ip{\linop{A}\vec{x}}{\vec{y}} + f(\vec{x}) + g(\vec{x}) - h^\ast(\vec{y})$.
For any saddle point $(\vec{x}^\star, \vec{y}^\star)\in \set{X} \times \set{Y}$, we must have that 
\begin{equation}\label{eqn:saddle-point-characterization}
    \mathcal{L}(\vec{x}^\star, \vec{y}) \leq \mathcal{L}(\vec{x}^\star, \vec{y}^\star) \leq \mathcal{L}(\vec{x}, \vec{y}^\star) 
\end{equation}
for any $(\vec{x}, \vec{y}) \in \set{X} \times \set{Y}$.
Therefore, $\sup_{\substack{\vec{x} \in \set{X} \\
\vec{y} \in \set{Y}}} \left \{\mathcal{L}(\vec{x}_t, \vec{y}) - \mathcal{L}(\vec{x}, \vec{y}_t) \right \} = 0$ if and only if $(\vec{x}_t, \vec{y}_t)$ is itself a saddle point.
Here, we assume that the supremum is well-defined and possibly taking the value $+\infty$.
Given a saddle point $(\mathbf{x}^\star, \mathbf{y}^\star)$, we can define the primal-dual gap of \pdhg at iteration $t$ as $\mathcal{G}_{x^\star, y^\star}(\mathbf{x}_t, \mathbf{y}_t) := \mathcal{L}(\vec{x}_t, \vec{y}^\star) - \mathcal{L}(\vec{x}^\star, \vec{y}_t)$.
Assume that maximum of $\mathcal{L}(\vec{x}_{t+1}, \vec{y})$ is attained at $\tilde{\vec{y}}_{t+1}$.
An upper bound for the primal objective residual can be derived as follows:
\begin{equation*}
\begin{aligned}
    &f(\vec{x}_{t+1}) + g(\vec{x}_{t+1}) + h(\linop{A}\vec{x}_{t+1}) = \max_{\vec{y} \in \set{Y}} \mathcal{L}(\vec{x}_{t+1}, \vec{y}) = \mathcal{L}(\vec{x}_{t+1}, \tilde{\vec{y}}_{t+1})  \mbox{ and }\\
    &f(\vec{x}^\star) + g(\vec{x}^\star) + h(\linop{A}\vec{x}^\star) = \max_{\vec{y} \in \set{Y}} \mathcal{L}(\vec{x}^\star, \vec{y}) \geq \mathcal{L}(\vec{x}^\star, \vec{y}_{t+1}),
\end{aligned}
\end{equation*}
which in turn implies that 
\begin{equation*}%
    f(\vec{x}_{t+1}) + g(\vec{x}_{t+1}) + h(\linop{A}\vec{x}_{t+1}) - [f(\vec{x}^\star) + g(\vec{x}^\star) + h(\linop{A}\vec{x}^\star)] \leq \mathcal{L}(\vec{x}_{t+1}, \tilde{\vec{y}}_{t+1}) - \mathcal{L}(\vec{x}^\star, \vec{y}_{t+1}).
\end{equation*}
Thus, it is sufficient to first bound $\mathcal{L}(\vec{x}_{t+1}, \vec{y}) - \mathcal{L}(\vec{x}, \vec{y}_{t+1})$.
Given this bound and following similar derivation as in \citep{chambolle2011first, chambolle2016ergodic} for the exact case and \citep{rasch2020inexact} for the inexact case, we can show ergodic convergence of the average primal-dual gap with respect to a saddle point $(\mathbf{x}^\star, \mathbf{y}^\star)$.
We will first present these results adapted to our inexact version of \pdhg. %
\section{Inexact PDHG} \label{sec:inexact-pdhg}
First, we define an inexact version of the \pdhg algorithm. 
The general iteration of the approximate PDHG for our model problem can be written as follows:
\begin{equation}
\tag{Inexact PDHG}
\begin{aligned}
    \vec{x}_{t+1}^\prime & = \smash{\prox^{(\epsilon_t)}_{\tau g}} \big(\vec{x}_t^\prime - \tau \grad f (\vec{x}_t^\prime) - \tau \linop{A}^\top \vec{y}_t^\prime \big) \\[0.5em]
    \vec{y}_{t+1}^\prime & = \smash{\prox^{(\delta_t)}_{\sigma h^\ast}} \big(\vec{y}_t^\prime + \sigma\linop{A} (2\vec{x}_{t+1}^\prime - \vec{x}_t^\prime) \big), \\[0.5em]
\end{aligned}
\label{alg:approx_pdhg}
\end{equation}
where we assume that the $\vec{x}'_{t+1}\in \set{X}$ is an inexact proximal point with precision of $\epsilon_t>0$ and $\vec{y}'_{t+1}\in\set{Y}$ is an inexact proximal point with precision of $\delta_t$. 
The error tolerance scheme is to discussed in this paper is to terminate with $\vec{x}_{t + 1}^\prime$ and $\vec{y}_{t + 1}^\prime$ when the resulting sum $\epsilon_t + \delta_t$ of the type-2 proximal errors for the first and second proximal subproblem satisfy
\begin{equation}\label{eqn:inexact-error-tolerance-schedule}
\bar{\epsilon}_t + \bar{\delta}_t := \frac{1}{t + 1}\sum_{i = 0}^t (\epsilon_i + \delta_i) = \frac{ \rho_1\kappa }{t + 1} \sum_{i = 0}^t \frac{\norm{\vec{x}_{i+1}^\prime - \vec{x}_i^\prime}^2}{2\tau} + \frac{\rho_2 \beta}{(\sigma\norm{\linop{A}} + \beta)(t + 1)} \sum_{i = 0}^t\frac{\norm{\vec{y}_{i+1}^\prime - \vec{y}_i^\prime}^2}{2\sigma}
\end{equation}
where $0 \leq \rho_1, \rho_2 < 1$, $\kappa := 1 - \tau L - \tau\beta \norm{\linop{A}} - \tau\sigma\norm{\linop{A}}^2$ for some $\beta > 0$ such that $\kappa > 0$.
Note that we set $0 \leq \rho_1, \rho_2 < 1$ so that it is easier to derive an upper bound for certain terms (see Corollary 1).
The motivation behind this choice of error tolerance scheme is to make the error tolerances $\epsilon_t$ and $\delta_t$ at iteration $t$ dependent on the fixed-point residual or the norm difference between consecutive iterates $(\vec{x}_t, \vec{y}_t)$ and $(\vec{x}_{t+1}, \vec{y}_{t+1})$. 
The first step is to show that the average primal-dual gap with respect to a saddle point $(\mathbf{x}^\star, \mathbf{y}^\star)$ for \approxpdhg~ converges at the rate of $\mathcal{O}(1/t)$ given the error tolerance schedule described in \eqref{eqn:inexact-error-tolerance-schedule}.
In the analysis that follows, unless stated otherwise, we make the following assumption. 
\begin{assumption}\label{ass:cvx}
    $f : \set{X} \to \reals$ is a convex, closed and proper  (CCP) function with $L$- Lipschitz continuous gradient and both $g : \set{X} \to \reals \cup \{+\infty\}$ and $h :\set{Y} \to \reals \cup \{+\infty\}$ are also CCP functions that are not necessarily differentiable.
\end{assumption}
\noindent We then use the above lemma to upper bound the quantity $\mathcal{L}(\vec{x}_{t + 1}, \vec{y}) - \mathcal{L}(\vec{x}, \vec{y}_{t + 1})$:
\begin{lemma}\label{lem:primal-dual-gap-bound-lemma-1} Let Assumption \ref{ass:cvx} hold.
Consider \approxpdhg~  applied to \eqref{prob:main-saddle-point}.
Then the following relation holds for any iteration $t$ of \approxpdhg:
\begin{multline*}
    \mathcal{L}(\vec{x}_{t+1}^\prime, \vec{y}) - \mathcal{L}(\vec{x}, \vec{y}_{t+1}^\prime)\leq \frac{1}{2\tau}\norm{\vec{x} - \vec{x}_t^\prime}^2 +  \frac{1}{2\sigma}\norm{\vec{y} -\vec{y}_t^\prime}^2 - \frac{1}{2\tau}\norm{\vec{x} - \vec{x}_{t + 1}^\prime}^2 - \frac{1 - \tau L}{2\tau}\norm{\vec{x}_t^\prime - \vec{x}_{t + 1}^\prime}^2 \\ - \frac{1}{2\sigma}\norm{\vec{y} - \vec{y}_{t + 1}^\prime}^2 - \frac{1}{2\sigma}\norm{\vec{y}_t^\prime - \vec{y}_{t + 1}^\prime}^2
    + \ip{\vec{y}_t^\prime - \vec{y}_{t+1}^\prime}{\linop{A}(\vec{x} - \vec{x}_{t + 1}^\prime)} + \ip{\linop{A}(\vec{x}_t^\prime - \vec{x}_{t+1}^\prime)}{ \vec{y} - \vec{y}_{t+1}^\prime} + \epsilon_t + \delta_t,~\forall \vec{x},~\vec{y} \in \set{X} \times \set{Y},
\end{multline*}
\end{lemma}
\begin{proof}
The proof mimics \citep[Lemma 3]{rasch2020inexact}. First, recall the \approxpdhg~ updates:
\begin{align}
    \vec{x}_{t+1}^\prime & = \smash{\prox^{(\epsilon_t)}_{\tau g}} \big(\vec{x}_t^\prime - \tau \grad f (\vec{x}_t^\prime) - \tau \linop{A}^\top \vec{y}_t^\prime \big) \label{eqn:approx-pdhg-update-1}\\[0.5em]
    \vec{y}_{t+1}^\prime & = \smash{\prox^{(\delta_t)}_{\sigma h^\ast}} \big(\vec{y}_t^\prime + \sigma\linop{A} (2\vec{x}_{t+1}^\prime - \vec{x}_t^\prime) \big),  \label{eqn:approx-pdhg-update-2}
\end{align}
Using \Cref{def:inexact-proximal-point}, we then say that \eqref{eqn:approx-pdhg-update-2} is equivalent to the following: 
\begin{equation} \label{eqn:first-order-optimality-cond-pdhg-update-2}
    \frac{\vec{y}_t^\prime + \sigma \linop{A} (2\vec{x}_{t+1}^\prime - \vec{x}_t^\prime) - \vec{y}_{t+1}^\prime}{\sigma} \in \partial_{\delta_t} h^\ast(\vec{y}_{t+1}^\prime).
\end{equation}
Using this and \Cref{def:inexact-subdiff}, we can say that the following holds for any $\vec{y} \in \set{Y}$:
\begin{align}
    h^\ast(\vec{y}_{t+1}^\prime) &\leq h^\ast(\vec{y}) + \ip{ \frac{\vec{y}_t^\prime + \sigma\linop{A} (2\vec{x}_{t+1}^\prime - \vec{x}_t^\prime) - \vec{y}_{t+1}^\prime}{\sigma}}{\vec{y}_{t+1}^\prime - \vec{y}} + \delta_t \nonumber \\
    &= h^\ast(\vec{y}) + \frac{1}{\sigma}\ip{\vec{y}_t^\prime - \vec{y}_{t+1}^\prime}{\vec{y}_{t+1}^\prime - \vec{y}} + \ip{\linop{A} (2\vec{x}_{t+1}^\prime - \vec{x}_t^\prime)}{\vec{y}_{t+1}^\prime - \vec{y}} + \delta_t \nonumber \\
    &\leq h^\ast(\vec{y}) + \frac{1}{2\sigma}\norm{\vec{y} - \vec{y}_t^\prime}^2 - \frac{1}{2\sigma}\norm{\vec{y} - \vec{y}_{t+1}^\prime}^2 - \frac{1}{2\sigma}\norm{\vec{y}_t^\prime - \vec{y}_{t+1}^\prime}^2 + \ip{\linop{A} (2\vec{x}_{t+1}^\prime - \vec{x}_t^\prime)}{\vec{y}_{t+1}^\prime - \vec{y}} + \delta_t \label{eqn:primal-dual-gap-bound-proof-1}
\end{align}
Next, using \Cref{def:inexact-proximal-point}, we can say that \eqref{eqn:approx-pdhg-update-1} is equivalent to the 
\begin{equation}\label{eqn:first-order-optimality-cond-pdhg-update-1}
    \frac{\vec{x}_t^\prime - \tau \nabla f(\vec{x}_t^\prime) - \tau \linop{A}^\top \vec{y}_t^\prime - \vec{x}_{t+1}^\prime}{\tau} \in \partial_{\epsilon_t} g(\vec{x}_{t+1}^\prime).
\end{equation}
Using this and \Cref{def:inexact-subdiff}, we can say that the following holds for any $\vec{x} \in \set{X}$:
\begin{align}
    g(\vec{x}_{t+1}^\prime) &\leq g(\vec{x}) + \ip{ \frac{\vec{x}_t^\prime - \tau \nabla f(\vec{x}_t^\prime) - \tau \linop{A}^\top \vec{y}_t^\prime - \vec{x}_{t+1}^\prime}{\tau}}{\vec{x}_{t+1}^\prime - 
    \vec{x}} + \epsilon_t \notag\\
    &\leq g(\vec{x}) + \frac{1}{2\tau}\norm{\vec{x}_t^\prime - \vec{x}}^2 - \frac{1}{2\tau}\norm{\vec{x}_t^\prime - 
    \vec{x}_{t+1}^\prime}^2 - \frac{1}{2\tau}\norm{\vec{x}_{t+1}^\prime - \vec{x}}^2 - \ip{ \nabla f(\vec{x}_t^\prime) }{ \vec{x}_{t+1}^\prime - 
    \vec{x}} + \ip{ \vec{y}_t^\prime}{\linop{A}(\vec{x} - \vec{x}_{t+1}^\prime)} + \epsilon_t \label{eqn:primal-dual-gap-bound-proof-2}
\end{align}
Finally, using \Cref{lem:three-points-descent}, we can say that
\begin{equation} \label{eqn:primal-dual-gap-bound-proof-3}
    f(\vec{x}_{t+1}^\prime) \leq f(\vec{x}) + \ip{ \nabla f(\vec{x}_t^\prime)}{ \vec{x}_{t+1}^\prime - \vec{x}} + \frac{L}{2}\norm{\vec{x}_{t+1}^\prime - \vec{x}_t^\prime}^2
\end{equation}
Finally, we add \eqref{eqn:primal-dual-gap-bound-proof-1}, \eqref{eqn:primal-dual-gap-bound-proof-2} and \eqref{eqn:primal-dual-gap-bound-proof-3}.
We then proceed to add and subtract the following terms on the right-hand side: $\ip{\linop{A}\vec{x}} {\vec{y}_{t+1}^\prime}$, $\ip{\linop{A}\vec{x}_{t+1}^\prime} {\vec{y}_{t+1}^\prime}$ and $\ip{\linop{A}\vec{x}_{t+1}^\prime} {\vec{y}}$.
Rearranging the resulting inequality gives the desired result.
\end{proof}
\noindent Next, given the upper bound on $\mathcal{L}(\vec{x}_{t+1}^\prime, \vec{y}) - \mathcal{L}(\vec{x}, \vec{y}_{t+1}^\prime)$ provided by \Cref{lem:primal-dual-gap-bound-lemma-1}, we can proceed to bound the average of this quantity. 
This is equivalent to Eq.23 and Theorem 1 in \citep{rasch2020inexact} for \approxpdhg.

\begin{theorem}\label{thm:ergodic-primal-residual}
Let Assumption \ref{ass:cvx} hold and assume that $(\mathbf{x}^\star, \mathbf{y}^\star)$ be a saddle point.
Consider \approxpdhg~ applied to \eqref{prob:main-saddle-point}.
Then the following relation holds for any iteration $t$ of \approxpdhg:
\begin{multline}\label{eq: thm1}
    \sum_{i = 0}^t \mathcal{L}(\vec{x}_{i+1}^\prime, \vec{y}^\star) - \mathcal{L}(\vec{x}^\star, \vec{y}_{i+1}^\prime)
    \leq \frac{1}{2\tau}\norm{\vec{x}^\star - \vec{x}_0^\prime}^2 +  \frac{1}{2\sigma}\norm{\vec{y}^\star -\vec{y}_0^\prime}^2 + \ip{\linop{A}(\vec{x}_0^\prime - \vec{x}^\star)}{\vec{y}^\star - \vec{y}_0^\prime}
     \\- \frac{\beta}{\sigma\norm{\linop{A}}
     + \beta}\sum_{i = 0}^t\frac{\norm{\vec{y}_{i+1}^\prime - \vec{y}_i^\prime}^2}{2\sigma} - \kappa \sum_{i = 0}^t\frac{\norm{\vec{x}_{i+1}^\prime - \vec{x}_{i}^\prime}^2}{2\tau} + \sum_{i = 0}^t (\epsilon_i + \delta_i).
\end{multline}
where $\kappa = 1 - \tau L  - \tau \beta \norm{\linop{A}} - \tau\sigma \norm{\linop{A}}^2$ for some $\beta > 0$ such that $\kappa>0$.
Assuming that error tolerance schedule $\{\epsilon_t\}_{t \geq 0}$ and $\{\delta_t\}_{t \geq 0}$ in \eqref{eqn:inexact-error-tolerance-schedule} is used, it follows that
\begin{equation*}
    \frac{1}{t+1}\sum_{i = 0}^t \mathcal{G}_{\mathbf{x}^\star, \mathbf{y}^\star}(\mathbf{x}_{i + 1}^\prime, \mathbf{y}_{i + 1}^\prime) \leq \frac{\Phi_0}{t + 1},
\end{equation*}
where 
\begin{equation*}
    \Phi_0 := \frac{1}{2\tau}\norm{\vec{x}^\star - \vec{x}_0^\prime}^2 +  \frac{1}{2\sigma}\norm{\vec{y}^\star -\vec{y}_0^\prime}^2 +
    \ip{\linop{A}(\vec{x}_0^\prime - \vec{x}^\star)}{\vec{y}^\star - \vec{y}_0^\prime}.
\end{equation*}
\end{theorem}
\begin{proof}
Given \Cref{lem:primal-dual-gap-bound-lemma-1}, the following holds for iteration $i$ of \approxpdhg:
\begin{multline}\label{eqn:primal-dual-gap-bound-2-proof-1}
    \mathcal{L}(\vec{x}_{i+1}^\prime, \vec{y}) - \mathcal{L}(\vec{x}, \vec{y}_{i+1}^\prime)\leq \frac{1}{2\tau}\norm{\vec{x} - \vec{x}_i^\prime}^2 +  \frac{1}{2\sigma}\norm{\vec{y} -\vec{y}_i^\prime}^2 - \frac{1}{2\tau}\norm{\vec{x} - \vec{x}_{i+1}^\prime}^2 - \frac{1 - \tau L}{2\tau}\norm{\vec{x}_i^\prime - \vec{x}_{i+1}^\prime}^2 \\ - \frac{1}{2\sigma}\norm{\vec{y} - \vec{y}_{i+1}^\prime}^2 - \frac{1}{2\sigma}\norm{\vec{y}_i^\prime - \vec{y}_{i+1}^\prime}^2
    + \ip{\vec{y}_i^\prime - \vec{y}_{i+1}^\prime}{\linop{A}(\vec{x} - \vec{x}_{i+1}^\prime)} + \ip{\linop{A}(\vec{x}_i^\prime - \vec{x}_{i+1}^\prime)}{ \vec{y} - \vec{y}_{i+1}^\prime} + \epsilon_i + \delta_i, \quad \vec{x},~\vec{y} \in \set{X} \times \set{Y},
\end{multline}
We now obtain a telescoping sum for the second inner product as follows by adding and subtracting the terms: $\ip{\linop{A}(\vec{x}_i^\prime - \vec{x})}{\vec{y} - \vec{y}_i^\prime}$ and $\ip{\linop{A}(\vec{x}_{i+1}^\prime - \vec{x})}{\vec{y} - \vec{y}_{i+1}^\prime}$ and then combining terms:
\begin{align*}
    &\langle \vec{y}_i^\prime - \vec{y}_{i+1}^\prime, \linop{A}(\vec{x} - \vec{x}_{i+1}^\prime) \rangle + \ip{\linop{A}(\vec{x}_i^\prime - \vec{x}_{i+1}^\prime)}{\vec{y} - 
    \vec{y}_{i+1}^\prime} \\
    &\quad = \ip{\linop{A}(\vec{x}_i^\prime - \vec{x})}{\vec{y} - \vec{y}_i^\prime} - \ip{\linop{A}(\vec{x}_{i+1}^\prime - \vec{x})
    }{\vec{y} - \vec{y}_{i+1}^\prime} + \ip{\linop{A}(\vec{x}_i^\prime - \vec{x}_{i+1}^\prime)}{\vec{y}_i^\prime - 
    \vec{y}_{i+1}^\prime} \\
    &\quad \leq \ip{\linop{A}(\vec{x}_i - \vec{x})}{\vec{y} - \vec{y}_i^\prime} - \ip{\linop{A}(\vec{x}_{i+1}^\prime - \vec{x})
    }{\vec{y} - \vec{y}_{i+1}^\prime} + (\tau\sigma \norm{\linop{A}}^2 + \tau \beta \norm{\linop{A}}) \frac{\norm{\vec{x}_i^\prime - \vec{x}_{i+1}^\prime}^2}{2\tau} + \frac{\sigma \norm{\linop{A}}}{\sigma \norm{\linop{A}} + \beta }\frac{\norm{\vec{y}_{i+1}^\prime - \vec{y}_i^\prime}^2}{2\sigma},
\end{align*}
where we used Cauchy-Schwartz inequality and Young's inequality in the last step to say that the following holds for any $\vec{x}, \vec{x}^\prime \in \set{X}$ and $\vec{y}, \vec{y}^\prime \in \set{Y}$:
\begin{equation}\label{eqn:fenchel-young-ineq-ip}
    \ip{\linop{A}(\vec{x} - \vec{x}^\prime)}{\vec{y} - \vec{y}^\prime} \leq \norm{\linop{A}}\norm{\vec{x} - \vec{x}^\prime}\norm{\vec{y} - \vec{y}^\prime} \leq \frac{\norm{\linop{A}}\gamma\tau}{2\tau}\norm{\vec{x} - \vec{x}^\prime}^2 + \frac{\norm{\linop{A}}\sigma}{2\gamma\sigma}\norm{\vec{y} - \vec{y}^\prime}^2
\end{equation}
where we chose $\gamma =  \sigma\norm{\linop{A}} + \beta$ for some $\beta > 0$.
Substituting this into \eqref{eqn:primal-dual-gap-bound-2-proof-1}, we get
\begin{multline} \label{eqn:primal-dual-gap-bound-2-proof-2}
    \mathcal{L}(\vec{x}_{i+1}^\prime, \vec{y}) - \mathcal{L}(\vec{x}, \vec{y}_{i+1}^\prime) \leq \frac{1}{2\tau}\norm{\vec{x} - \vec{x}_i^\prime}^2 +  \frac{1}{2\sigma}\norm{\vec{y} -\vec{y}_i^\prime}^2 - \frac{1}{2\tau}\norm{\vec{x} - \vec{x}_{i+1}^\prime}^2 - \frac{1 - \tau L}{2\tau}\norm{\vec{x}_i^\prime - \vec{x}_{i+1}^\prime}^2  \\ - \frac{1}{2\sigma}\norm{\vec{y} - \vec{y}_{i+1}^\prime}^2 +
    \ip{\linop{A}(\vec{x}_i^\prime - \vec{x})}{\vec{y} - \vec{y}_i^\prime} - \ip{\linop{A}(\vec{x}_{i+1}^\prime - \vec{x})
    }{\vec{y} - \vec{y}_{i+1}^\prime} 
    + (\tau\sigma \norm{\linop{A}}^2 + \tau \beta \norm{\linop{A}}) \frac{\norm{\vec{x}_i^\prime - \vec{x}_{i+1}^\prime}^2}{2\tau} \\
    - \frac{\beta}{\sigma\norm{\linop{A}} + \beta}\frac{\norm{\vec{y}_{i+1}^\prime - \vec{y}_i^\prime}^2}{2\sigma} + \epsilon_i + \delta_i,\quad \vec{x},~\vec{y} \in \set{X} \times \set{Y},
\end{multline}
Now we sum \eqref{eqn:primal-dual-gap-bound-2-proof-2} from $ i = 0, \ldots, t$ to get the following:
\begin{multline} \label{eqn:primal-dual-gap-bound-2-proof-3}
    \sum_{i = 0}^t \mathcal{L}(\vec{x}_{i+1}^\prime, \vec{y}) - \mathcal{L}(\vec{x}, \vec{y}_{i+1}^\prime) \leq \frac{1}{2\tau}\norm{\vec{x} - \vec{x}_0^\prime}^2 +  \frac{1}{2\sigma}\norm{\vec{y} -\vec{y}_0^\prime}^2 - \frac{1}{2\tau}\norm{\vec{x} - \vec{x}_{t + 1}^\prime}^2  - \frac{1}{2\sigma}\norm{\vec{y} - \vec{y}_{t + 1}^\prime}^2  \\   - \kappa \sum_{i = 0}^t\frac{\norm{\vec{x}_i^\prime - \vec{x}_{i+1}^\prime}^2}{2\tau} +
    \ip{\linop{A}(\vec{x}_0^\prime - \vec{x})}{\vec{y} - \vec{y}_0^\prime} - \ip{\linop{A}(\vec{x}_{t + 1}^\prime - \vec{x})
    }{\vec{y} - \vec{y}_{t + 1}^\prime} \\
    - \frac{\beta}{\sigma\norm{\linop{A}} + \beta}\sum_{i = 0}^t\frac{\norm{\vec{y}_{i+1}^\prime - \vec{y}_i^\prime}^2}{2\sigma} + \sum_{i = 0}^t (\epsilon_i + \delta_i),\quad \vec{x},~\vec{y} \in \set{X} \times \set{Y},
\end{multline}
where $\kappa = 1 - \tau L  - \tau \beta \norm{\linop{A}} - \tau\sigma \norm{\linop{A}}^2$.
Therefore, a condition for convergence is that $\kappa > 0$ or $\tau L  + \tau \beta \norm{\linop{A}} + \tau\sigma \norm{\linop{A}}^2 < 1$.
Assuming that $\linop{A} \neq 0$, we then proceed to apply Fenchel-Young inequality \eqref{eqn:fenchel-young-ineq-ip} again to the term $\ip{\linop{A}(\vec{x}_{t + 1}^\prime - \vec{x})}{\vec{y} - \vec{y}_{t + 1}^\prime}$ with $\gamma = \frac{1 - \tau L}{\tau \norm{\linop{A}}}$ as follows:
\begin{equation*}
    \ip{\linop{A}(\vec{x}_{t + 1}^\prime - \vec{x})}{\vec{y} - \vec{y}_{t + 1}^\prime} \leq \frac{1 - \tau L}{2\tau}\norm{\vec{x}_{t + 1}^\prime - \vec{x}}^2 + \frac{1}{2\sigma} \left[ \frac{\tau\sigma \norm{\linop{A}}^2}{1 - \tau L} \right ]\norm{\vec{y} - \vec{y}_{t+1}^\prime}^2
\end{equation*}
Note that $\frac{\tau\sigma \norm{\linop{A}}^2}{1 - \tau L} 
< 1$ due to the relationship that $\kappa > 0$. 
For $\linop{A} = 0$, the above inequality is vacuously true.
Substituting this and rearranging proves the following:
\begin{multline*}
    \sum_{i = 0}^t \Big( \mathcal{L}(\vec{x}_{i+1}^\prime, \vec{y}) - \mathcal{L}(\vec{x}, \vec{y}_{i+1}^\prime) \Big)  + \frac{L}{2}\norm{\vec{x}_{t + 1}^\prime - \vec{x}}^2 + \frac{1}{2\sigma} \left[ 1 - \frac{\tau\sigma \norm{\linop{A}}^2}{1 - \tau L} \right ]\norm{\vec{y} - \vec{y}_{t+1}^\prime}^2 \\
    \leq \frac{1}{2\tau}\norm{\vec{x} - \vec{x}_0^\prime}^2 +  \frac{1}{2\sigma}\norm{\vec{y} -\vec{y}_0^\prime}^2  - \kappa \sum_{i = 0}^t\frac{\norm{\vec{x}_i^\prime - \vec{x}_{i+1}^\prime}^2}{2\tau} +
    \ip{\linop{A}(\vec{x}_0^\prime - \vec{x})}{\vec{y} - \vec{y}_0^\prime}
     \\- \frac{\beta}{\sigma\norm{\linop{A}}
     + \beta}\sum_{i = 0}^t\frac{\norm{\vec{y}_{i+1}^\prime - \vec{y}_i^\prime}^2}{2\sigma} + \sum_{i = 0}^t (\epsilon_i + \delta_i),\quad \forall\vec{x},~\vec{y} \in \set{X} \times \set{Y},
\end{multline*}
The final result can be obtained by setting $(\vec{x}, \vec{y}) = (\vec{x}^\star, \vec{y}^\star)$ and dropping all the non-negative terms on left-hand side except for $\sum_{i = 0}^t \mathcal{L}(\vec{x}_{i+1}^\prime, \vec{y}^\star) - \mathcal{L}(\vec{x}^\star, \vec{y}_{i+1}^\prime)$.
\end{proof}
\noindent From \Cref{thm:ergodic-primal-residual}, it is clear that the average primal-dual gap converges to $0$ in an ergodic sense with rate $\mathcal{O}(1/t)$ if error tolerance schedule $\{\epsilon_t\}_{t \geq 0}$ and $\{\delta_t\}_{t \geq 0}$ in \eqref{eqn:inexact-error-tolerance-schedule} is used.
The following corollary of \Cref{thm:ergodic-primal-residual} helps to show that the average distance between the iterates of \approxpdhg~ goes to 0 given the choice of error tolerance schedule described in \eqref{eqn:inexact-error-tolerance-schedule}.
\begin{corollary}\label{corr:iterate-to-iterate-bounds} Given the choice of error tolerance schedule \eqref{eqn:inexact-error-tolerance-schedule}, the following holds for any iteration $t$ of \approxpdhg:
\begin{align*} 
    &\frac{1}{t + 1}\sum_{i = 0}^t\frac{\norm{\vec{x}_i^\prime - \vec{x}_{i+1}^\prime}^2}{2\tau} \leq \frac{\Phi_0}{(1 - \rho_1)\kappa(t + 1)}, \\
    &\frac{1}{t + 1}\sum_{i = 0}^t\frac{\norm{\vec{y}_i^\prime - \vec{y}_{i+1}^\prime}^2}{2\sigma} \leq \frac{(\sigma\norm{\linop{A}} + \beta)\Phi_0}{(1 - \rho_2)\beta (t + 1)}, \mbox{ and }\\
    &\bar{\epsilon}_t + \bar{\delta}_t \leq \frac{\rho_1\Phi_0}{(1 - \rho_1)(t + 1)} + \frac{\rho_2\Phi_0}{(1 - \rho_2)(t + 1)}.
\end{align*}
where $\kappa := 1 - \tau L - \tau\beta \norm{\linop{A}} - \tau\sigma\norm{\linop{A}}^2$ for some $\beta > 0$ such that $\kappa > 0$, $\rho_1$ and $\rho_2$ are the chosen scaling parameters for the error tolerance schedule in \eqref{eqn:inexact-error-tolerance-schedule} and 
\begin{equation*}
    \Phi_0 := \frac{1}{2\tau}\norm{\vec{x}^\star - \vec{x}_0^\prime}^2 +  \frac{1}{2\sigma}\norm{\vec{y}^\star -\vec{y}_0^\prime}^2 +
    \ip{\linop{A}(\vec{x}_0^\prime - \vec{x}^\star)}{\vec{y}^\star - \vec{y}_0^\prime}.
\end{equation*}
\end{corollary}
\begin{proof} This follows directly from \Cref{thm:ergodic-primal-residual}, the fact that $\sum_{i = 0}^t \mathcal{L}(\vec{x}_{i+1}^\prime, \vec{y}^\star) - \mathcal{L}(\vec{x}^\star, \vec{y}_{i+1}^\prime) \geq 0$ and \eqref{eqn:inexact-error-tolerance-schedule}. 
\end{proof}
Using \Cref{corr:iterate-to-iterate-bounds}, we can then make the following stronger statement about the convergence of the iterate-to-iterate distance:

\begin{corollary} \label{corr:iterate-to-iterate-bounds-inf-sum} Given the choice of error tolerance schedule \eqref{eqn:inexact-error-tolerance-schedule}, the following holds:
\begin{align*}
     &\sum_{t = 0}^\infty \frac{\norm{\vec{x}_t^\prime - \vec{x}_{t+1}^\prime}^2}{2\tau} \leq \frac{\Phi_0}{(1 - \rho_1)\kappa}, &\quad \lim_{t \rightarrow \infty} \norm{\vec{x}_t^\prime - \vec{x}_{t+1}^\prime}^2 = 0 \\
     &\sum_{t = 0}^\infty \frac{\norm{\vec{y}_t^\prime - \vec{y}_{t+1}^\prime}^2}{2\sigma} \leq \frac{(\sigma\norm{\linop{A}} + \beta)\Phi_0}{(1 - \rho_2)\beta}, &\quad \lim_{t \rightarrow \infty} \norm{\vec{y}_t^\prime - \vec{y}_{t+1}^\prime}^2 = 0\\
     &\sum_{t = 0}^\infty (\epsilon_t + \delta_t) \leq \frac{\rho_1\Phi_0}{1 - \rho_1} + \frac{\rho_2\Phi_0}{1 - \rho_2}, &\mbox{ and}\quad \lim_{t \rightarrow \infty} \epsilon_t + \delta_t = 0.
\end{align*}
\end{corollary}
\begin{proof}
\Cref{corr:iterate-to-iterate-bounds} implies that $\left \{\sum_{i = 0}^t\frac{\norm{\vec{x}_i^\prime - \vec{x}_{i+1}^\prime}^2}{2\tau}\right \}_{t \geq 0}$, $\left \{\sum_{i = 0}^t\frac{\norm{\vec{y}_i^\prime - \vec{y}_{i+1}^\prime}^2}{2\sigma}\right \}_{t \geq 0}$ are bounded monotonically increasing sequences and the desired results follow from the monotone convergence theorem for series.
\end{proof}
For the remainder of the proofs, we introduce the following notation that allows us to concisely describe the convergence of the iterates of \pdhg and \approxpdhg. 
Define $\|\vec{z}\|_Z := \sqrt{\ip{\vec{z}}{\vec{z}}_Z}$ for any $\vec{z} = (\vec{x}, \vec{y}) \in \set{X} \times \set{Y}$, where $\ip{\tilde{\vec{z}}}{\hat{\vec{z}}}_Z = \frac{1}{\tau}\ip{\tilde{\vec{x}}}{\hat{\vec{x}}} + \frac{1}{\sigma}\ip{\tilde{\vec{y}}}{\hat{\vec{y}}}$ and $\hat{\vec{z}} + \tilde{\vec{z}} := (\hat{\vec{x}} + \tilde{\vec{x}}, \hat{\vec{y}} + \tilde{\vec{y}})$ for any $\hat{\vec{z}} = (\hat{\vec{x}}, \hat{\vec{y}})$ and $\tilde{\vec{z}} = (\tilde{\vec{x}}, \tilde{\vec{y}})$. 
Similar to Eq. 22 and 23 in \citep{condat2013primal}, we can rewrite the first-order optimality conditions for the iteration in \approxpdhg~ compactly as 
\begin{equation}\label{eqn:first-order-optimal-cond-pdhg-compact}
    -\begin{bmatrix}
        \nabla f(\mathbf{x}_t^\prime) \\
        \mathbf{0}
        \end{bmatrix} \in \begin{bmatrix}\partial_{\epsilon_t} g(\mathbf{x}_{t+1}^\prime) + \linop{A}^\top \mathbf{y}_{t+1}^\prime \\
        -\linop{A}\mathbf{x}_{t+1}^\prime +  \partial_{\delta_t} h^\ast(\mathbf{y}_{t+1}^\prime) \end{bmatrix} + \begin{bmatrix} \frac{1}{\tau} I & -\linop{A}^\top \\
            -\linop{A} & \frac{1}{\sigma} I
            \end{bmatrix}\begin{bmatrix} \mathbf{x}_{t+1}^\prime - \mathbf{x}_t^\prime \\
            \mathbf{y}_{t+1}^\prime - \mathbf{y}_t^\prime
            \end{bmatrix}
\end{equation}
Define the operators $\linop{L}$, $\linop{C}$ and $\linop{D}$ as follows:
\begin{equation*}
    \linop{L}: \begin{bmatrix}
    \vec{x} \\
    \vec{y}
    \end{bmatrix} \rightarrow \begin{bmatrix} \frac{1}{\tau} I & -\linop{A}^\top \\
    -\linop{A} & \frac{1}{\sigma} I
    \end{bmatrix}\begin{bmatrix}
    \vec{x} \\
    \vec{y}
    \end{bmatrix},\quad \linop{C}_{\epsilon_t, \delta_t}: \begin{bmatrix}
    \vec{x} \\
    \vec{y}
    \end{bmatrix} \rightarrow \begin{bmatrix} \partial_{\epsilon_t} g(\mathbf{x}) + \linop{A}^\top \mathbf{y} \\
    -\linop{A}\mathbf{x} + \partial_{\delta_t} h^\ast (\mathbf{y}) 
    \end{bmatrix} \quad \mbox{and} \quad \linop{D}: \begin{bmatrix}
    \vec{x} \\
    \vec{y}
    \end{bmatrix} \rightarrow \begin{bmatrix}
        \nabla f(\mathbf{x}) \\
        \mathbf{0}
        \end{bmatrix}
\end{equation*}
We analogously define $\mathcal{C}$ operator to be the error-free version of $\mathcal{C}_{\epsilon_t, \delta_t}$ (i.e. $\epsilon_t = 0$ and $\delta_t = 0$ for all $t$).
Then \eqref{eqn:first-order-optimal-cond-pdhg-compact} can be expressed more compactly as
\begin{equation*}
    -\linop{D}(\mathbf{z}_t^\prime) \in \linop{C}_{\epsilon_t, \delta_t}(\mathbf{z}_{t+1}^\prime) + \linop{L}(\mathbf{z}_{t+1}^\prime - \mathbf{z}_t^\prime).
\end{equation*}
Rearranging the above and solving for $\mathbf{z}_{t+1}^\prime$, we get the forward-backward splitting operator $\linop{T}_{\epsilon_t, \delta_t}$ for \approxpdhg~ as
\begin{equation*}
    \linop{T}_{\epsilon_t, \delta_t}(\mathbf{z}_{t + 1}) := (I + \linop{L}^{-1} \circ \linop{C}_{\epsilon_t, \delta_t})^{-1}(I - \linop{L}^{-1} \circ \linop{D})(\mathbf{z}_n).
\end{equation*}
We can analogously define the operator $\linop{T}$ for the exact version of \pdhg.
Given this notation, the \pdhg~  update and \approxpdhg~ update at iteration $t$ is given by $\vec{z}_{t + 1} = \linop{T}(\vec{z}_t)$ and $\vec{z}_{t+1}^\prime = \linop{T}_{\epsilon_t, \delta_t}(\vec{z}_t^\prime)$, respectively.

Next, we note that the following holds for any $\mathbf{z} = (\mathbf{x}, \mathbf{y})$
\begin{align*}
    \langle \mathbf{z}, \linop{L} \mathbf{z} \rangle_L &= \frac{1}{\tau}\|\mathbf{x}\|^2 + \frac{1}{\sigma}\|\mathbf{y}\|^2 -2 \langle \linop{A}\mathbf{x}, \mathbf{y} \rangle \\
    &\geq \frac{1}{\tau}\|\mathbf{x}\|^2 + \frac{1}{\sigma}\|\mathbf{y}\|^2 -2 \|\linop{A}\|\|\mathbf{x}\|\| \mathbf{y}\| \\
    &\geq c(\|\mathbf{x}\|^2 + \|\mathbf{y}\|^2)
\end{align*}
where $c = \frac{1}{2}\left(\frac{1}{\tau} + \frac{1}{\sigma} - \sqrt{\left (\frac{1}{\tau} - \frac{1}{\sigma}\right )^2 - 4 \|\linop{A}\|^2}\right)$. 
The condition $c > 0$ holds whenever $1 - \tau\sigma \|\linop{A}\|^2 > 0$, which in turn holds whenever $\kappa > 0$.
Therefore, $\linop{L}$ is positive definite and we can define the corresponding norm $\norm{\cdot}_L$ on $\set{X} \times \set{Y}$ with respect to the linear operator $\linop{L}$ by $\norm{z}_L := \sqrt{\ip{\vec{z}}{\linop{L}\vec{z}}}$ for any $\vec{z} \in \set{X} \times \set{Y}$.
We can show that $\norm{\cdot}_Z$ and $\norm{\cdot}_L$ are equivalent norms, so that convergence w.r.t. $\norm{\cdot}_L$ is equivalent to convergence w.r.t. $\norm{\cdot}_Z$.
\begin{lemma} \label{lem:z-norm-l-norm-rel} For any sequence $\{\vec{u}_t\}_{t \geq 0} \subset \set{X} \times \set{Y}$ and $\vec{u}^\star \in \set{X} \times \set{Y}$, the following holds:
\begin{equation*}
    \lim_{t \rightarrow \infty } \norm{\vec{u}_t - \vec{u}^\star}_Z = 0 \iff \lim_{t \rightarrow \infty } \norm{\vec{u}_t - \vec{u}^\star}_L = 0
\end{equation*}
\end{lemma}
\begin{proof} This follows directly from \Cref{lem:convergence-equiv-norms-general} since $\norm{\cdot}_Z$ and $\norm{\cdot}_L$ are equivalent norms defined on the finite dimensional vector space $\set{X} \times \set{Y}$ as per \Cref{lem:equiv-norms-finite-spaces}. 
\end{proof}
Thus, to prove convergence w.r.t. $\norm{\cdot}_Z$ norm, it is sufficient to show convergence w.r.t. $\norm{\cdot}_L$ norm. 
Here, using the $\norm{\cdot}_L$ norm, we can write \Cref{lem:primal-dual-gap-bound-lemma-1} in a more compact form that is also more convenient for proving convergence of \approxpdhg~ to a saddle point w.r.t. $\norm{\cdot}_Z$ using an approach similar to \citep{wu2022general}:
\begin{theorem} \label{lem:iterate-convergence-decrease-cond}
Consider \approxpdhg~ applied to \eqref{prob:main-saddle-point} and assume that the error tolerance schedule in \eqref{eqn:inexact-error-tolerance-schedule} holds.
Assume that a saddle point exists for \eqref{prob:main-saddle-point}.
Let $\vec{z}_t^\prime := (\vec{x}_t^\prime, \vec{y}_t^\prime)$ represent the iterates at iteration $t$ of \approxpdhg.
Then there exists a saddle point $\vec{z}^\infty := (\mathbf{x}^\infty, \mathbf{y}^\infty)$ of \eqref{prob:main-saddle-point} such that $\lim_{t \rightarrow \infty} \norm{\vec{z}_t^\prime - \vec{z}^\infty}_Z = 0$ holds.
\end{theorem}
\begin{proof}
Let $\vec{z}^\star := (\vec{x}^\star, \vec{y}^\star)$ be a saddle point.
Starting with \Cref{lem:primal-dual-gap-bound-lemma-1} with $\vec{x} = \vec{x}^\star$ and $\vec{y} = \vec{y}^\star$, we apply the Law of Cosines to get 
\begin{multline*}
    \mathcal{L}(\vec{x}_{t+1}^\prime, \vec{y}^\star) - \mathcal{L}(\vec{x}^\star, \vec{y}_{t+1}^\prime)\leq   \frac{1}{\tau}\ip{\vec{x}_t^\prime - \vec{x}_{t+1}^\prime}{\vec{x}_{t+1}^\prime - \vec{x}^\star} + \frac{1} {\sigma}\ip{\vec{y}_t^\prime - \vec{y}_{t+1}^\prime}{\vec{y}_{t+1}^\prime - \vec{y}^\star} + \frac{L}{2}\norm{\vec{x}_t^\prime - \vec{x}_{t + 1}^\prime}^2 \\ 
    + \ip{\vec{y}_t^\prime - \vec{y}_{t+1}^\prime}{\linop{A}(\vec{x}^\star - \vec{x}_{t + 1}^\prime)} + \ip{\linop{A}(\vec{x}_t^\prime - \vec{x}_{t+1}^\prime)}{ \vec{y}^\star - \vec{y}_{t+1}^\prime} + \epsilon_t + \delta_t.   
\end{multline*}
This can then be written more compactly in terms of the linear operator $\linop{L}$ as follows:
\begin{equation*}
    \mathcal{L}(\vec{x}_{t+1}^\prime, \vec{y}^\star) - \mathcal{L}(\vec{x}^\star, \vec{y}_{t+1}^\prime) \leq \ip{\vec{z}_t^\prime - \vec{z}_{t+1}^\prime}{\vec{z}_{t+1}^\prime - \vec{z}^\star}_L + \frac{L}{2}\norm{\vec{x}_t^\prime - \vec{x}_{t + 1}^\prime}^2 + \epsilon_t + \delta_t. 
\end{equation*}
We then apply Law of Cosines to the first term to get
\begin{equation}\label{eqn:convergence-to-limit-L-norm-rel}
    \mathcal{L}(\vec{x}_{t+1}^\prime, \vec{y}^\star) - \mathcal{L}(\vec{x}^\star, \vec{y}_{t+1}^\prime) \leq -\frac{1}{2}\norm{\vec{z}_t^\prime - \vec{z}_{t+1}^\prime}_L^2 - \frac{1}{2}\norm{\vec{z}_{t+1}^\prime - \vec{z}^\star}_L^2 + \frac{1}{2}\norm{\vec{z}_t^\prime - \vec{z}^\star}_L^2 + \frac{L}{2}\norm{\vec{x}_t^\prime - \vec{x}_{t + 1}^\prime}^2 + \epsilon_t + \delta_t. 
\end{equation}
Using the fact that $(\vec{x}^\star, \vec{y}^\star)$ is a saddle point implies that $\mathcal{L}(\vec{x}_{t+1}^\prime, \vec{y}^\star) - \mathcal{L}(\vec{x}^\star, \vec{y}_{t+1}^\prime) \geq 0$ and rearranging, we get the following:
\begin{equation*}
    \norm{\vec{z}_{t+1}^\prime - \vec{z}^\star}_L^2 \leq \norm{\vec{z}_t^\prime - \vec{z}^\star}_L^2 + 2\epsilon_t + 2\delta_t + L\norm{\vec{x}_t^\prime - \vec{x}_{t + 1}^\prime}^2 - \norm{\vec{z}_t^\prime -\vec{z}_{t+1}^\prime}_L^2.
\end{equation*}
For any $s \leq t$, we have 
\begin{equation}\label{eqn:convergence-to-limit-recursive-rel}
    \norm{\vec{z}_{t+1}^\prime - \vec{z}^\star}_L^2 \leq \norm{\vec{z}_s^\prime - \vec{z}^\star}_L^2 + 2\sum_{i = s}^t (\epsilon_i + \delta_i) + \sum_{i = s}^t L\norm{\vec{x}_i^\prime - \vec{x}_{i + 1}^\prime}^2,
\end{equation}
where we dropped the last negative term.
Expanding this recursive relationship fully, we get 
\begin{equation*}
    \norm{\vec{z}_{t+1}^\prime - \vec{z}^\star}_L^2 \leq \norm{\vec{z}_0^\prime - \vec{z}^\star}_L^2 + 2\sum_{i = 0}^t (\epsilon_i + \delta_i)  + L \sum_{i = 0}^t \norm{\vec{x}_i^\prime - \vec{x}_{i + 1}^\prime}^2.
\end{equation*}
Given the above and the boundedness of the last two terms above from \Cref{corr:iterate-to-iterate-bounds}, we can find some $C > 0$ such that $\norm{\vec{z}_{t+1}^\prime - \vec{z}^\star}_L \leq C$, which means the sequence $\{\vec{z}_t^\prime\}_{t \geq 0}$ is bounded. 
Given that $\{\vec{z}_t^\prime\}_{t \geq 0}$ is bounded and $\|\vec{z}_t^\prime - \vec{z}_{t+1}^\prime\|_Z \rightarrow 0$ given \Cref{corr:iterate-to-iterate-bounds-inf-sum}, the sequence $\{\vec{z}_t^\prime\}_{t \geq 0}$ has a subsequential limit $\vec{z}^\infty := (\vec{x}^\infty, \vec{y}^\infty) \in \set{X} \times \set{Y}$ and a subsequence $\{\vec{z}_{t_k}^\prime\}_{k \geq 0}$ such that $\vec{z}_{t_k}^\prime \rightarrow \vec{z}^\infty$.
Given that both $g$ and $h$ are proper, closed and convex functions (and hence lower semicontinuous) and $\ip{\linop{A}\vec{x}}{\vec{y}}$ is continuous function of $\vec{x}$ and $\vec{y}$, it is clear that $\mathcal{L}(\vec{x}, \vec{y})$ is lower semicontinuous w.r.t. to $\mathbf{x}$ and fixed $\mathbf{y}$ so that 
\begin{equation*}
    \mathcal{L}(\vec{x}^\infty, \vec{y}) \leq \liminf_{k \rightarrow \infty} \mathcal{L}(\mathbf{x}_{t_{k+1}}, \mathbf{y})
\end{equation*}
and upper semicontinuous with respect to $\mathbf{y}$ and fixed $\mathbf{x}$ so that 
\begin{equation*}
    -\mathcal{L}(\mathbf{x}, \mathbf{y}^\infty) \leq \liminf_{k \rightarrow \infty} -\mathcal{L}(\mathbf{x}, \mathbf{y}_{t_{k+1}}).
\end{equation*}
Putting it together, it means that
\begin{align*}
     \mathcal{L}(\vec{x}^\infty, \vec{y}) - \mathcal{L}(\vec{x}, \vec{y}^\infty) &\leq \liminf_{k \rightarrow \infty} \{\mathcal{L}(\vec{x}_{t_{k+1}}^\prime, \vec{y}) - \mathcal{L}(\vec{x}, \vec{y}_{t_{k+1}}^\prime) \} \\
     &\leq \liminf_{k \rightarrow \infty} \left \{ -\frac{1}{2}\norm{\vec{z}_{t_k}^\prime - \vec{z}_{t_k+1}^\prime}_L^2 - \frac{1}{2}\norm{\vec{z}_{t_k+1}^\prime - \vec{z}}_L^2 + \frac{1}{2}\norm{\vec{z}_{t_k}^\prime - \vec{z}}_L^2 + \frac{L}{2}\norm{\vec{x}_{t_k}^\prime - \vec{x}_{t_k + 1}^\prime}^2 + \epsilon_{t_k} + \delta_{t_k} \right \} \\
     &= 0,
\end{align*}
for any $\vec{z} := (\vec{x}, \vec{y}) \in \set{X} \times \set{Y}$, where we used \eqref{eqn:convergence-to-limit-L-norm-rel} with $\vec{z}^\star$ replaced with $\vec{z}$ and \Cref{corr:iterate-to-iterate-bounds-inf-sum} to say that $\epsilon_{t_k} + \delta_{t_k} \rightarrow 0$.
Taking the supremum on the left-hand side w.r.t. to $(\vec{x}, \vec{y}) \in \set{X} \times \set{Y}$, we get
\begin{equation*}
    \sup_{\substack{\vec{x} \in \set{X} \\
\vec{y} \in \set{Y}}} \mathcal{L}(\vec{x}^\infty, \vec{y}) - \mathcal{L}(\vec{x}, \vec{y}^\infty) = 0,
\end{equation*}
which means that $\vec{z}^\infty := (\vec{x}^\infty, \vec{y}^\infty)$ is a saddle point (see \Cref{sec:saddle-point-prelim}).
Now it remains to show that $\{\vec{z}_t^\prime\}_{t \geq 0}$ converges to $\vec{z}^\infty$.
Given any $\varepsilon > 0$ and \Cref{corr:iterate-to-iterate-bounds-inf-sum}, we can find some $k_0 > 0$ such that 
\begin{equation*}
    \norm{\vec{z}_{k_0}^\prime - \vec{z}^\infty}_L^2 \leq \frac{\varepsilon^2}{2} \mbox{ and } 2\sum_{i = k_0}^t (\epsilon_i + \delta_i) + \sum_{i = k_0}^t L\norm{\vec{x}_i^\prime - \vec{x}_{i + 1}^\prime}^2 \leq \frac{\varepsilon^2}{2}
\end{equation*}
holds for all $t \geq k_0$.
Given this and \eqref{eqn:convergence-to-limit-L-norm-rel} with $\vec{z}^\star$ replaced by $\vec{z}^\infty$, we have that 
\begin{equation*}
    \norm{\vec{z}_{t+1}^\prime - \vec{z}^\infty}_L^2 \leq \norm{\vec{z}_{k_0}^\prime - \vec{z}^\infty}_L^2 + 2\sum_{i = k_0}^t (\epsilon_i + \delta_i) + \sum_{i = k_0}^t L\norm{\vec{x}_i^\prime - \vec{x}_{i + 1}^\prime}^2 \leq \frac{\varepsilon^2}{2} + \frac{\varepsilon^2}{2} = \varepsilon^2.
\end{equation*}
for all $t \geq k_0$.
Thus, we have $\norm{\vec{z}_t^\prime - \vec{z}^\infty}_L \rightarrow 0$ and then $\norm{\vec{z}_t^\prime - \vec{z}^\infty}_Z \rightarrow 0$ by \Cref{lem:z-norm-l-norm-rel}.
This proves the desired result.
\end{proof}
We now turn our attention to the fixed-point characterization of \pdhg~ and \approxpdhg.
To do this, we first derive the following relationship between exact and inexact proximal points: 
\begin{lemma}\label{lem:inexact-exact-prox-rel-lem}
Suppose $g: \mathbb{R}^n \rightarrow \mathbb{R} \cup \{+\infty\}$ be a closed, convex and proper function.
If $\vec{u}^\prime = \prox_{\tau g}^{(\epsilon)}(\vec{x}^\prime)$ and $\vec{u} = \prox_{\tau g}(\vec{x})$, then 
\begin{equation*}
    \norm{\vec{u} - \vec{u}^\prime}^2 \leq \tau \epsilon + \langle \vec{x}^\prime - \vec{x}, \vec{u}^\prime - \vec{u} \rangle. 
\end{equation*}
\end{lemma}
\begin{proof}
Given both \Cref{def:inexact-proximal-point}, we have that 
\begin{align*}
    &\frac{\vec{x} - \vec{u}}{\tau} \in \partial g(\vec{u}) \\
    &\frac{\vec{x}^\prime - \vec{u}^\prime}{\tau} \in \partial_\epsilon g(\vec{u}^\prime).
\end{align*}
Using \Cref{def:inexact-subdiff}, we have that
\begin{align*}
    &g(\vec{u}) \geq g(\vec{u}^\prime) + \ip{\frac{\vec{x}^\prime - \vec{u}^\prime}{\tau}}{\vec{u} - \vec{u}^\prime} - \epsilon \\
    &g(\vec{u}^\prime) \geq g(\vec{u}) + \ip{\frac{\vec{x} - \vec{u}}{\tau}}{\vec{u}^\prime - \vec{u}}.
\end{align*}
Adding these together, we get
\begin{equation*}
    0 \geq \frac{1}{\tau} \langle \vec{x}^\prime - \vec{u}^\prime + \vec{u} - \vec{x}, \vec{u} - \vec{u}^\prime \rangle - \epsilon = \frac{1}{\tau}\norm{\vec{u} - \vec{u}^\prime}^2 +  \frac{1}{\tau} \langle \vec{x}^\prime - \vec{x}, \vec{u} - \vec{u}^\prime \rangle - \epsilon
\end{equation*}
Multiplying both sides by $\tau$ and rearranging this proves the desired result.
\end{proof}
Applying \Cref{lem:inexact-exact-prox-rel-lem} to \pdhg and \approxpdhg, we obtain the following bound on the difference between their update operators.

\begin{corollary}\label{corr:inexact-pdhg-exact-pdhg-operator-rel}
For iteration $t$ of $\approxpdhg$, the following holds:
\begin{equation*}
    \norm{\linop{T}_{\epsilon_t,\delta_t}(\vec{z}_t)-\linop{T}(\vec{z}_t)}_Z^2 \leq \epsilon_t+
    \left(
    \sqrt{\delta_t}
    +2\sqrt{\tau\sigma}\norm{\linop{A}}
    \sqrt{\epsilon_t}
    \right)^2,
\end{equation*}
where $\vec{z}_t := (\vec{x}_t,\vec{y}_t)\in\set{X}\times\set{Y}$.
The same bound holds with \(\vec{z}_t\) replaced by \(\vec{z}_t^\prime\).
The above result holds for every selection of the operator $\linop{T}_{\epsilon_t,\delta_t}$.
\end{corollary}

\begin{proof}
Fix an iteration $t$. 
For the sake of brevity, define $\vec{u}_t := \vec{x}_t-\tau\nabla f(\vec{x}_t) -\tau\linop{A}^\top\vec{y}_t$.
Then $\vec{x}_{t+1} :=\prox_{\tau g}(\vec{u}_t)$ and $\vec{x}_{t+1}^\prime :=\prox_{\tau g}^{(\epsilon_t)}(\vec{u}_t)$.
By \Cref{lem:inexact-exact-prox-rel-lem}, we get 
\begin{equation*}
    \norm{\vec{x}_{t+1}^\prime-\vec{x}_{t+1}}^2
    \leq \tau\epsilon_t.
\end{equation*}
Next, define $\vec{v}_t :=\vec{y}_t+\sigma\linop{A}(2\vec{x}_{t+1}-\vec{x}_t)$ and $\vec{v}_t^\prime :=\vec{y}_t+\sigma\linop{A}(2\vec{x}_{t+1}^\prime-\vec{x}_t)$.
The dual updates are then given by $\vec{y}_{t+1} =\prox_{\sigma h^\ast}(\vec{v}_t)$ and $\vec{y}_{t+1}^\prime =\prox_{\sigma h^\ast}^{(\delta_t)}(\vec{v}_t^\prime)$.
By the triangle inequality, \Cref{lem:inexact-exact-prox-rel-lem}, and the nonexpansiveness of the exact proximal operator, we have that
\begin{align*}
    \norm{\vec{y}_{t+1}^\prime-\vec{y}_{t+1}} &\leq \norm{\prox_{\sigma h^\ast}^{(\delta_t)}(\vec{v}_t^\prime) -\prox_{\sigma h^\ast}(\vec{v}_t^\prime)
    } +
    \norm{\prox_{\sigma h^\ast}(\vec{v}_t^\prime) -\prox_{\sigma h^\ast}(\vec{v}_t)} \\
    &\leq
    \sqrt{\sigma\delta_t}
    +\norm{\vec{v}_t^\prime-\vec{v}_t}\\
    &\leq
    \sqrt{\sigma\delta_t}
    +2\sigma\norm{\linop{A}}
    \norm{\vec{x}_{t+1}^\prime-\vec{x}_{t+1}}\\
    &\leq
    \sqrt{\sigma\delta_t} +2\sigma\norm{\linop{A}}\sqrt{\tau\epsilon_t}.
\end{align*}
Consequently, using the definition of the $\|\cdot\|_Z$ norm,
\begin{align*}
    \norm{\linop{T}_{\epsilon_t,\delta_t}(\vec{z}_t)
    -\linop{T}(\vec{z}_t)}_Z^2 &=
    \frac{1}{\tau}
    \norm{\vec{x}_{t+1}^\prime-\vec{x}_{t+1}}^2
    +\frac{1}{\sigma}
    \norm{\vec{y}_{t+1}^\prime-\vec{y}_{t+1}}^2\\
    &\leq
    \epsilon_t+
    \left(
    \sqrt{\delta_t}
    +2\sqrt{\tau\sigma}\norm{\linop{A}}
    \sqrt{\epsilon_t}
    \right)^2,
\end{align*}
which proves the result. The same argument applies to
\(\vec{z}_t^\prime\).
\end{proof}
In order to provide fixed-point characterization of \approxpdhg, we provide a definition of a fixed point of \approxpdhg: 
\begin{definition}\label{def:fixed-point-approx-pdhg} A point $(\vec{x}^\star, \vec{y}^\star)$ is a \textbf{fixed point} of \approxpdhg~ if $\vec{z}_t^\prime \rightarrow \vec{z}^\star$ and
\begin{equation*}
    \lim_{t \rightarrow \infty} \norm{\vec{z}_t^\prime - \linop{T}_{\epsilon_t, \delta_t}(\vec{z}_t^\prime)}_Z = 0.
\end{equation*}
\end{definition}
We now can finally provide a fixed point characterization for \pdhg and \approxpdhg:
\begin{theorem}[Fixed Point Characterization] \label{thm:fixed-point-characterization}
Suppose the error tolerance schedule in \eqref{eqn:inexact-error-tolerance-schedule} is used with \approxpdhg.
Then (1), (2) and (3) are equivalent and (4) implies (3):
\begin{align*}
\textup{(1)}\quad & \vec{x}^\star \text{ is an optimal solution to } \textup{(\ref{prob:main-primal})}; \\
\textup{(2)}\quad & (\vec{x}^\star, \vec{y}^\star) \text{ is a saddle point of } \textup{\eqref{prob:main-saddle-point}}; \\
\textup{(3)}\quad & (\vec{x}^\star, \vec{y}^\star) \text{ is a fixed point of the \pdhg algorithm}; \\
\textup{(4)}\quad & (\vec{x}^\star, \vec{y}^\star) \text{ is a fixed point of the \approxpdhg algorithm}.
\end{align*}
\end{theorem}
\begin{proof}
\noindent (3) $\implies$ (2): Suppose $(\vec{x}^\star,\vec{y}^\star) \in \set{X} \times \set{Y}$ is a fixed point of \pdhg, i.e., 
\begin{align}
    \vec{x}^\star & = \smash{\prox_{\tau g}} \big(\vec{x}^\star - \tau \grad f (\vec{x}^\star) - \tau \linop{A}^\top \vec{y}^\star \big) \label{eqn:fixed-point-step-1} \\
    \vec{y}^\star & = \smash{\prox_{\sigma h^\ast}} \big(\vec{y}^\star + \sigma\linop{A} \vec{x}^\star \big) \label{eqn:fixed-point-step-2}
\end{align}
Then, using \Cref{lem:prox-lemma}, we can say that \eqref{eqn:fixed-point-step-1} is equivalent to the 
\begin{equation*}
    -\nabla f(\vec{x}^\star) - \linop{A}^\top \vec{y}^\star\in \partial g(\vec{x}^\star).
\end{equation*}
Using this, we can say that the following holds for any $\vec{x} \in \set{X}$:
\begin{align}
    g(\vec{x}) &\geq g(\vec{x}^\star) - \ip{ \nabla f(\vec{x}^\star) + \linop{A}^\top \vec{y}^\star}{\vec{x} - 
    \vec{x}^\star} \notag \\
    \iff & \quad \ip{ \grad f (\vec{x}^\star)}{\vec{x} - \vec{x}^\star} + \ip{\vec{y}^\star}{\linop{A}\vec{x} - \linop{A}\vec{x}^\star} \geq g(\vec{x}^\star) - g(\vec{x}), \qquad \forall \vec{x} \in \set{X} \label{eqn:fixed-point-prox-g}
\end{align}
Again, using \Cref{lem:prox-lemma}, we get $\linop{A}\vec{x}^\star \in \partial h^\ast(\vec{y}^\star)$.
Thus, we can say that the following holds for any $\vec{y} \in \set{Y}$:
\begin{align}
    h^\ast(\vec{y}) &\geq h^\ast(\vec{y}^\star) + \ip{ \linop{A}\vec{x}^\star}{\vec{y} - \vec{y}^\star}  \notag \\
    \iff & \quad \ip{ \linop{A}\vec{x}^\star}{\vec{y}^\star - \vec{y}} \geq h^\ast(\vec{y}^\star) - h^\ast(\vec{y}), \qquad \forall \vec{y} \in \set{Y} \label{eqn:fixed-point-prox-h}
\end{align}
Finally, using convexity of $f$, we can say that
\begin{equation} \label{eqn:fixed-point-f-descent}
     \ip{\nabla f(\vec{x}^\star)}{\vec{x}^\star - \vec{x}}  \geq f(\vec{x}^\star) - f(\vec{x}) 
\end{equation}
Adding \eqref{eqn:fixed-point-prox-g}, \eqref{eqn:fixed-point-prox-h} and \eqref{eqn:fixed-point-f-descent} together and then rearranging, we get
\begin{equation*}
       \ip{\linop{A}\vec{x}}{\vec{y}^\star} + f(\vec{x}) + g(\vec{x}) - h^\ast(\vec{y}^\star)  - [\ip{\linop{A}\vec{x}^\star}{\vec{y}} + f(\vec{x}^\star) + g(\vec{x}^\star) - h^\ast(\vec{y})] \geq 0,\quad \vec{x},~\vec{y} \in \set{X} \times \set{Y}.
\end{equation*}
Therefore, based on the saddle point characterization in \eqref{eqn:saddle-point-characterization}, $(\vec{x}^\star, \vec{y}^\star)$ is saddle point since $\mathcal{L}(\mathbf{x}, \mathbf{y}^\star) \geq \mathcal{L}(\mathbf{x}^\star, \mathbf{y})$.

\noindent (2) $\implies$ (1): Suppose that $(\vec{x}^\star, \vec{y}^\star)$ is saddle point. The fact that $\vec{x}^\star$ is a minimizer of \eqref{prob:main-primal} follows from taking the maximum over left-hand side with respect to $\vec{y} \in \set{Y}$ and minimum over the right-hand side with respect to $\vec{x} \in \set{X}$ of \eqref{eqn:saddle-point-characterization}:
\begin{equation*}
    (f + g + h \circ \linop{A})(\vec{x}^\star) = \max_{\vec{y} \in \set{Y}} \mathcal{L}(\vec{x}^\star, \vec{y}) \leq \mathcal{L}(\vec{x}^\star, \vec{y}^\star) \leq \min_{\vec{x} \in \set{X}} \mathcal{L}(\vec{x}, \vec{y}^\star) \leq \mathcal{L}(\vec{x}, \vec{y}^\star) \leq \max_{\vec{y} \in \set{Y}} \mathcal{L}(\vec{x}, \vec{y}) = (f + g + h \circ \linop{A})(\vec{x}).
\end{equation*}
\noindent (1) $\implies$ (3): Next, we will show that the optimal solutions to \eqref{prob:main-primal} are fixed points of \pdhg. 
Suppose that $\vec{x}^*$ is the optimal solution for \eqref{prob:main-primal}.  
Using Corollary 16.50(v) and Proposition 6.19(x) from \citep{bauschke2020correction} and the condition that there exists some $\vec{x}_0 \in \mathrm{ri}(\mathrm{dom}(f)) \cap \mathrm{ri}(\mathrm{dom}(g))$ such that $\linop{A}\vec{x}_0 \in \mathrm{ri}(\mathrm{dom}(h))$ (see \Cref{sec:intro}), we can write the first-order optimality condition for \eqref{prob:main-primal} at point $\vec{x}^*$ as 
\begin{equation}\label{eqn:first-order optimality condition-unconstrained}
    0 \in \nabla f (\vec{x}^*) + \partial g(\vec{x}^*) + \linop{A}^\top \partial h(\linop{A}\vec{x}^*).
\end{equation}
Then \eqref{eqn:first-order optimality condition-unconstrained} implies there is $\vec{u} + \vec{v} + \vec{w} = \vec{0}$ such that
\begin{equation*}
    \vec{u} = \nabla f( \vec{x}^*),\quad \vec{v} \in \partial g(\vec{x}^*), \quad \mbox{and} \quad \vec{w} \in \linop{A}^\top\partial h(\linop{A}\vec{x}^*). 
\end{equation*}
Choose some $\vec{y}^* \in \partial h(\linop{A}\vec{x}^*)$ such that $\vec{w} = \linop{A}^\top \vec{y}^\star$. 
Then, $\vec{v} = - \vec{u} - \vec{w} = - \nabla f(\vec{x}^*) - \linop{A}^\top \vec{y}^*$. 
From \Cref{lem:prox-lemma}, we have the following:
\begin{equation}
    \tau \vec{v} \in \partial (\tau g)(\vec{x}^*) \iff \vec{x}^* = \prox_{\tau g}(\tau \vec{v} + \vec{x}^*) = \prox_{\tau g}(\vec{x}^* - \tau \nabla f(\vec{x}^*) - \tau \linop{A}^\top \vec{y}^*)
\end{equation}
Lastly, we have that $\sigma \linop{A}\vec{x}^\star \in \partial (\sigma h^\ast)(\vec{y}^\star)$ implies that $\vec{y}^\star = \smash{\prox_{\sigma h^\ast}} \big(\vec{y}^\star + \sigma\linop{A} \vec{x}^\star \big)$.
This proves the other direction. \\

\noindent (4) $\implies$ (3): We now show that a fixed point of \approxpdhg~ is a fixed point of \pdhg.
Suppose $\mathbf{z}^\star := (\vec{x}^\star, \vec{y^\star})$ is a fixed point of \approxpdhg~ (see \Cref{def:fixed-point-approx-pdhg}), which means that $\vec{z}_t^\prime \rightarrow \vec{z}^\star$ and
\begin{equation*}
    \lim_{t \rightarrow \infty} [\norm{\vec{z}_t^\prime - \linop{T}_{\epsilon_t, \delta_t}(\vec{z}_t^\prime)}_Z = 0.
\end{equation*}
Given this, the following holds:
\begin{align*}
    \norm{\vec{z}^\star - \linop{T}(\vec{z}^\star)}_Z &= \lim_{t \rightarrow \infty} \norm{\vec{z}_t^\prime - \linop{T}(\vec{z}_t^\prime)}_Z \\
    &\leq \lim_{t \rightarrow \infty} [\norm{\vec{z}_t^\prime - \linop{T}_{\epsilon_t, \delta_t}(\vec{z}_t^\prime)}_Z + \norm{ \linop{T}_{\epsilon_t, \delta_t}(\vec{z}_t^\prime) - \linop{T}(\vec{z}_t^\prime)}_Z ] \\
    &\leq 0 + \lim_{t \rightarrow \infty} \sqrt{    \epsilon_t+
    \left(
    \sqrt{\delta_t}
    +2\sqrt{\tau\sigma}\norm{\linop{A}}
    \sqrt{\epsilon_t}
    \right)^2} = 0
\end{align*}
where we used continuity of $T$ (since both proximal operators are continuous operators) in the first step, again used \Cref{corr:inexact-pdhg-exact-pdhg-operator-rel} to bound the second norm in the second-to-last step and finally \Cref{corr:iterate-to-iterate-bounds-inf-sum} in the last step to say that $\sqrt{    \epsilon_t+\left(\sqrt{\delta_t} +2\sqrt{\tau\sigma}\norm{\linop{A}} \sqrt{\epsilon_t}\right)^2} \rightarrow 0$.
\end{proof}
\section{Application of Inexact PDHG to \eqref{prob:main-2}}\label{sec:inexact-pdhg-application-p-prime}

We now consider the problem template described in \eqref{prob:main-2} with $\set{H} \subset \set{X}$ and $\set{K} \subset \set{Y}$ being convex and compact sets and the dimensions $n$ and $d$ chosen appropriately as per the problem. 
The proximal subproblem in \eqref{eqn:approx-pdhg-update-1} then becomes an Euclidean projection onto the set $\set{H}$.
After applying the generalized Moreau decomposition, the proximal subproblem in \eqref{eqn:approx-pdhg-update-2} then becomes an Euclidean projection onto the set $\set{K}$.
For simplicity of presentation and implementation, we assume that the Euclidean projection onto $\set{K}$ has an exact closed form solution, while the closed-form projection onto $\set{H}$ is either computationally intensive or difficult to obtain and can therefore be done only approximately. 
We can then use the Frank-Wolfe algorithm to solve the inner projection onto $\set{H}$.

\subsection{Application of \eqref{prob:main-2} to Semidefinite Programming}
Despite only considering inexactness with respect to projection onto $\set{H}$, the problem template described in \eqref{prob:main-2} is widely useful and prevalent.
For instance, consider the following positive semidefinite program (SDP):
\begin{equation} \label{eqn:sample-sdp}
\begin{aligned}
& \underset{\mathbf{X} \in\mathbb{R}^{n \times n}}{\text{minimize}} & & f(\mathbf{X}) \\
& \text{subject to} & & \operatorname{Tr}(\mathbf{A}_i \mathbf{X}) = b_i,\quad i=1,\ldots,m,\\
& & & \operatorname{Tr}(\mathbf{X}) = R \\
& & & \mathbf{X} \in \set{S}^n_+.
\end{aligned}
\end{equation}
The most prominent example of an SDP that falls into this template is Goemans and Williamson's SDP relaxation for MaxCut \citep{goemans1995improved}.
The theta number of a graph is the optimal value of the SDP relaxation of the maximum clique problem, which also falls into the above template \citep{de2002approximation}.
This problem template also has a rich set of applications including statistics \citep{huang2014scalable}, microscopy \citep{horstmeyer2015solving}, robotics \citep{rosen2019se} as well as machine learning tasks such as assessing robustness of neural networks \citep{raghunathan2018semidefinite}.
This problem falls into the template of problem \eqref{prob:main-2} with $\set{H} = \set{S}_+^n \cap \{\mathbf{X}~:~\operatorname{Tr}(\mathbf{X}) = R \}$ (i.e. a spectrahedron) and $\set{K} = \{\vec{b}\}$ and $\linop{A}$ is constructed by combining all the affine constraints $\operatorname{Tr}(\mathbf{A}_i \mathbf{X}) = b_i$, $i = 1, \ldots, m$ into a single constraint of the form $\linop{A}\mathbf{X} = \vec{b}$ and the $i$th element of $\vec{b}$ is $b_i$.
Here, the closed-form projection of a general symmetric matrix onto the $\set{H}$, which is a spectrahedron, involves computing a full eigendecomposition of a matrix followed by a simplex projection, which can be difficult to compute for a large sparse matrix. 
For a large sparse matrix that may or may not fit into the memory, it is easier to do an approximate projection using the Frank-Wolfe algorithm onto the set $\set{H}$, which involves evaluating the LMO of the gradient of the projection objective at the current iterate and the main computational overhead is incurred by computation of the smallest eigenpair of this gradient.
On the other hand, the projection onto $\set{K}$ is simply $\vec{b}$ since $\set{K} = \{\vec{b}\}$.

\subsection{Inexact PDHG for \eqref{prob:main-2}}
We now discuss how to modify \approxpdhg~ to solve \eqref{prob:main-2}. 
First, we define the indicator function $i_{\set{H}}$ over the set $\set{H}$ as follows:
\begin{equation}\label{eqn:g-def-p-prime}
    g(\vec{x}) = i_{\set{H}}(\vec{x}) := \begin{cases}
        0 & \mbox{if }  \vec{x} \in \set{H} \\
        +\infty & \mbox{if } \vec{x} \not\in \set{H}.
    \end{cases}
\end{equation}
Given this definition of the function $g$, the proximal subproblem in \eqref{eqn:approx-pdhg-update-1} then becomes an Euclidean projection onto the set $\set{H}$, which we define as follows:
\begin{definition}[Exact Projection onto $\set{H}$]
Let $\set{H} \subset \set{X} \subseteq \mathbb{R}^d$ be a convex and closed set.
Let $g$ be defined as in \eqref{eqn:g-def-p-prime}.
For any $\vec{v} \in \mathbb{R}^d$, we can define the Euclidean projection $\mbox{proj}_{\set{H}}(\vec{v})$ of $\vec{v}$ onto the set $\set{H}$ as follows:
\begin{equation*}
    \vec{w} = \proj_{\set{H}}(\vec{v}) := \prox_{g}(\vec{v}) =  \arg\min_{\vec{x} \in \set{H}} \left \{ \frac{1}{2}\norm{\vec{x} - \vec{v}}^2 \right \},
\end{equation*}
where $\tau > 0$.
\end{definition}
We can apply Frank-Wolfe algorithm or any advanced variant of the same algorithm to obtain an approximate proximal point in the sense of \Cref{def:inexact-proximal-point} or approximate projection onto $\set{H}$  and the degree of inexactness is quantified by the Frank-Wolfe gap.
This relationship is formalized in the following lemma:
\begin{lemma} \label{lem:lmo-error-inexact-subdiff} Let $\set{H} \subset \set{X}$ be a convex and compact set.
Let $g$ be defined as in \eqref{eqn:g-def-p-prime}.
Now suppose that we are given an arbitrary point $\vec{v} \in \set{X}$ and consider the problem of computing $\proj_{\set{H}}(\vec{v})$.
We can define the Frank-Wolfe gap $g_{FW}$ at $\vec{u} \in \set{H}$ as follows:
\begin{equation*}
    g_{FW}(\vec{u}) := \max_{\vec{s} \in \set{H}} \ip{\vec{u} - \vec{v}}{\vec{u} - \vec{s}}.
\end{equation*}
Then $\vec{u}$ satisfies the \textbf{inexact proximal point condition}:
\begin{equation*}
    \vec{v} - \vec{u} \in \partial_\epsilon g(\vec{u}) \mbox{ with } \epsilon \geq g_{FW}(\vec{u}).
\end{equation*}
and we refer to $\vec{u}$ as an \textbf{inexact projection} of $\vec{v}$ onto the set $\set{H}$ with precision $\epsilon$ and denote it as  $\vec{u} := \smash{\proj^{(\epsilon)}_{\set{H}}}(\vec{v})$.
\end{lemma}
\begin{proof} 
We begin by recalling the definition of the $\epsilon$-subdifferential given in \Cref{def:inexact-subdiff} at point $\vec{u}$:
\begin{equation*}
    \vec{g} \in \partial_\epsilon g(\vec{u}) \iff g(\vec{x}) \geq g(\vec{u}) + \ip{\vec{g}}{\vec{x} - \vec{ 
       u}} - \epsilon,~ \forall~\vec{x} \in \set{X}.
\end{equation*}
Since $\vec{u} \in \set{H}$, we have that $g(\vec{u}) = 0$. Since $g(\vec{x}) = +\infty$ if $\vec{x} \not\in \set{H}$ which trivially satisfies the above inequality, we can consider only the case $\vec{x} \in \set{H}$, which means that we can assume that $g(\vec{x}) = 0$ without loss of generality.
Therefore, the condition in \Cref{def:inexact-proximal-point}, which says that
\begin{equation*}
    \vec{v} - \vec{u}\in \partial_\epsilon g(\vec{u}),
\end{equation*}
is equivalent to 
\begin{equation*}
    0 \geq \ip{\vec{v} - \vec{u}}{\vec{x} - \vec{u}} - \epsilon,~\forall~\vec{x} \in \set{H} \iff \epsilon \geq \ip{\vec{u} - \vec{v}}{\vec{u} - \vec{x}},~\forall~\vec{x} \in \set{H}. 
\end{equation*}
This inequality is only satisfied if $\epsilon \geq g_{FW}(\vec{u})$.
This completes the proof.
\end{proof}
Given the notion of an inexact projection onto $\set{H}$, we can proceed to derive the form of the \approxpdhg.
Using the generalized Moreau decomposition, %
\begin{equation*}
    \vec{y}_{t+1}^\prime = \vec{y}_t^\prime + \sigma\linop{A} (2\vec{x}_{t+1}^\prime - \vec{x}_t^\prime) - \sigma \smash{\prox_{\sigma^{-1} h}} \left(\frac{\vec{y}_t^\prime + \sigma\linop{A} (2\vec{x}_{t+1}^\prime - \vec{x}_t^\prime)}{\sigma} \right ).
\end{equation*}
By introducing the variable $\vec{w}_{t+1}^\prime := \vec{y}_t^\prime + \sigma\linop{A} (2\vec{x}_{t+1}^\prime - \vec{x}_t^\prime)$ %
\begin{equation}
\tag{Approximate PDHG - P'}
\begin{aligned}
    \vec{x}_{t+1}^\prime & = \smash{\proj^{(\epsilon_t)}_{\set{H}}} \big(\vec{x}_t^\prime - \tau \grad f (\vec{x}_t^\prime) - \tau \linop{A}^\top \vec{y}_t^\prime \big) \\[0.5em]
    \vec{w}_{t+1}^\prime & = \vec{y}_t^\prime + \sigma\linop{A} (2\vec{x}_{t+1}^\prime - \vec{x}_t^\prime) \big), \\[0.5em]
    \vec{y}_{t+1}^\prime &= \vec{w}_{t+1}^\prime - \sigma \smash{\proj_{\set{K}}} \left(\frac{\vec{w}_{t+1}^\prime}{\sigma} \right).
\end{aligned}
\label{alg:approx_pdhg-p-prime}
\end{equation}
Here $\smash{\proj^{(\epsilon_t)}_{\set{H}}}(\cdot)$ denotes an approximate Euclidean projection onto $\set{H}$ and  $\smash{\proj_{\set{K}}}(\cdot)$ is the exact projection onto $\set{K}$.
\subsection{Frank-Wolfe Algorithm for Projection Subproblem in \eqref{prob:main-2}}
We now discuss the basic Frank-Wolfe algorithm to approximately project a point in $\set{X}$ onto the convex and compact set $\set{H} \subset \set{X}$.
Given a $\mathbf{G} \in \set{X}$, the goal of the approximate projection subproblem onto $\set{H}$ is to find an $\mathbf{H} \in \set{H}$ such that the Frobenius norm of the difference between $\mathbf{G}$ and $\mathbf{H}$ is minimized.
\begin{equation}\label{eqn:proj-subproblem-H}
    \min_{\mathbf{H} \in \set{H}} \phi(\mathbf{H}) \mbox{ with } \phi(\mathbf{H}) :=  \frac{1}{2}\|\mathbf{H} - \mathbf{G}\|_F^2
\end{equation}
We then proceed to apply $n_t$ iteration of the Frank-Wolfe algorithm to this problem for the epoch $t$ of \approxpdhg. 
For inner iteration $k$, the step direction $\mathbf{S}_k$ is a solution of the following problem known as the Linear Minimization Oracle (LMO): 
\begin{equation}\label{eqn:lmo-step}
    \mathbf{S}_k = \arg\min_{\mathbf{U} \in \set{H}}\langle \nabla \phi(\mathbf{H}_k), \mathbf{U}\rangle, 
\end{equation}
The solution of \eqref{eqn:lmo-step} depends on the problem at hand. 
For trace-constrained SDP \eqref{eqn:sample-sdp}, the solution of the LMO is given by $\vec{S}_k = R \vec{v}_{min}\vec{v}_{min}^\top$, where $\vec{v}_{min}$ is the smallest eigenvector of $\mathbf{H}_k - \mathbf{G}$. 

Finally, the current estimate $\mathbf{H}_k$ can be updated to an approximate estimate $\mathbf{H}_{k + 1}$ by taking a step with size $\eta_k$ towards $\mathbf{S}_k$. 
\begin{equation*}
    \mathbf{H}_{k + 1} = (1 - \eta_k) \mathbf{H}_k + \eta_k \mathbf{S}_k
\end{equation*}
For added efficiency, assuming that the Frank-Wolfe algorithm is run for $N_t$ steps for iteation $t$ of \approxpdhg, the estimate of the projection $\mathbf{H}_{N_t}$ is saved and used as a starting estimate for the same projection onto the $\set{H}$ for the next iteration of \approxpdhg.

The step size $\eta_k$ can be found by solving the following optimization problem, which amounts to doing a line search:
\begin{equation}
    \eta_k = \arg\min_{\eta \in [0, 1]} \frac{1}{2} \|(1 - \eta) \mathbf{H}_k + \eta \mathbf{S}_k - \mathbf{G} \|_F^2
\end{equation}
This problem has the following closed form solution:
\begin{equation}
    \eta_k = \min \Bigg\{ \max \Big\{ \frac{\langle \mathbf{H}_k - \mathbf{S}_k  ,\mathbf{H}_k - \mathbf{G}\rangle}{\|\mathbf{H}_k - \mathbf{S}_k \|_F^2}, 0 \Big\},1 \Bigg\} .
\end{equation}

After $k$ iterations, the Frank-Wolfe gap $G(\mathbf{H}_k)$ for estimated projection $\mathbf{H}_k$ is given by

\begin{equation*}
    g_{FW}(\mathbf{H}_k) = \max_{\mathbf{U} \in \set{H}} \ip{\mathbf{H}_k - \mathbf{U}}{\mathbf{H}_k - \mathbf{G}}  = \ip{ \mathbf{H}_k - \mathbf{S}_k}{\mathbf{H}_k - \mathbf{G}} 
\end{equation*}
where $\mathbf{S}_k$ is the estimated step direction for iteration $k$. 
As per \Cref{lem:lmo-error-inexact-subdiff}, in order to obtain an inexact projection $\smash{\proj_{\set{H}}}^{(\epsilon)}(\mathbf{G})$ onto $\set{H}$ with precision of $\epsilon$, we can run Frank-Wolfe algorithm or any variant till 
\begin{equation*}
    g_{FW}(\mathbf{H}_k) =  \ip{\mathbf{H}_k - \mathbf{S}_k}{\mathbf{H}_k - \mathbf{G}} \leq \epsilon.
\end{equation*}
The pseudocode for approximate projection onto $\set{H}$ with the Frank-Wolfe algorithm is provided in \Cref{alg:approx_proj_fw}.
\begin{algorithm} 
\setstretch{1.25}
\caption{Approximate Projection onto $\set{H}$ with Frank-Wolfe: $\mathrm{FW}(\mathbf{G}, \mathbf{H}_0, N_t, \epsilon)$}
\label{alg:approx_proj_fw}
\begin{algorithmic}[1]
\State \textbf{Input: } $\mathbf{G} \in \set{X}$, $\mathbf{H}_0 \in \set{H}$, the maximum number of iterations $N_t$, tolerance $\epsilon$
\For{$k = 0, \ldots, N_t$}
\State Compute gradient: $\nabla \phi(\mathbf{H}_k) = \mathbf{H}_k - \mathbf{G}$
\State Compute LMO: $\mathbf{S}_k = \arg\min_{\mathbf{U} \in \set{H}} \langle \nabla \phi(\mathbf{H}_k), \mathbf{U} \rangle$
\State Frank-Wolfe direction: $\mathbf{D}_k^{FW} := \mathbf{S}_k - \mathbf{H}_k$
\State Compute FW gap: $g_k^{FW} := \langle -\nabla \phi(\mathbf{H}_k), \mathbf{D}_k^{FW} \rangle$
\If{$g_k^{FW} \leq \epsilon$}
    \State \Return $\mathbf{H}_k$
\EndIf
\State Line Search: $\eta_k = \min \Big\{ \max \big\{ \frac{\langle \mathbf{H}_k - \mathbf{S}_k  ,\mathbf{H}_k - \mathbf{G}\rangle}{\|\mathbf{H}_k - \mathbf{S}_k \|_F^2}, 0 \big\},1 \Big\}$
\State Update: $\mathbf{H}_{k + 1} = (1 - \eta_k) \mathbf{H}_k + \eta_k \mathbf{S}_k$
\EndFor
\State Return $\mathbf{H}_{N_t}$
\end{algorithmic}
\end{algorithm}
\subsection{Spectrahedron Frank-Wolfe for Projection Subproblem in \eqref{prob:main-2}}
The main problem with the vanilla Frank-Wolfe algorithm described in \Cref{alg:approx_proj_fw} is that the convergence to the optimal projection is sublinear and the large number of LMO calls significantly slows down the solution of projection subproblem, which in turn slows down the convergence of \approxpdhg~ due to the projection subproblem being solved for every iteration of \approxpdhg. 
The Frank-Wolfe algorithm also exhibits the ``zig-zagging'' phenomena where the current Frank-Wolfe iterate moves in a zig-zag manner between the vertices of the optimal face where the optimal projection $\mathbf{H}^\star$ lies. 
All advanced variants of Frank-Wolfe such as Away Steps (AFW) \citep{abadie1970integer, guelat1986some, lacoste2015global}, Pairwise (PFW) \citep{mitchell1974finding, lacoste2015global} and Fully Corrective Frank-Wolfe (FCW) \citep{holloway1974extension, lacoste2015global} attempt to address this shortcoming by either reducing the number of LMO calls or making use of information learned from previous LMO calls to learn some effective representation of $\set{H}$ in order to search $\set{H}$ more efficiently.
See \citep{braun2025conditional} for a more comprehensive list of more advanced Frank-Wolfe variants along with applications.
See \Cref{tab:fw_comparison_t} for a summary of different advanced variants of Frank-Wolfe in terms of computational complexity. 

One problem with advanced variants of Frank-Wolfe discussed above is that Away Steps, Pairwise and Fully Corrective Frank-Wolfe require the set $\set{H}$ to be a bounded convex polytope in order to achieve linear convergence, while the spectrahedron or the intersection of the positive semidefinite cone with the bounded trace is not a convex polytope and has "infinitely" many vertices and faces. 
To address this problem, \citep{garber2026randomized} showed it is possible to achieve linear convergence after the objective residual decreases below a certain threshold when optimizing a smooth and convex function over the spectrahedron given (1) a quadratic growth condition, which is satisfied since $\phi$ is strongly convex and (2) a strict complementarity condition, which amounts to a positive eigengap at $\nabla \phi(\mathbf{H}^\star)$ and can be stated precisely as follows: Suppose that $\set{H} \subset \mathbb{R}^{n \times n}$ and $\operatorname{rank}(\mathbf{H}^\star) = r^\star$ with $r^\star < n$, then
\begin{equation}\label{eqn:strict-complementarity-eigengap}
    \lambda_{n - r^\star}(\nabla \phi(\mathbf{H}^\star)) - \lambda_{n - r^\star + 1}(\nabla \phi(\mathbf{H}^\star)) = \delta
\end{equation}
for some $\delta > 0$, where $\lambda_i(\nabla \phi(\mathbf{H}^\star))$ denotes the $i$th largest eigenvalue of $\nabla \phi(\mathbf{H}^\star)$ and $\lambda_1(\nabla \phi(\mathbf{H}^\star)) \geq \ldots \geq \lambda_n(\nabla \phi(\mathbf{H}^\star))$.
The algorithm works by taking a regular LMO update, an away step and a randomized pairwise step on a single iteration and choosing the next FW iterate $\mathbf{H}_{k+1}$ based on which of the three steps lowers the projection objective $\phi$ the most. 
The problem of the spectrahedron having "infinite" number of faces is handled by taking a randomized pairwise step that replaces a random rank-1 component of the current iterate $\mathbf{H}_k$ with a new rank-1 component that reduces the objective residual by a constant fraction in expectation if the current iterate $\mathbf{H}_k$ is not aligned with the face defined by the $r^\star$-dimensional subspace defined by the current gradient direction $\nabla \phi (\mathbf{H}_k)$. 
If the iterate $\mathbf{H}_k$ is sufficiently aligned with the face that is defined by the currrent gradient direction $\nabla \phi (\mathbf{H}_k)$, then either an away step or a traditional FW update may lead to a greater decrease in the projection objective $\phi$ and also reduce the objective residual by a constant fraction, so that an overall linear convergence in objective can be achieved (see \Cref{alg:spectrahedron-fw}).
The strict complementarity condition in \eqref{eqn:strict-complementarity-eigengap}, while restrictive, is generally satisfied in practice for any choice of $\mathbf{H}^\star$ and $\mathbf{G}$ as the probability of an exactly zero eigengap is very small.
Furthermore, \citep{garber2026randomized} found that Spectrahedron Frank-Wolfe (SFW)  (see \Cref{alg:spectrahedron-fw}) converges linearly in practice even without strict complementarity condition and has recently proposed an alternative version of the algorithm that converges linearly without strict complementarity assumption \citep{garber2026linear} around the same time as the publication of this paper.
We use the version of the algorithm described in \citep{garber2026randomized}.
The implementation details for Spectrahedron Frank-Wolfe algorithm are in \Cref{app:sfw-impl-details}.
We used ARPACK's implementation of thick-restart Lanczos algorithm \citep{lehoucq1998arpack} as the backend eigensolver for Spectrahedron Frank-Wolfe with tolerance of $10^{-8}$.
Since \citep{garber2026randomized} only provides analysis with unit trace spectrahedron and assuming no eigensolver error, the convergence of Spectrahedron Frank-Wolfe algorithm for a non-unit trace and eigensolver error is provided in \Cref{app:approx-eigvec-comp}.
We find that Spectrahedron Frank-Wolfe often converges to machine precision in terms of LMO-error despite this eigensolver error and this is due to the errors terms decreasing at a geometric rate (see \Cref{thm:sfw-convergence-inexact-evec} and \Cref{thm:lanczos-iter-sfw}).
\begin{algorithm}
\caption{Spectrahedron Frank-Wolfe with Away and Pairwise Steps \citep{garber2026randomized}}
\label{alg:spectrahedron-fw}
\begin{algorithmic}[1]
\State \textbf{Input:} $\mathbf{G} \in \set{X}, \; \mathbf{H}_0 \in \set{H}$, maximum iterations $N_t$
\For{$k = 0,\dots,N_t$}
    \State $\vec{v}_{k,-} = \arg\max_{v \in \mathrm{Im}(\mathbf{H}_k): \|\vec{v}\|= \sqrt{R}} \vec{v}^\top \nabla \phi(\mathbf{H}_k) \vec{v}$
    \State $\lambda_k = (\vec{v}_{k,-}^\top \mathbf{H}_k^\dagger \vec{v}_{k,-})^{-1}$
    \If{$\lambda_k < 1$}
        \State $\mathbf{H}^{\text{drop}}_{k+1} = \frac{1}{1-\lambda_k}\big(\mathbf{H}_k - \lambda_k \vec{v}_{k,-} \vec{v}_{k,-}^\top\big)$
        \If{$\phi(\mathbf{H}^{\text{drop}}_{k+1}) \leq \phi(\mathbf{H}_k)$}
            \State $\mathbf{H}_{k+1} = \mathbf{H}^{\text{drop}}_{k+1}$ \Comment{take a drop step}
        \EndIf
    \EndIf
    \If{drop step was not taken}
        \State $\vec{v}_{k,+} \gets$ leading eigenvector of $-\nabla \phi(\mathbf{H}_k)$ scaled by $\sqrt{R}$
        \State $\eta_k^{\text{FW}} = \arg\min_{\eta \in [0,1]} \phi\big((1-\eta) \mathbf{H}_k + \eta \vec{v}_{k,+} \vec{v}_{k,+}^\top\big)$
        \State $\mathbf{H}^{\text{FW}}_{k+1} = (1-\eta_k^{\text{FW}}) \mathbf{H}_k + \eta_k^{\text{FW}} \vec{v}_{k,+} \vec{v}_{k,+}^\top$
        \State $\eta_k^{\text{AFW}} = \arg\min_{\eta \in [0,\lambda_k]} \phi\!\left(\frac{1}{1-\eta}(\mathbf{H}_k - \eta \vec{v}_{k,-} \vec{v}_{k,-}^\top)\right)$
        \State $\mathbf{H}^{\text{AFW}}_{k+1} = \tfrac{1}{1-\eta_k^{\text{AFW}}}(\mathbf{H}_k - \eta_k^{\text{AFW}} \vec{v}_{k,-} \vec{v}_{k,-}^\top)$
        \State $\vec{u}_{k,-} \gets$ random unit vector in $\mathrm{Im}(\mathbf{H}_k)$ scaled by $\sqrt{R}$
        \State $\gamma_k = (\vec{u}_{k,-}^\top \mathbf{H}_k^\dagger \vec{u}_{k,-})^{-1}$
        \State $\vec{u}_{k,+} \gets$ leading eigenvector of $\gamma_k \vec{u}_{k,-} \vec{u}_{k,-}^\top - \nabla \phi(\mathbf{H}_k)$ scaled by $\sqrt{R}$
        \State $\mathbf{H}^{\text{PFW}}_{k+1}  = \mathbf{H}_k + \gamma_k \big(\vec{u}_{k,+} \vec{u}_{k,+}^\top - \vec{u}_{k,-} \vec{u}_{k,-}^\top\big)$
        \State $\mathbf{H}_{k+1} = \arg\min_{\mathbf{H} \in \{\mathbf{H}^{\text{FW}}_{k+1}, \mathbf{H}^{\text{AFW}}_{k+1}, \mathbf{H}^{\text{PFW}}_{k+1}\}} \phi(\mathbf{H})$
    \EndIf
\EndFor
\State \Return $\mathbf{H}_{N_t}$
\end{algorithmic}
\end{algorithm}

\subsection{Theoretical Guarantees for \eqref{prob:main-2}}
Now we consider the convergence rate of \approxpdhg~ over \eqref{prob:main-2} assuming that some Frank-Wolfe variant is used to solve the projection subproblem in \eqref{eqn:proj-subproblem-H}.
The rate of convergence of each Frank-Wolfe variant is summarized in \Cref{tab:fw_comparison_t}.
In \Cref{sec:inexact-pdhg}, we showed that showed that the ergodic rate of \approxpdhg~ is $\mathcal{O}(1/t)$. 
In the following theorem, for simplicity, though the theory only provides an ergodic rate of $\mathcal{O}(1/t)$, we \textit{assume} that the iterate-to-iterate distance for \approxpdhg~ is given by $\norm{\vec{z}_t^\prime - \vec{z}_{t+1}^\prime}_Z^2 = \frac{C}{(t+1)^\alpha}$ for some $\alpha > 1$ and come up with a bound on the total number of inner iterations to achieve overall accuracy of $\varepsilon > 0$ with \approxpdhg~ over \eqref{prob:main-2}.
Note that this type of polynomial decay regime may not have been shown to hold theoretically in general, we observe this type of behavior from many applications including the ones that we present in the experimental section later.

\begin{theorem}\label{thm:AFW}
Consider \approxpdhg~ applied to \eqref{prob:main-2}.
Suppose that $\set{H}$ is a general convex and compact set.
Define $\vec{z}_t := (\vec{x}_t, \vec{y}_t)$
and $\vec{z}_t^\prime := (\vec{x}_t^\prime, \vec{y}_t^\prime)$. 
Assume that the error tolerance schedule $\{\epsilon_t\}$ is chosen as per \eqref{eqn:inexact-error-tolerance-schedule} by setting
\begin{equation}\label{eqn:inexact-error-tolerance-schedule-p-prime}
    \epsilon_t = \rho_1\kappa  \frac{\norm{\vec{x}_{t+1}^\prime - \vec{x}_t^\prime}^2}{2\tau}.
\end{equation}
Let the following conditions hold:
\begin{itemize}
    \item Assume that the  squared primal iterate-to-iterate distance for \approxpdhg~ is given by $\norm{\vec{x}_t^\prime - \vec{x}_{t+1}^\prime}^2 = \frac{C}{(t+1))^\alpha}$ for some $\alpha > 1$. %
    \item Apply AFW, PFW, FCFW or SFW (if $\set{H}$ is a spectrahedron) algorithm to solve the projection subproblem in \eqref{eqn:proj-subproblem-H}.
    We run AFW, PFW, FCFW or SFW algorithm until the Frank-Wolfe gap is under $\epsilon_t$ for the proximal subproblem in \eqref{eqn:proj-subproblem-H}.
    The FW variant either converges linearly with the type-2 error or LMO error after $k$ inner iterations being $\mathcal{O}(e^{-\rho k})$ for some $\rho > 0$ if $\set{H}$ is bounded convex polytope or if SFW is used under certain assumptions (see \Cref{thm:sfw-convergence-inexact-evec}) and otherwise sublinearly with the error after $k$ inner iterations being $\mathcal{O}(1/k)$.
    \item At outer iteration $t$, the subproblem is solved to accuracy $\epsilon_t$. %
\end{itemize}
Then, to obtain an overall accuracy of at most $\varepsilon > 0$ in terms of the squared primal iterate-to-iterate distance with \approxpdhg, the total number of inner iterations required is
\begin{equation*}
    \mathcal{O}\left( \varepsilon^{-1/\alpha} \log \left( \frac{1}{\varepsilon} \right) \right)
\end{equation*}
if $\set{H}$ is a bounded convex polytope or a spectrahedron (with SFW as chosen FW variant with conditions in \Cref{thm:sfw-convergence-inexact-evec} satisfied) and %
\begin{equation*}
    \mathcal{O}\left( \varepsilon^{-1 -\frac{1}{\alpha}} \right ).
\end{equation*}
if $\set{H}$ is a general convex and compact set without any other special structure.
\end{theorem}
\begin{proof}
From the assumptions, it follows that
\begin{equation*}
    \epsilon_t =  \frac{\rho \kappa C}{2\tau(t+1)^\alpha}
\end{equation*}
Since $\alpha>1$, the series $\sum_{t=0}^{\infty}\epsilon_t$ converges. 

First, consider the case where $\set{H}$ is a bounded
convex polytope, or where $\set{H}$ is a spectrahedron
and SFW is the chosen variant with required conditions for linear convergence satisfied. 
Then the Frank-Wolfe variant converges linearly and the number of inner iterations $N_t$ at outer iteration $t$ is given by
\begin{equation*}
N_t=\mathcal{O}\left(\log \left (\frac{1}{\epsilon_t} \right ) \right).
\end{equation*}
Therefore, setting $c = \frac{\rho \kappa C}{2\tau}$, the total number of inner iterations up to iteration $T$ is
\begin{align*}
    \sum_{t = 0}^T N_t &= \sum_{t = 0}^T \mathcal{O}\left(\log \left (\frac{1}{\epsilon_t} \right ) \right) \\
    &= \sum_{t = 0}^T \mathcal{O} \left(\alpha \log(t+1)) - \log(c) \right) \\
    &=\mathcal{O} \left( \sum_{t=0}^T \log(t+1) \right ) \\
    &= \mathcal{O} (T \log(T)),
\end{align*}
where we used the following inequality in the last step:
\begin{equation*}
    \int_0^{T-1} \log(x + 1)dx \leq \sum_{
    t = 0}^{T-1} \log(t + 1) \leq \int_0^{T-1} \log(x + 1)dx + \log(T-1)
\end{equation*}
Next, consider the case where $\set{H}$ is a general
convex and compact set, then the Frank-Wolfe variant has worst-case sublinear convergence and 
\begin{equation*}
N_t=\mathcal{O}\left(\frac{1}{\epsilon_t}\right).
\end{equation*}
In this case, the total number of inner iterations is given by 
\begin{align*}
    \sum_{t = 0}^T N_t &= \sum_{t = 0}^T \mathcal{O}\left(\frac{1}{\epsilon_t}\right) \\
    &= \sum_{t = 0}^T \mathcal{O}\left((t+1)^\alpha \right) \\
    &= \mathcal{O}(T^{1+\alpha}).
\end{align*}
Finally, suppose that the overall accuracy is measured by
the squared primal iterate-to-iterate distance. To ensure that $\|\vec x_T^\prime-\vec x_{T+1}^\prime\|^2\leq\varepsilon$, it suffices to run $\mathcal O(\varepsilon^{-1/\alpha})$ outer iterations.
This completes the proof.
\end{proof}
\begin{table}
\centering
\large  
\resizebox{\textwidth}{!}{%
\begin{tabular}{@{} l p{3.8cm} p{4.2cm} p{2.8cm} p{3.5cm} @{}}
\toprule
\textbf{Algorithm} & \multicolumn{2}{c}{\textbf{Convergence on Projection Subproblem \eqref{eqn:proj-subproblem-H}}} & \textbf{Per-Iteration Cost} & \textbf{Reference(s)} \\
\cmidrule(lr){2-3}
& \textbf{General $\set{H}$ (e.g., Spectrahedron in \eqref{eqn:sample-sdp})} 
& \textbf{Bounded Convex Polytope} 
&  &  \\
\midrule
Frank-Wolfe (FW) 
& $\mathcal{O}(1/k)$ 
& $\mathcal{O}(1/k)$ %
& Low (1 LMO call) 
& \citep{frank1956algorithm, jaggi2013revisiting} \\

Away-step FW (AFW) 
& $\mathcal{O}(1/t(k))$ 
& $\mathcal{O}(\exp(-\rho t(k)))$, $t(k) \geq k/2$ 
& Moderate (LMO + away step) 
& \citep{abadie1970integer, guelat1986some, lacoste2015global} \\

Pairwise FW (PFW) 
& $\mathcal{O}(1/t(k))$ 
& $\mathcal{O}(\exp(-\rho t(k)))$, $t(k) \geq k/(3|\set{A}|! + 1)$ 
& Moderate (LMO + pairwise update) 
& \citep{lacoste2015global} \\

Fully Corrective FW (FCFW) 
& $\mathcal{O}(1/t(k))$ 
& $\mathcal{O}(\exp(-\rho k))$ 
& High (QP over active set) 
& \citep{holloway1974extension, lacoste2015global} \\

Lazy FW (LFW) 
& Likely $\mathcal{O}(1/k)$ (not provided) 
& Likely $\mathcal{O}(\exp(-\Theta(k)))$ (not provided) 
& Very Low (cached LMO) 
& \citep{braun2017lazy} \\

Spectrahedron FW (SFW) 
& Linear convergence in expectation under certain assumptions (see \Cref{thm:sfw-convergence-inexact-evec} with $\theta_k = 0$)
& N/A 
& Moderate (LMO + pairwise update + away step) 
& \citep{garber2026randomized} \\
\bottomrule
\end{tabular}
}
\caption{Comparison of Frank-Wolfe algorithm variants with convergence rates in terms of iteration count \( k \).}
\label{tab:fw_comparison_t}
\end{table}
Based on this, if we assume that the primal iterate-to-iterate distance of \approxpdhg~ goes down at polynomial decay rate of $C/(t + 1)^\alpha$ for some $C > 0$ for \eqref{eqn:sample-sdp} and we use Spectrahedron FW for the corresponding projection subproblem, then we need $\mathcal{O}\left (\varepsilon^{-1/\alpha} \log \left ( \frac{1}{\varepsilon}\right ) \right )$ total inner iterations for solving \eqref{eqn:sample-sdp} to accuracy of $\varepsilon$ over the objective residual after the objective residual decreases below a specific threshold detailed in \Cref{thm:sfw-convergence-inexact-evec} assuming the projection subproblems satisfy strict complementarity (see \eqref{eqn:strict-complementarity-eigengap} for definition of strict complementarity).
We use spectral warm starts as a heuristic to achieve linear convergence with Spectrahedron FW by getting the initial objective residual for the projection subproblem to be low enough as per \Cref{thm:sfw-convergence-inexact-evec} but this is not a provably guaranteed way to achieve this.
Otherwise, we generally need $\mathcal{O}(\varepsilon^{-1-1/\alpha})$ total inner iterations to achieve an accuracy of $\varepsilon$.
\section{Experiments} \label{sec:experiments}
\subsection{Comparison of Spectrahedron Frank-Wolfe and Vanilla Frank-Wolfe}
In this experiment, we systematically benchmark Spectrahedron Frank-Wolfe (SFW) and Vanilla FW (VFW) projection algorithms on randomly generated symmetric matrices of varying sizes, ranging from $n=500$ to 
$n=10,000$ with all elements are from a standard normal distribution prior to symmetrization. 
Both employ low-rank factorization representation based on eigendecomposition of the projection iterate to conserve memory.
For each matrix, we compute the exact projection onto the spectrahedron using a closed-form eigendecomposition as a baseline. 
Both FW algorithms are initialized with a low-rank factorization and are run for a fixed number of iterations. 
All of the experiments here were conducted on a Linux machine equipped with an Intel Core Ultra 9 275HX processor featuring 24 physical cores and 24 hardware threads.
We record the objective value, relative error with respect to the exact solution that is based on full eigendecomposition, the primal-dual gap (FW gap), the number of iterations, and the computational time for each method. 
This setup enables a detailed comparison of the algorithms' efficiency and accuracy as the problem dimension increases and provides a controlled environment to assess the scalability and practical performance of Frank-Wolfe-based projection methods for large-scale semidefinite programs.
The results are summarized in \Cref{tab:fw_runtime_speed} and \Cref{tab:fw_accuracy_iterations} and they show that the speedup of SFW relative to the exact solver generally increases with dimension and that SFW generally scales better than VFW with increasing dimension, achieving a lower relative error and LMO error with lesser number of iterations than VFW.
The non-monotonic increase (ex. $n = 2000$) in runtime with dimension can be explained by difficulty posed by spread of the eigenvalue spectrum varying from one instance to the other.
\begin{table}[H]
\centering
\caption{Runtime and speedup comparison between exact solver, Spectrahedron Frank Wolfe (SFW), and Vanilla Frank Wolfe (VFW). The speedup is relative to the exact solver.}
\label{tab:fw_runtime_speed}
\begin{tabular}{r
                S[table-format=2.4]
                S[table-format=2.4]
                S[table-format=2.4]
                S[table-format=1.2]
                S[table-format=1.2]}
\toprule
{$n$} & {Exact time (s)} & {SFW time (s)} & {VFW time (s)} & {SFW speedup} & {VFW speedup} \\
\midrule
 500  & 0.2765 & 0.3279 & 2.3350 & 0.84 & 0.12 \\
1000  & 0.0705 & 0.2310 & 2.3540 & 0.31 & 0.03 \\
2000  & 0.3480 & 0.2867 & 0.7600 & 1.21 & 0.46 \\
4000  & 2.2146 & 1.1208 & 6.6909 & 1.98 & 0.33 \\
5000  & 4.0454 & 0.9805 & 9.2166 & 4.13 & 0.44 \\
7000  & 10.0460 & 2.9516 & 16.5774 & 3.40 & 0.61 \\
10000 & 28.4494 & 7.2194 & 31.3728 & 3.94 & 0.91 \\
\bottomrule
\end{tabular}
\end{table}
\begin{table}[H]
\centering
\caption{Accuracy, iteration count, and linear minimization oracle (LMO) error for Spectrahedron Frank Wolfe (SFW) and Vanilla Frank Wolfe (VFW).}
\label{tab:fw_accuracy_iterations}
\begin{tabular}{r
                S[table-format=1.2e-1]
                S[table-format=1.2e-1]
                S[table-format=2.0]
                S[table-format=2.0]
                S[table-format=1.2e-1]
                S[table-format=1.2e-1]}
\toprule
{$n$} 
& {SFW rel. error} 
& {VFW rel. error} 
& {SFW iters} 
& {VFW iters} 
& {SFW LMO error} 
& {VFW LMO error} \\
\midrule
 500  & 5.78e-16 & 7.37e-07 & 3 & 25 & 0.00e+00 & 3.59e-06 \\
1000  & 1.61e-16 & 6.33e-07 & 3 & 25 & 1.42e-14 & 4.22e-05 \\
2000  & 4.53e-16 & 6.55e-09 & 3 & 7  & 2.84e-14 & 9.23e-09 \\
4000  & 2.39e-06 & 1.39e-07 & 4 & 25 & 2.84e-14 & 4.83e-06 \\
5000  & 1.42e-14 & 3.04e-08 & 3 & 25 & 1.42e-14 & 8.56e-08 \\
7000  & 8.71e-06 & 3.44e-08 & 3 & 25 & 2.84e-09 & 1.00e-06 \\
10000 & 6.48e-07 & 2.70e-08 & 4 & 25 & 0.00e+00 & 1.36e-04 \\
\bottomrule
\end{tabular}
\end{table}

\subsection{Spectral warm starts}

The projection onto the spectrahedron is repeatedly solved along the
outer PDHG trajectory.  Since consecutive projection subproblems are
closely related, we seek to construct an inexpensive spectral
approximation that provides a favorable initialization for the subsequent
projection iterations.

Rather than solving the associated eigenvalue
problem to high accuracy, we compute only the dominant part of its spectrum using an adaptive Lanczos procedure.  
Starting with a relatively small Krylov subspace of dimension $k_{\mathrm{init}}$ and updating to a dimension of $k$, Lanczos computes
$k$ Ritz pairs:
\begin{equation*}
    (\mathbf{u}_i,\lambda_i)_{i=1}^{k},
    \qquad
    \lambda_1 \geq \lambda_2 \geq \cdots \geq \lambda_k .
\end{equation*}
The Ritz vectors are validated using their residuals:
\begin{equation*}
    r_i
    =
    \left\|
        A \mathbf{u}_i-\lambda_i \mathbf{u}_i
    \right\|_2, \quad \forall i \in \{0, \ldots, k\}.
\end{equation*}
We deliberately use a looser eigensolver tolerance than in the high-accuracy projection routine, since the purpose of the warm start is to obtain a useful approximation at low computational cost rather than to solve the projection problem to high precision.

Given the computed dominant eigenvalues, we construct the corresponding spectrahedral coefficients by water-filling.  
In particular, for a
trace-radius $R$, we set
\begin{equation*}
    \alpha_i = (\lambda_i-\theta)_+,
\end{equation*}
where $\theta$ is chosen so that
\begin{equation*}
    \sum_i \alpha_i = R.
\end{equation*}
The resulting factorized point is
\begin{equation*}
    X_{\mathrm{WS}}
    =
    U\operatorname{Diag}(\alpha)U^{\mathsf T},
\end{equation*}
where $\mathbf{U}=[\mathbf{u}_1,\ldots,\mathbf{u}_k]$ contains the computed Ritz vectors.

The Lanczos subspace is enlarged adaptively when the available spectrum
is insufficient for the water-filling construction.  
We use the following relative boundary test:
\begin{equation*}
    \beta_k = \frac{\alpha_k}{\max\{1,|\alpha_1|\}}
\end{equation*}
as an inexpensive measure of spectral completeness.  
If $\beta_k$ exceeds a prescribed threshold, we increase the Lanczos
dimension:
\begin{equation*}
    k \leftarrow k+k_{\mathrm{step}},
\end{equation*}
and recycle the dominant Ritz vector as the starting vector for the next eigensolve.  The procedure terminates once the computed spectrum is deemed sufficient or a prescribed maximum subspace dimension is reached.

Most importantly, the spectral warm start does not perform Frank-Wolfe iterations.  
Its computational role is solely to identify a sufficiently accurate dominant spectral subspace and to construct a feasible point using this for Spectrahedron Frank-Wolfe (SFW).  
The subsequent SFW iterations are then started from this point.
\begin{figure}[t]
    \centering
    \resizebox{\linewidth}{!}{%
    \begin{tikzpicture}[
        >=Latex,
        font=\footnotesize,
        node distance=9mm and 9mm,
        every node/.style={align=center},
        process/.style={
            rectangle, rounded corners=3pt, draw=black!70, thick,
            minimum width=29mm, minimum height=12mm,
            fill=gray!10
        },
        data/.style={
            rectangle, rounded corners=3pt, draw=blue!60!black, thick,
            minimum width=29mm, minimum height=12mm,
            fill=blue!8
        },
        decision/.style={
            diamond, aspect=2, draw=orange!70!black, thick,
            inner sep=0pt, minimum width=29mm, minimum height=14mm,
            fill=orange!12
        },
        arrow/.style={->, thick, black!75},
        lbl/.style={font=\scriptsize, inner sep=1.5pt}
    ]

    \node[data] (A)
        {PDHG operator input\\[-0.5mm]
         $\vec{x}_t-\tau\grad f(\vec{x}_t)-\tau\linop{A}^\top\vec{y}_t$};

    \node[process, right=of A] (lanczos)
        {Adaptive Lanczos\\[-0.5mm]
         $k=k_{\rm init}$, $\texttt{tol}=\tau_{\rm SWS}$};

    \node[data, right=of lanczos] (ritz)
        {Ritz pairs\\[-0.5mm]
         $(U_k,\lambda_k)$};

    \node[process, right=of ritz] (residual)
        {Ritz residual\\[-0.5mm] validation};

    \node[process, below=16mm of residual] (waterfill)
        {Water-filling\\[-0.5mm]
         $\alpha_i=(\lambda_i-\theta)_+$};

    \node[decision, left=of waterfill] (complete)
        {Spectrum\\[-0.5mm] sufficient?};

    \node[data, left=of complete] (xws)
        {Spectral warm start\\[-0.5mm]
         $X_{\rm WS}=U\,\mathrm{diag}(\alpha)\,U^\top$};

    \node[process, left=of xws] (sfw)
        {SFW projection /\\[-0.5mm] subsequent PDHG};

    \draw[arrow] (A) -- (lanczos);
    \draw[arrow] (lanczos) -- (ritz);
    \draw[arrow] (ritz) -- (residual);
    \draw[arrow] (residual) -- (waterfill);
    \draw[arrow] (waterfill) -- (complete);
    \draw[arrow] (complete) -- node[above, lbl] {yes} (xws);
    \draw[arrow] (xws) -- (sfw);

    \draw[arrow, dashed, orange!70!black]
        (complete.north) -- node[right, lbl, pos=0.35] {no}
        ($(complete.north)+(0,8mm)$)
        -| (lanczos.south);
    \node[lbl, text width=42mm, orange!50!black]
        at ($(lanczos.south)!0.5!(ritz.south)+(0,-7mm)$)
        {$k\leftarrow k+k_{\rm step}$;\\ recycle dominant Ritz vector};

    \node[font=\scriptsize\itshape, text width=60mm, below=5mm of $(xws.south)!0.5!(sfw.south)$]
        {No Frank--Wolfe iterations are performed\\ during the warm start.};

    \end{tikzpicture}%
    }
    \caption{
        Spectral warm start based on an adaptive Lanczos approximation.
        Starting from the PDHG operator, a low-accuracy Lanczos eigensolve
        computes the dominant Ritz pairs.  After residual validation, the
        eigenvalues are water-filled to construct a feasible spectrahedral
        approximation.  If the computed spectrum is insufficient, the
        Lanczos subspace is enlarged and the dominant Ritz vector is recycled.
        The resulting $X_{\rm WS}$ is then supplied to the subsequent
        spectrahedron projection/SFW procedure.
    }
    \label{fig:spectral-warm-start}
\end{figure}
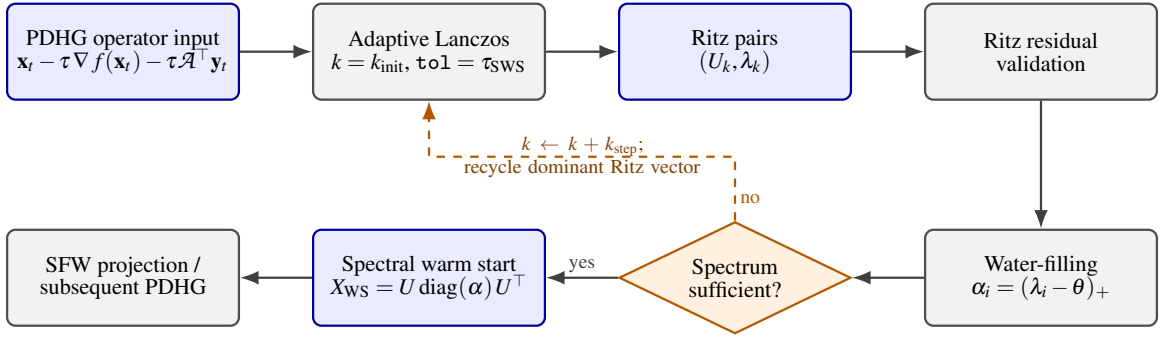
\subsection{Inexact PDHG with Spectral Warm Starts and SFW}
We compare two implementations of the projection step in the outer PDHG
method. The first uses an accurate Lanczos-based projection with eigensolver
tolerance $10^{-8}$. The second combines an adaptive Lanczos spectral warm
start (nominal tolerance $10^{-2}$) with Spectrahedron Frank--Wolfe (SFW),
using at most 25 SFW iterations per projection and we set a floor of $10^{-6}$ for the inner subproblem tolerance computed using \eqref{eqn:inexact-error-tolerance-schedule-p-prime}. 
Both methods use the same MaxCut SDP formulation and are evaluated on GSet instances G38, G55, G60 and G70.
We report objective residual and diagonal feasibility error against both outer iterations and cumulative matrix-vector products (matvecs). 
The latter is the more informative measure of computational work because it accounts for the different costs of the two projection routines in a architecture-neutral manner unlike runtime, which could vary based on language of implementation for various components and the architecture used for evaluation.
The spectral warm-start experiments were conducted on HPC2N nodes equipped with two Intel Xeon Gold 6132 processors (28 physical CPU cores in total), 192 GB of RAM, and two NVIDIA Tesla V100 GPUs per node.
We ran warm-started SFW for a fixed number of epochs, measured the number of matvecs and ran accurate-Lanczos baseline for the same number of epochs up to the same number of matvecs used by the warm-started SFW run.

Across all three instances, the two methods follow similar outer convergence trajectories initially. 
The difference becomes most visible when the objective
residual begins its final decrease. 
On G55 and G60, the accurate-Lanczos run stagnates at an objective residual of roughly $10^{-3}$ within the available matvec budget, whereas warm-started SFW continues to reduce the residual to around $10^{-4}$ or below. 
On G38, both methods reach a much smaller objective residual and the warm-started SFW run reaches its minimum with fewer matvecs.
For G70, warm-started SFW reaches a smaller objective residual for the given matvec budget.
Thus, the benefit is not that every outer iteration is cheaper in isolation by the same factor, but that the proposed projection strategy converts a comparable matvec budget into more useful outer progress, especially on the larger instances.

The feasibility curves tell a consistent story. Both methods reduce the diagonal feasibility error by several orders of magnitude, but warm-started SFW generally reaches a lower terminal error within the available matvec budget.
The results indicate that solving every projection to the stringent $10^{-8}$ Lanczos tolerance is unnecessary for effective outer convergence on these instances. 
Reusing spectral information and correcting the projection with a small number of SFW steps yields a more favorable
accuracy-work trade-off.

\begin{figure}[t]
    \centering
    \begin{subfigure}{0.48\textwidth}
        \centering
        \includegraphics[width=\linewidth]{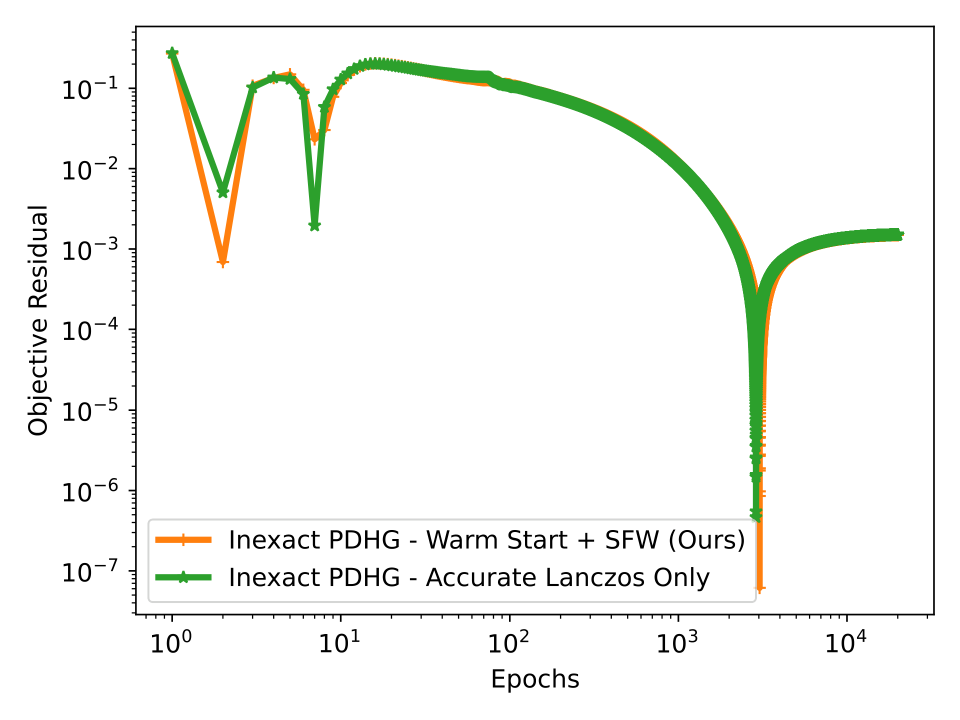}
        \caption{Objective residual vs. outer iterations.}
    \end{subfigure}
    \hfill
    \begin{subfigure}{0.48\textwidth}
        \centering
        \includegraphics[width=\linewidth]{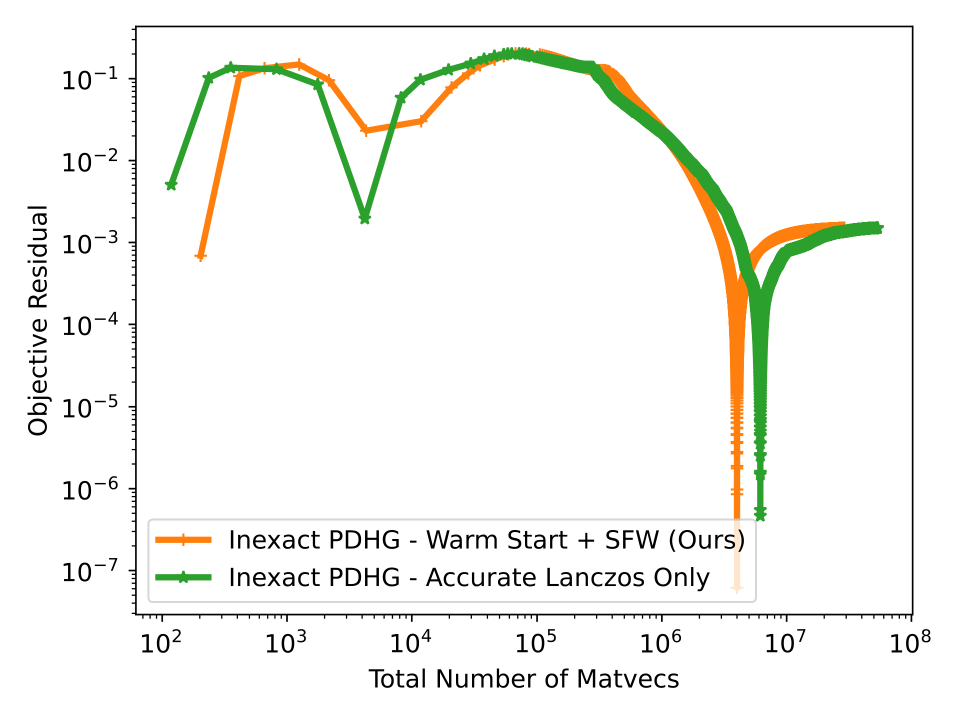}
        \caption{Objective residual vs. cumulative matvecs.}
    \end{subfigure}

    \vspace{0.25cm}
    \begin{subfigure}{0.48\textwidth}
        \centering
        \includegraphics[width=\linewidth]{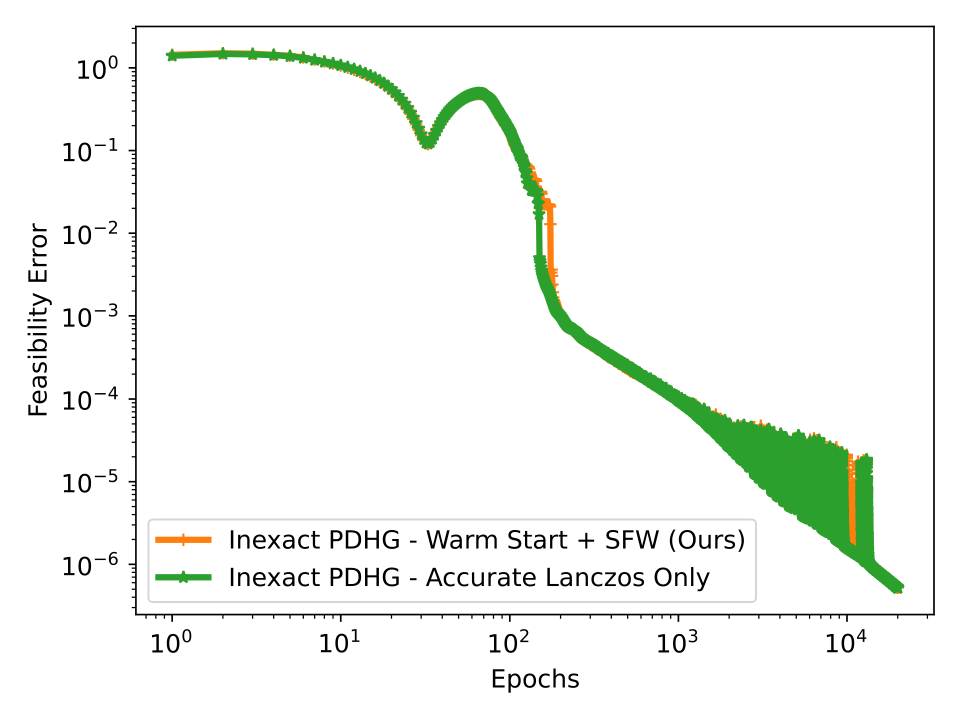}
        \caption{Feasibility error vs. outer iterations.}
    \end{subfigure}
    \hfill
    \begin{subfigure}{0.48\textwidth}
        \centering
        \includegraphics[width=\linewidth]{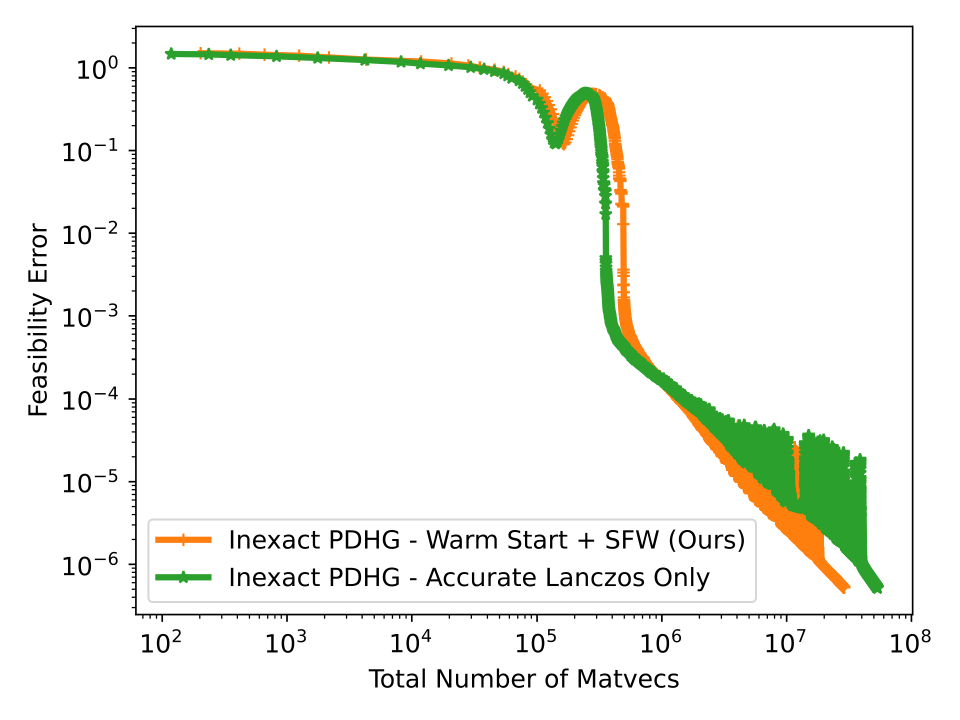}
        \caption{Feasibility error vs. cumulative matvecs.}
    \end{subfigure}
    \caption{Outer convergence on GSet G38 for Inexact PDHG with warm-started
    SFW and Inexact PDHG with accurate Lanczos projection ($10^{-8}$).
    The two methods achieve similar accuracy, while warm-started SFW reaches
    its smallest observed objective residual with less cumulative projection
    work.}
    \label{fig:g38_inexact_comparison}
\end{figure}

\begin{figure}[t]
    \centering
    \begin{subfigure}{0.48\textwidth}
        \centering
        \includegraphics[width=\linewidth]{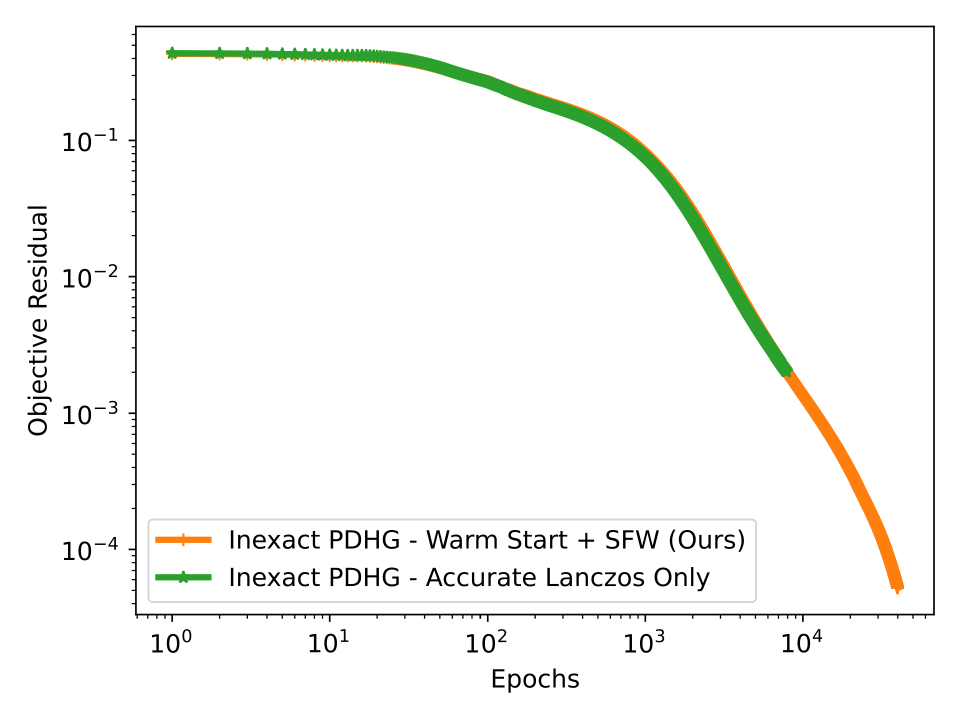}
        \caption{Objective residual vs. outer iterations.}
    \end{subfigure}
    \hfill
    \begin{subfigure}{0.48\textwidth}
        \centering
        \includegraphics[width=\linewidth]{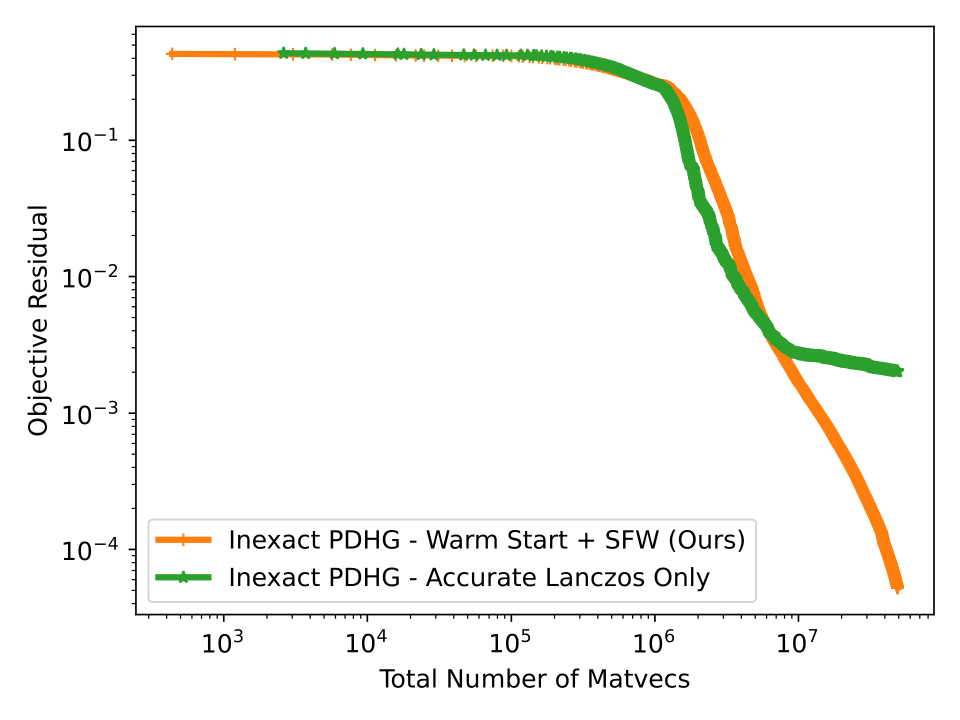}
        \caption{Objective residual vs. cumulative matvecs.}
    \end{subfigure}

    \vspace{0.25cm}
    \begin{subfigure}{0.48\textwidth}
        \centering
        \includegraphics[width=\linewidth]{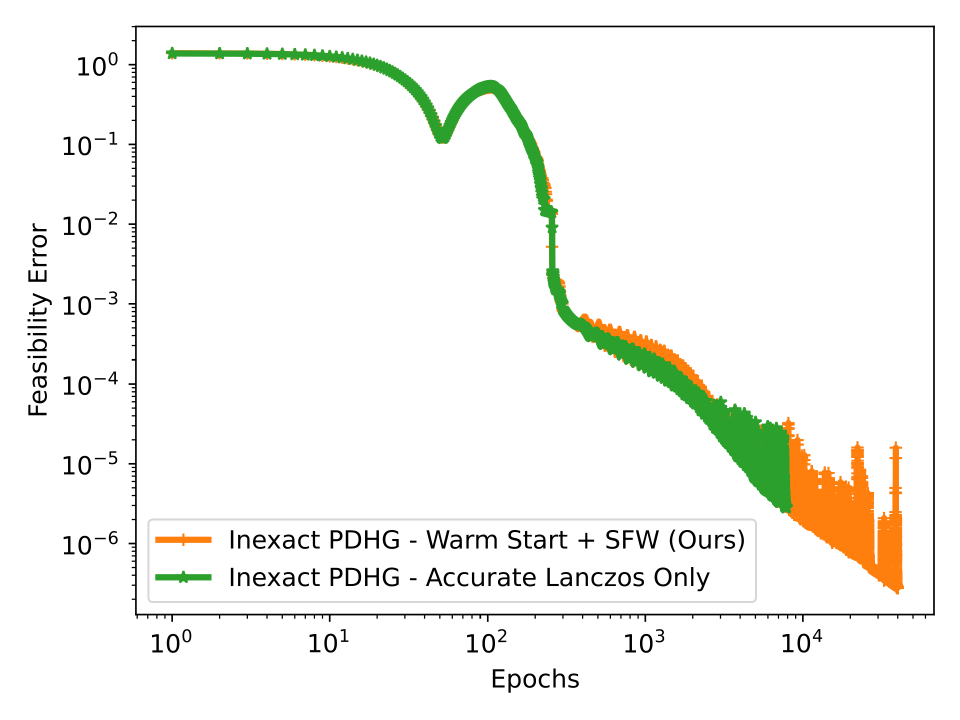}
        \caption{Feasibility error vs. outer iterations.}
    \end{subfigure}
    \hfill
    \begin{subfigure}{0.48\textwidth}
        \centering
        \includegraphics[width=\linewidth]{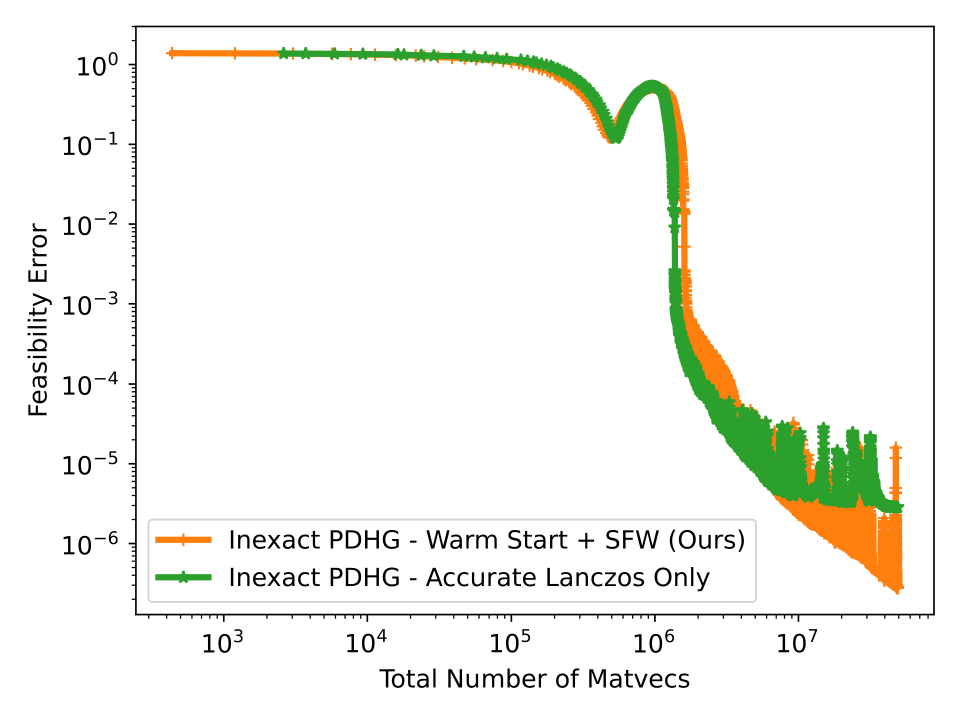}
        \caption{Feasibility error vs. cumulative matvecs.}
    \end{subfigure}
    \caption{Outer convergence on GSet G55. The methods initially behave
    similarly, but the accurate-Lanczos run levels off near an objective
    residual of $10^{-3}$, whereas warm-started SFW continues toward
    $10^{-4}$ with additional matvecs. The feasibility errors are comparable
    during the main decrease, with warm-started SFW attaining a smaller
    terminal error within the available matvec budget.}
    \label{fig:g55_inexact_comparison}
\end{figure}

\begin{figure}[t]
    \centering
    \begin{subfigure}{0.48\textwidth}
        \centering
        \includegraphics[width=\linewidth]{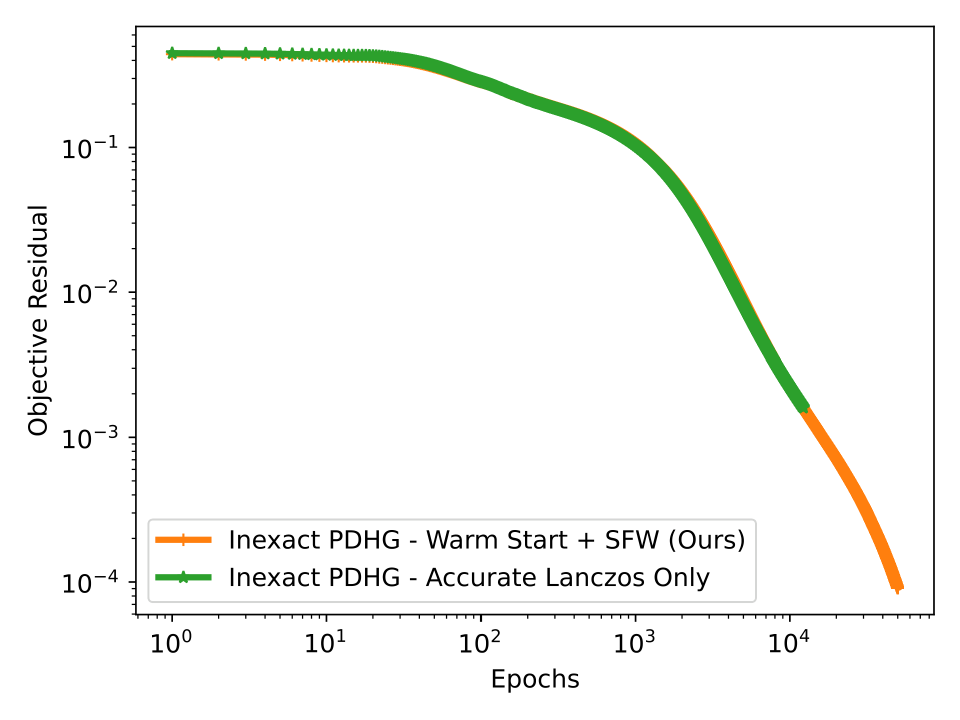}
        \caption{Objective residual vs. outer iterations.}
    \end{subfigure}
    \hfill
    \begin{subfigure}{0.48\textwidth}
        \centering
        \includegraphics[width=\linewidth]{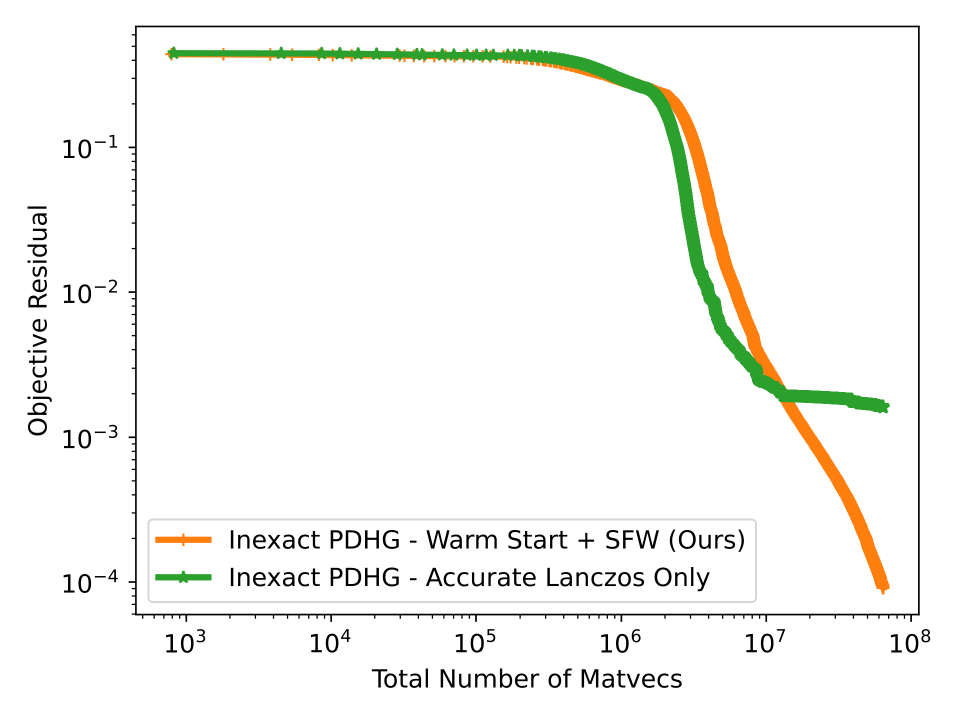}
        \caption{Objective residual vs. cumulative matvecs.}
    \end{subfigure}

    \vspace{0.25cm}
    \begin{subfigure}{0.48\textwidth}
        \centering
        \includegraphics[width=\linewidth]{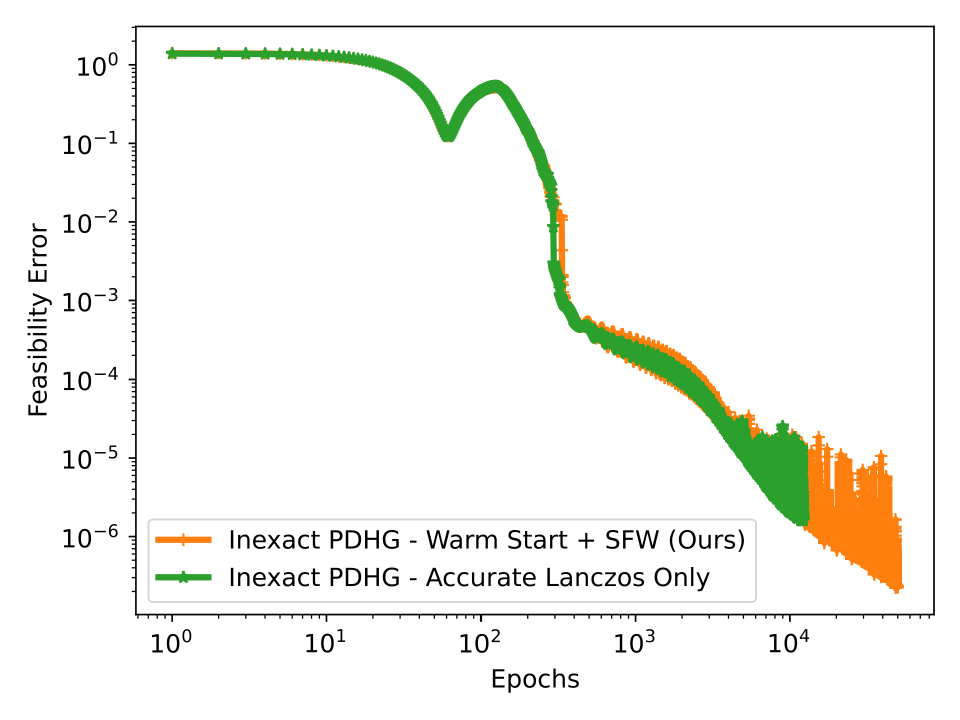}
        \caption{Feasibility error vs. outer iterations.}
    \end{subfigure}
    \hfill
    \begin{subfigure}{0.48\textwidth}
        \centering
        \includegraphics[width=\linewidth]{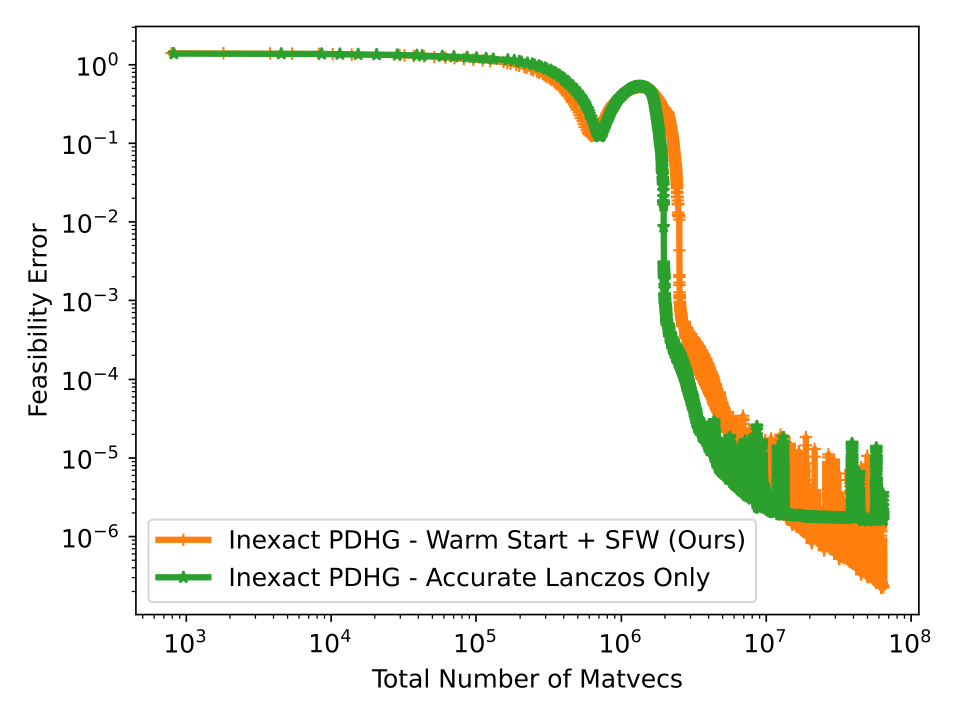}
        \caption{Feasibility error vs. cumulative matvecs.}
    \end{subfigure}
    \caption{Outer convergence on GSet G60. Warm-started SFW continues to
    decrease the objective residual after the accurate-Lanczos run has
    slowed down in terms of decrease in objective residual with matvecs, reaching below $10^{-4}$ in the plotted range. It also
    attains a lower terminal feasibility error, illustrating the advantage
    of measuring progress against cumulative matvec work.}
    \label{fig:g60_inexact_comparison}
\end{figure}

\begin{figure}[t]
    \centering
    \begin{subfigure}{0.48\textwidth}
        \centering
        \includegraphics[width=\linewidth]{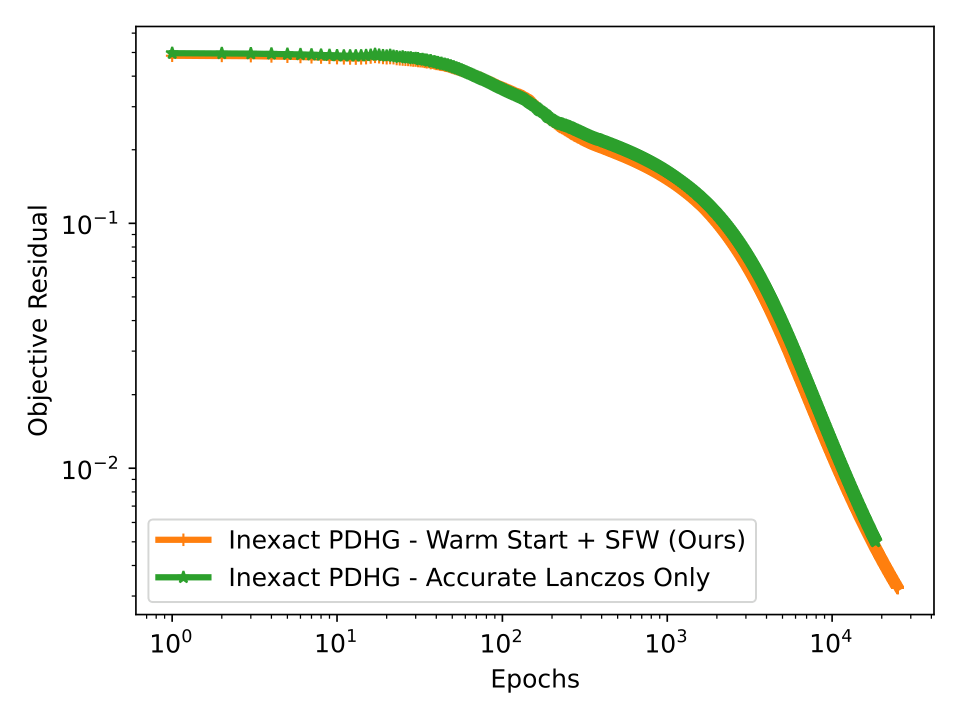}
        \caption{Objective residual vs. outer iterations.}
    \end{subfigure}
    \hfill
    \begin{subfigure}{0.48\textwidth}
        \centering
        \includegraphics[width=\linewidth]{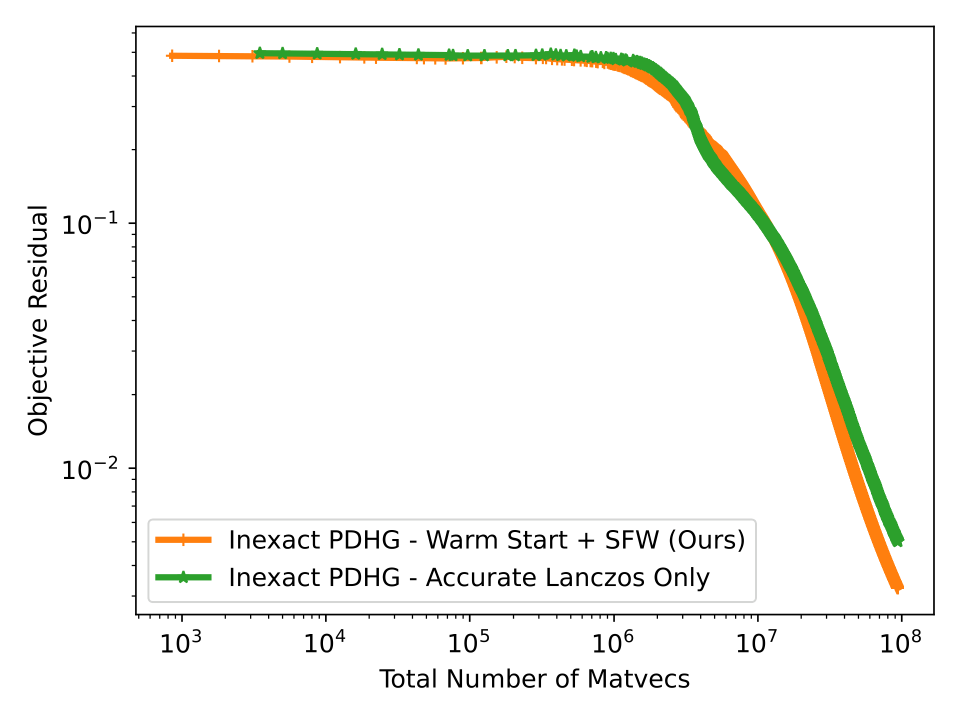}
        \caption{Objective residual vs. cumulative matvecs.}
    \end{subfigure}

    \vspace{0.25cm}
    \begin{subfigure}{0.48\textwidth}
        \centering
        \includegraphics[width=\linewidth]{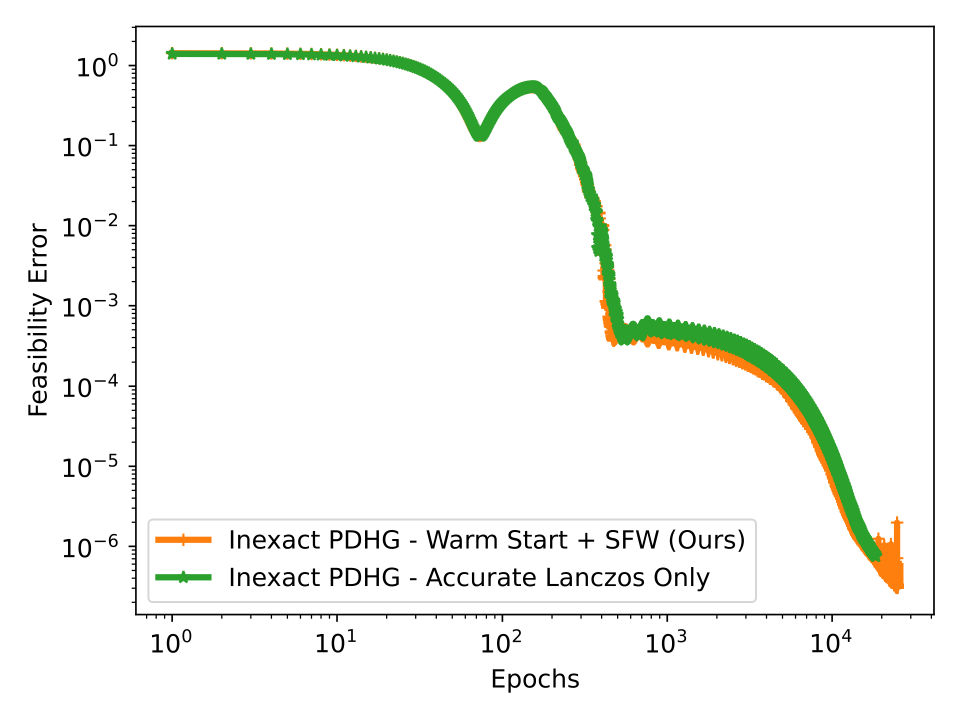}
        \caption{Feasibility error vs. outer iterations.}
    \end{subfigure}
    \hfill
    \begin{subfigure}{0.48\textwidth}
        \centering
        \includegraphics[width=\linewidth]{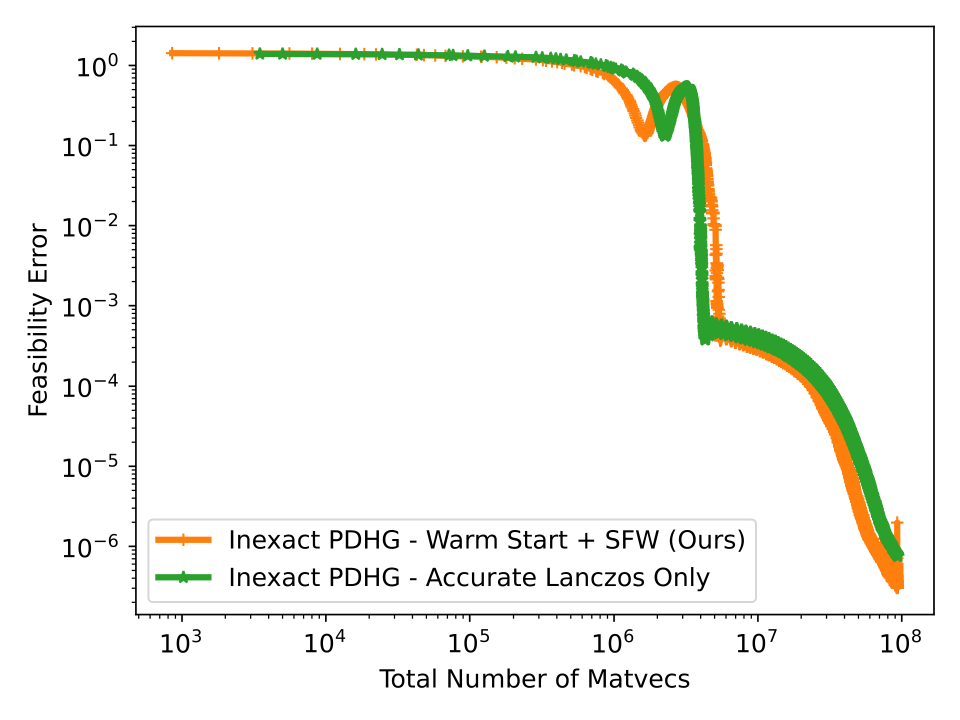}
        \caption{Feasibility error vs. cumulative matvecs.}
    \end{subfigure}
    \caption{Outer convergence on GSet G70. The methods behave similarly but the warm-started SFW run reaches a smaller objective residual and feasibility error than accurate-Lanczos run in the end.}
    \label{fig:g70_inexact_comparison}
\end{figure}

\paragraph{Projection diagnostics.}
The number of SFW iterations per projection varies across the PDHG run rather than being fixed at its maximum of 25 SFW iterations, and both the projection tolerance and final LMO error also vary over
the outer trajectory.
Generally, a lower projection error tolerance is demanded at the end of the PDHG run than at the beginning of the PDHG run. 
The corrected primal-rank plots for G38, G55 and G60 show a similar qualitative pattern for the two projection strategies: the rank rises rapidly to a large value during the early iterations, then decreases sharply and eventually settles at a much smaller rank. 
The two rank trajectories largely overlap, with only modest differences in the late stages since SFW slightly adjusts the rank after warm-start procedure for fine tuning the projection.
Thus, the work-normalized advantage of warm-started SFW is not accompanied by a
qualitatively different primal-rank trajectory. 
Rather, the plots are consistent with the interpretation that reusing spectral information and
applying a small, adaptive number of SFW corrections reduces projection work
while preserving the outer method's overall trajectory.

\begin{figure}[t]
    \centering
    \begin{subfigure}{0.48\textwidth}
        \centering
        \includegraphics[width=\linewidth]{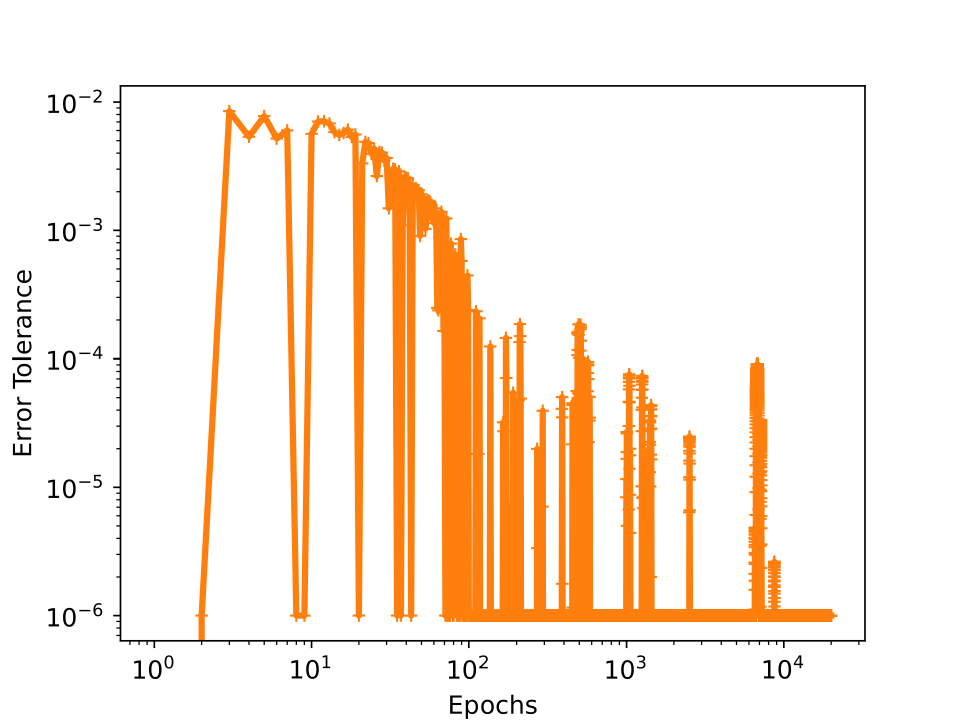}
        \caption{Projection error tolerance.}
    \end{subfigure}
    \hfill
    \begin{subfigure}{0.48\textwidth}
        \centering
        \includegraphics[width=\linewidth]{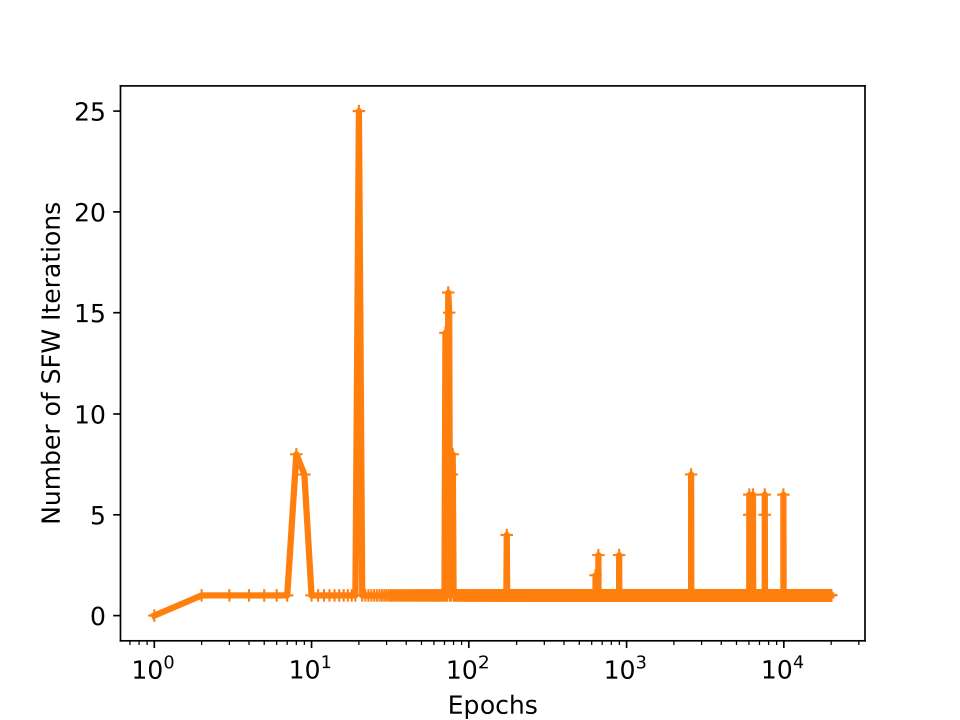}
        \caption{SFW iterations per projection.}
    \end{subfigure}

    \vspace{0.25cm}
    \begin{subfigure}{0.48\textwidth}
        \centering
        \includegraphics[width=\linewidth]{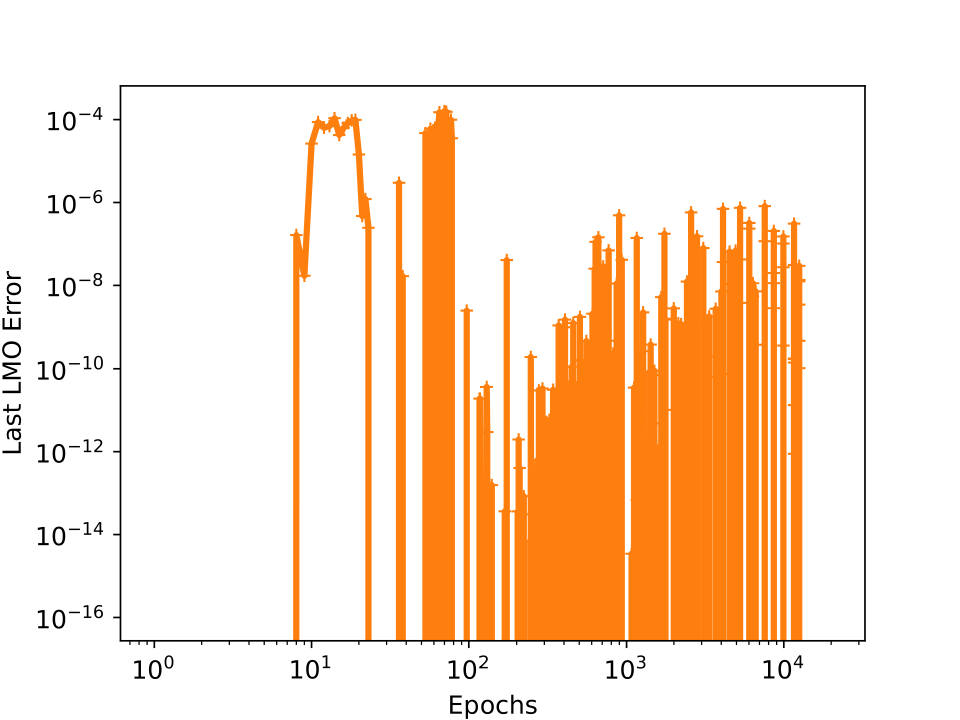}
        \caption{Final LMO error.}
    \end{subfigure}
    \hfill
    \begin{subfigure}{0.48\textwidth}
        \centering
        \includegraphics[width=\linewidth]{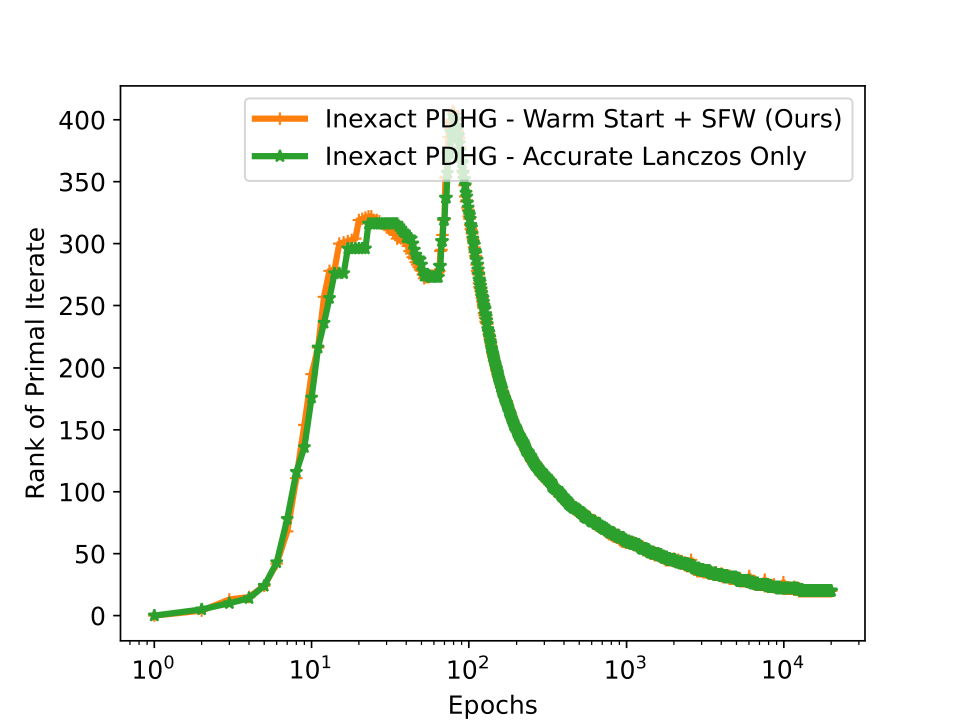}
        \caption{Rank of the primal iterate.}
    \end{subfigure}
    \caption{Projection diagnostics for warm-started SFW on G38. The inner
    iteration count adapts to the projection subproblem, while the LMO error
    and primal rank vary over the outer trajectory.}
    \label{fig:g38_projection_diagnostics}
\end{figure}

\begin{figure}[t]
    \centering
    \begin{subfigure}{0.48\textwidth}
        \centering
        \includegraphics[width=\linewidth]{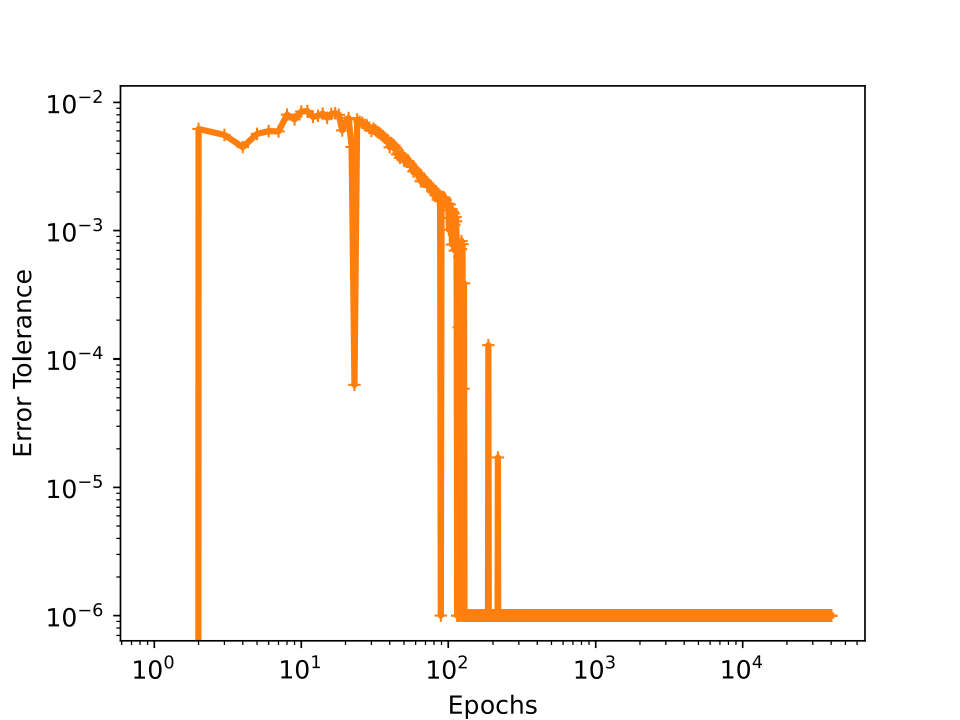}
        \caption{Projection error tolerance.}
    \end{subfigure}
    \hfill
    \begin{subfigure}{0.48\textwidth}
        \centering
        \includegraphics[width=\linewidth]{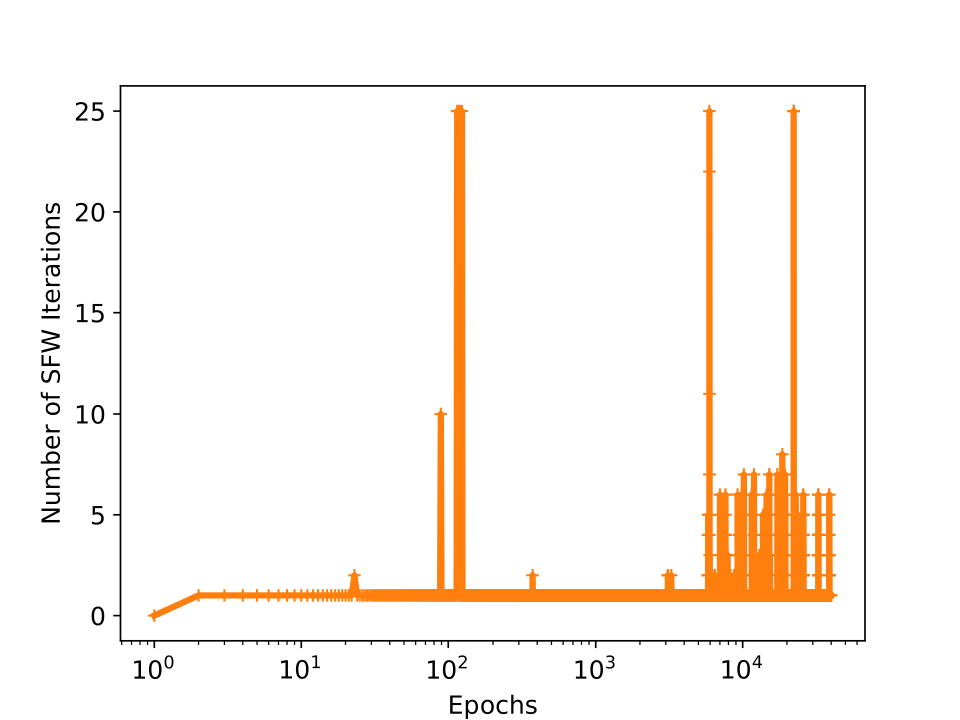}
        \caption{SFW iterations per projection.}
    \end{subfigure}

    \vspace{0.25cm}
    \begin{subfigure}{0.48\textwidth}
        \centering
        \includegraphics[width=\linewidth]{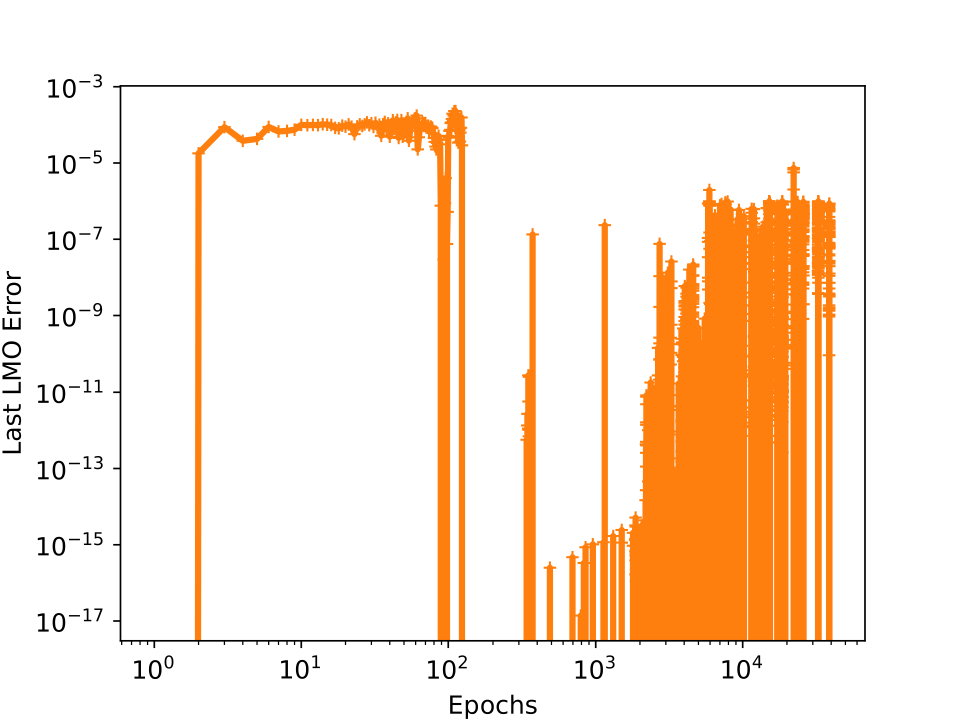}
        \caption{Final LMO error.}
    \end{subfigure}
    \hfill
    \begin{subfigure}{0.48\textwidth}
        \centering
        \includegraphics[width=\linewidth]{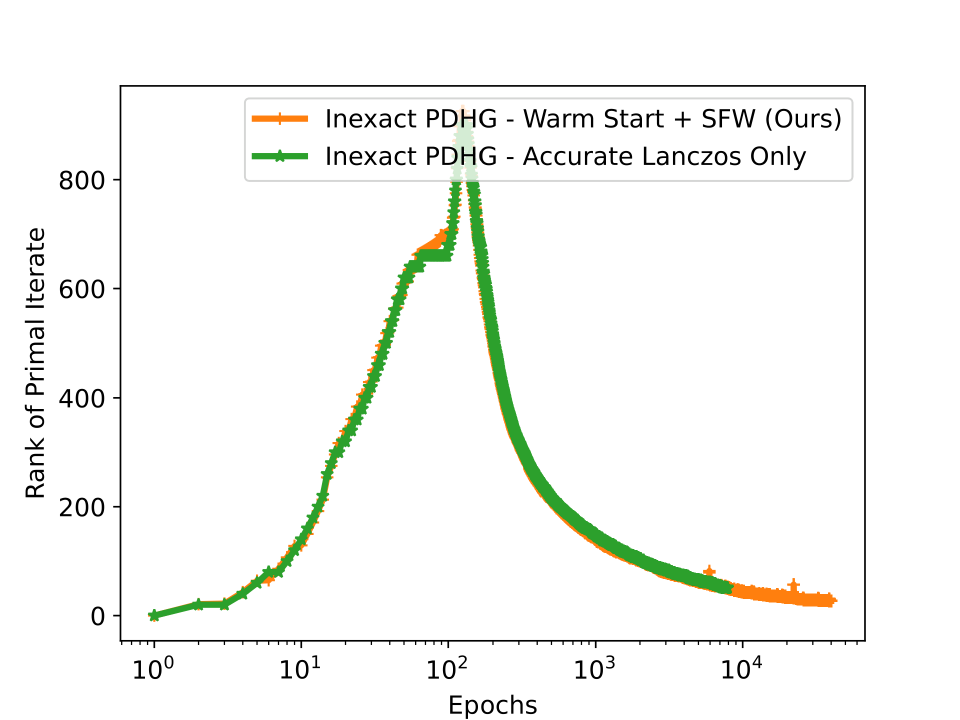}
        \caption{Rank of the primal iterate.}
    \end{subfigure}
    \caption{Projection diagnostics on G55. The rank of the primal iterate
    is shown for warm-started SFW and accurate Lanczos. Both trajectories
    exhibit a rapid rise followed by a sharp decline and largely overlap.}
    \label{fig:g55_projection_diagnostics}
\end{figure}

\begin{figure}[t]
    \centering
    \begin{subfigure}{0.48\textwidth}
        \centering
        \includegraphics[width=\linewidth]{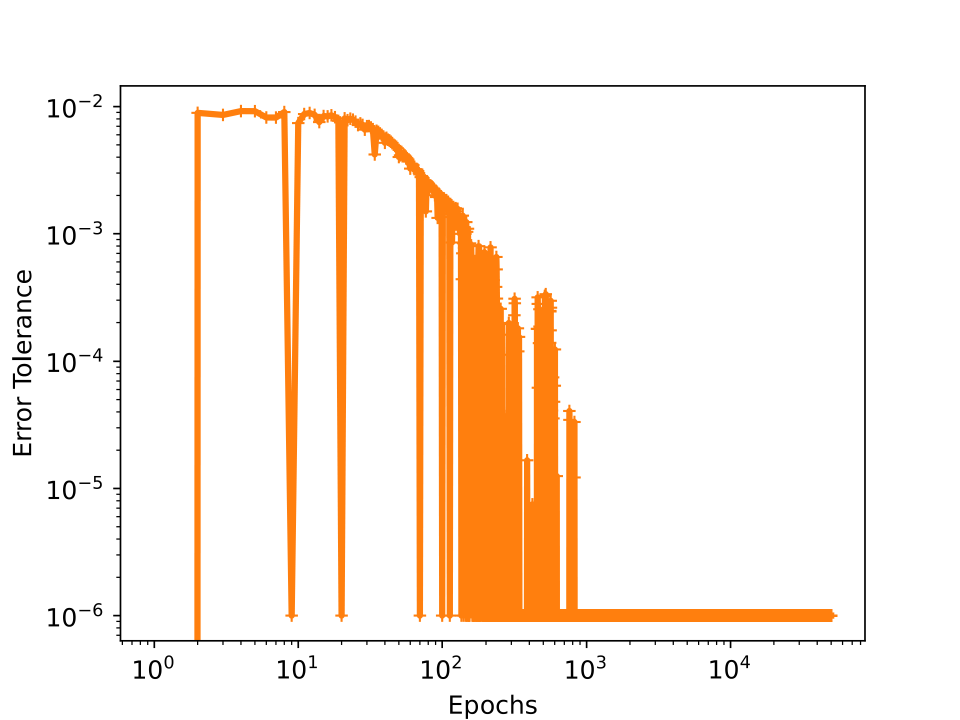}
        \caption{Projection error tolerance.}
    \end{subfigure}
    \hfill
    \begin{subfigure}{0.48\textwidth}
        \centering
        \includegraphics[width=\linewidth]{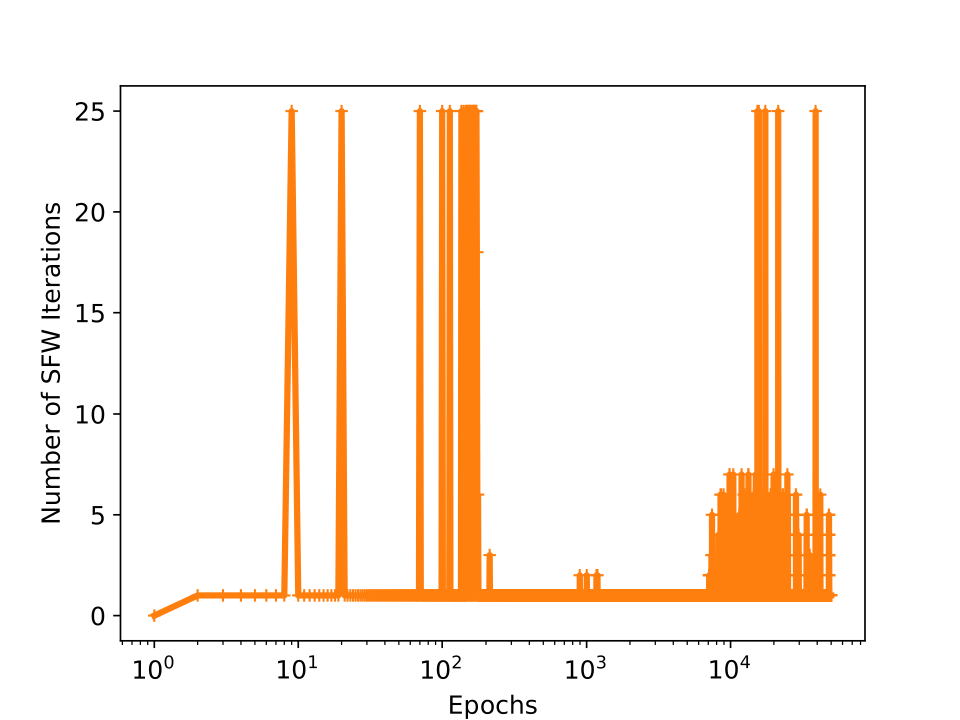}
        \caption{SFW iterations per projection.}
    \end{subfigure}

    \vspace{0.25cm}
    \begin{subfigure}{0.48\textwidth}
        \centering
        \includegraphics[width=\linewidth]{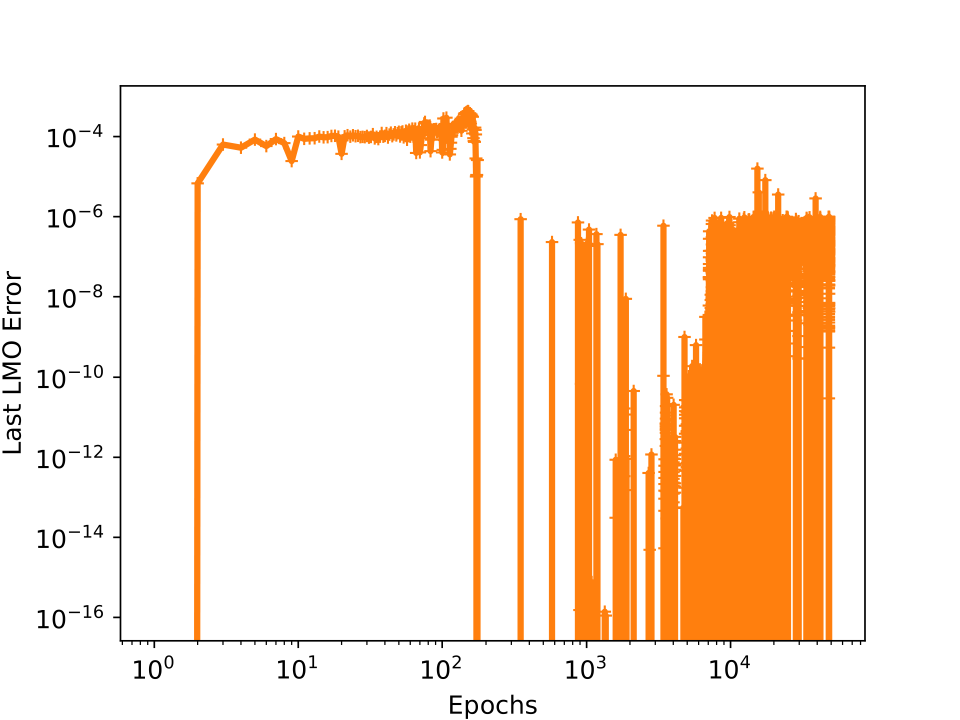}
        \caption{Final LMO error.}
    \end{subfigure}
    \hfill
    \begin{subfigure}{0.48\textwidth}
        \centering
        \includegraphics[width=\linewidth]{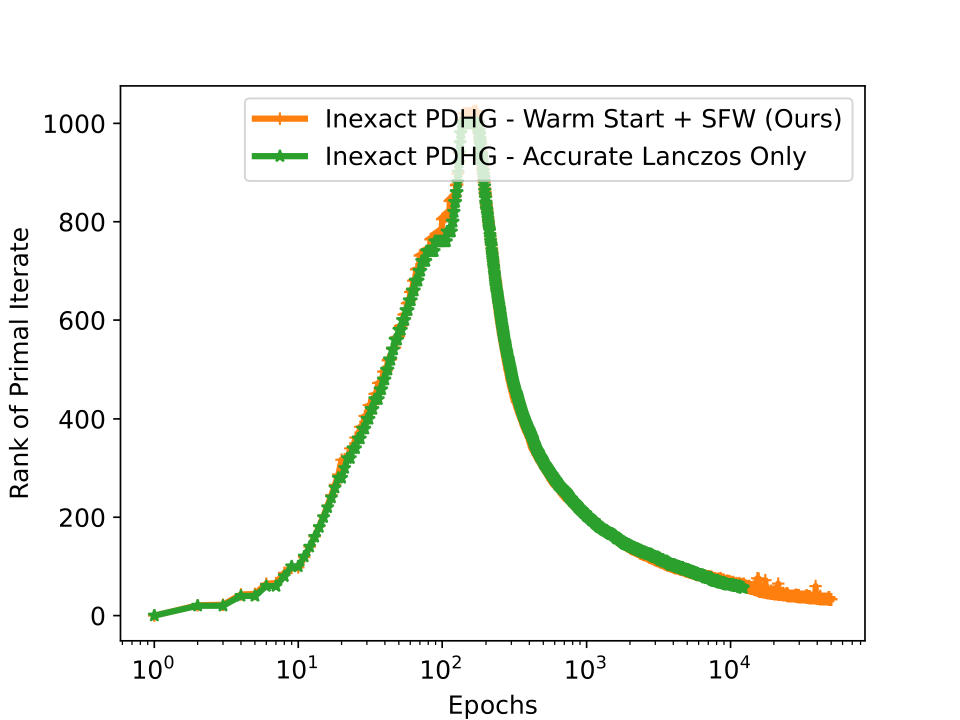}
        \caption{Rank of the primal iterate.}
    \end{subfigure}
    \caption{Projection diagnostics on G60. The rank of the primal iterate
    is shown for warm-started SFW and accurate Lanczos. As on G55, the rank
    trajectories have a similar rise-and-fall pattern and largely overlap.}
    \label{fig:g60_projection_diagnostics}
\end{figure}

\begin{figure}[t]
    \centering
    \begin{subfigure}{0.48\textwidth}
        \centering
        \includegraphics[width=\linewidth]{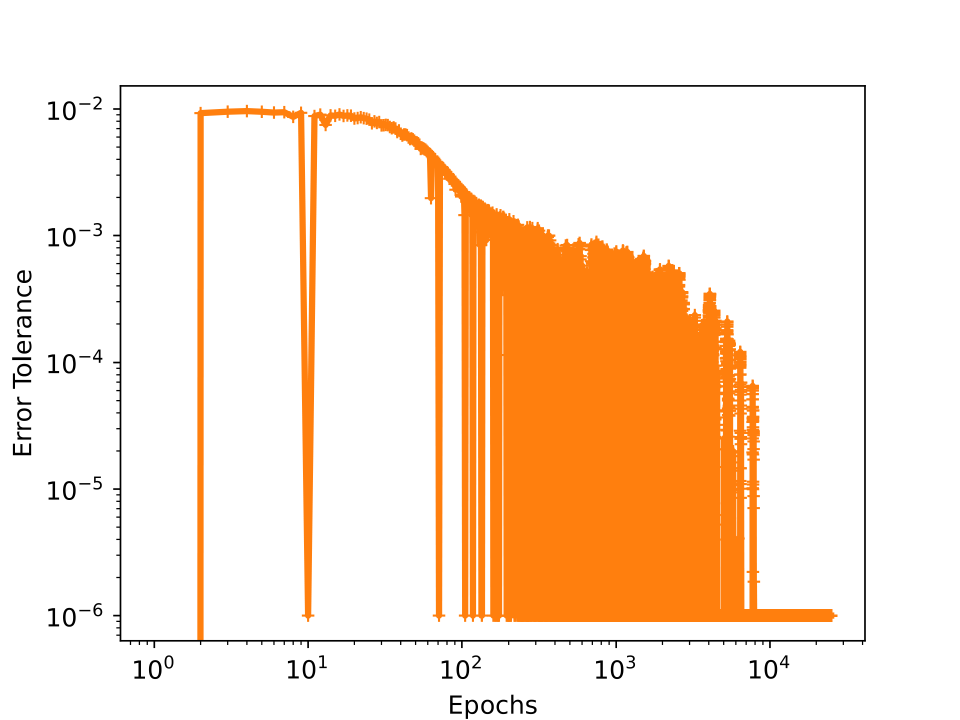}
        \caption{Projection error tolerance.}
    \end{subfigure}
    \hfill
    \begin{subfigure}{0.48\textwidth}
        \centering
        \includegraphics[width=\linewidth]{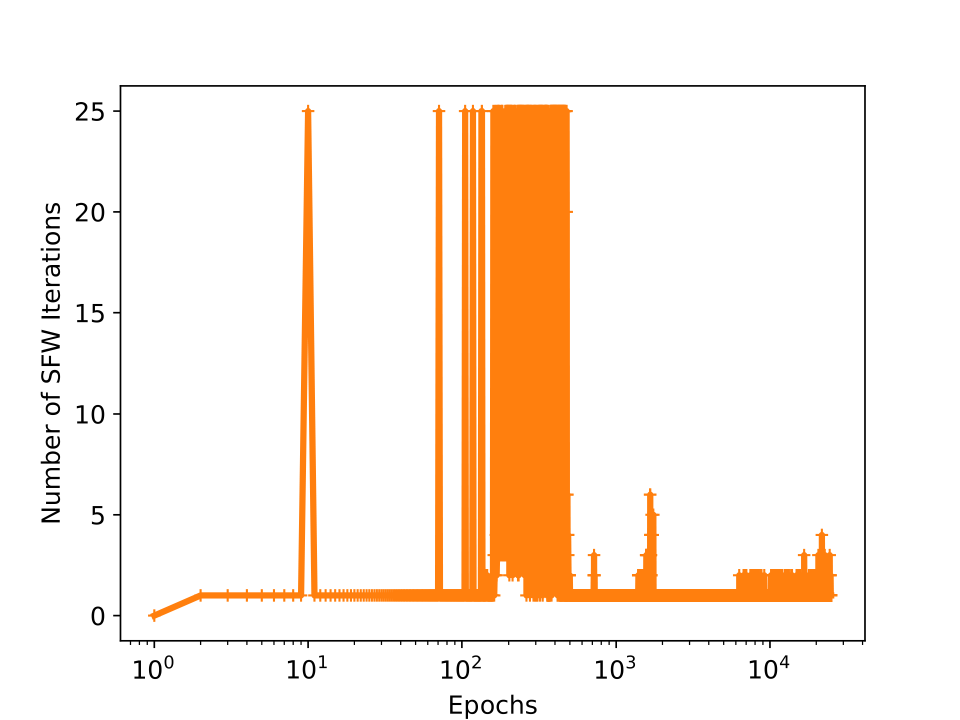}
        \caption{SFW iterations per projection.}
    \end{subfigure}

    \vspace{0.25cm}
    \begin{subfigure}{0.48\textwidth}
        \centering
        \includegraphics[width=\linewidth]{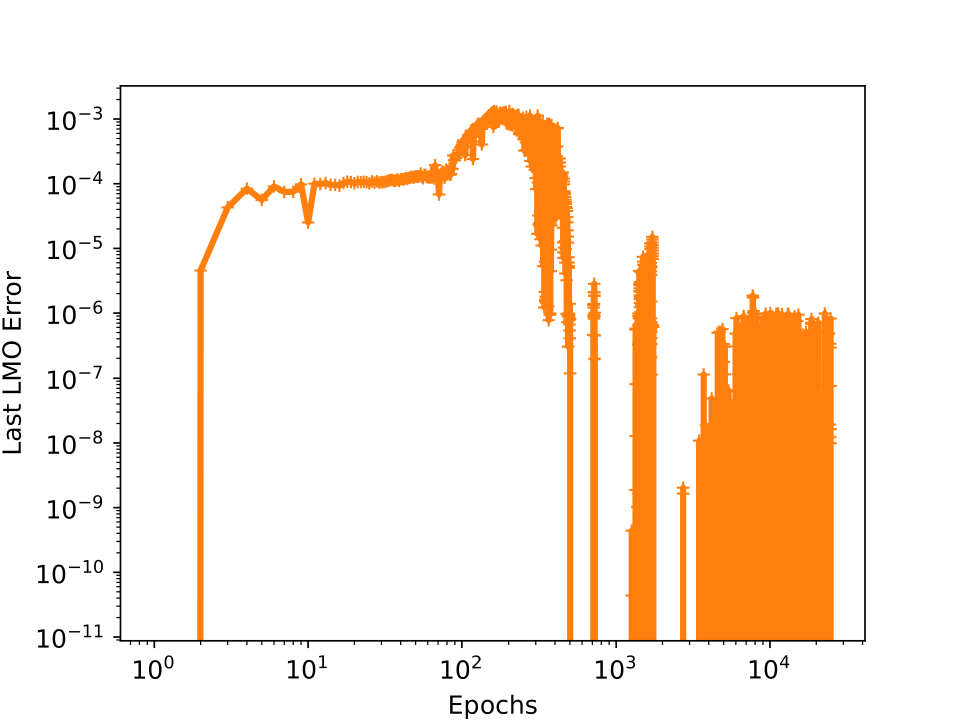}
        \caption{Final LMO error.}
    \end{subfigure}
    \hfill
    \begin{subfigure}{0.48\textwidth}
        \centering
        \includegraphics[width=\linewidth]{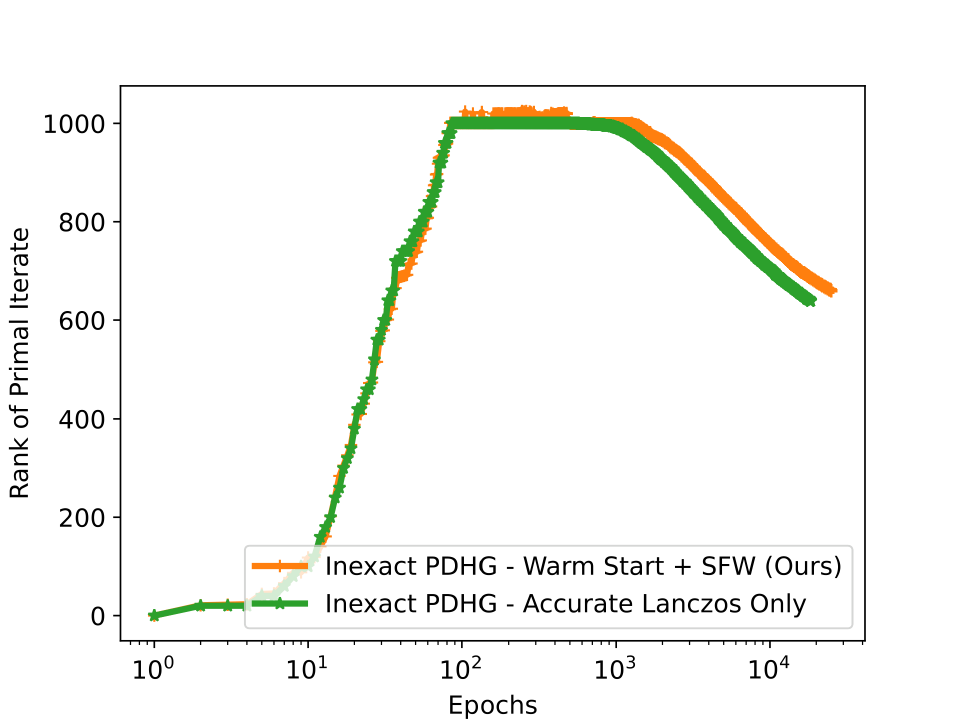}
        \caption{Rank of the primal iterate.}
    \end{subfigure}
    \caption{Projection diagnostics on G70. The rank of the primal iterate is continuing to decrease for G70 for both methods.}
    \label{fig:g70_projection_diagnostics}
\end{figure} %
\section{Conclusion}
\label{sec:conclusion}

We studied the computational trade-offs arising from inexact spectrahedral projections in primal-dual hybrid gradient (PDHG) methods for large-scale semidefinite optimization. 
Our approach combines adaptive Lanczos spectral warm starts with Spectrahedron Frank-Wolfe
(SFW), reusing spectral information from the preceding iterate and applying a small, controlled number of correction steps rather than solving each projection subproblem to high accuracy.

Experiments on the GSet MaxCut instances G38, G55, G60 and G70 demonstrate the practical benefits of this strategy. 
Compared with Inexact PDHG using an accurate Lanczos projection with eigensolver tolerance $10^{-8}$, the proposed warm-started SFW approach generally achieves more favorable objective-residual and feasibility-error reductions when computational work is measured by cumulative matrix-vector products. 
The improvement is particularly pronounced on G55 and G60,
where the accurate-Lanczos method makes less progress in the available
work budget, while warm-started SFW continues to reduce the objective residual. 
On G38, the two methods exhibit more similar convergence behavior, with warm-started SFW attaining its best observed objective residual with lower projection work. 
These results illustrate that the
accuracy required of an inner projection should be considered in the context of the outer optimization algorithm, rather than prescribed independently of it.

The rank trajectories provide additional insight into this behavior.
For G38, G55 and G60, both projection strategies exhibit rapid growth in the
rank of the primal iterate followed by a substantial decrease and eventual stabilization at much smaller ranks. 
For G70, the rank of the primal iterate exhibits similar behavior but the rank is still dropping at the end of the PDHG run.
Their trajectories are qualitatively similar, suggesting that the computational advantage of warm-started SFW does not require a fundamentally different outer
trajectory. 
Instead, it is consistent with exploiting spectral information across successive subproblems and adapting the amount of
inner work to the current projection.

Several questions remain open. 
A broader evaluation on semidefinite programs beyond MaxCut would help determine how these benefits depend on the structure of the constraint operator $\linop{A}$ and the spectral geometry of the projection subproblems. 
One must also characterize the work of determining the away step in Spectrahedron Frank-Wolfe, which amounts to solving a $r \times r$ linear system (see \Cref{app:away-step-atom-calc}) for a primal iterate of rank $r$, which can be small when the rank $r$ is much smaller than $n$ but could become more dominant as the rank grows or explodes in the middle of the PDHG run but for simplicity, this has been omitted out of this study.
A more systematic characterization of the relationship between spectral separation, projection accuracy,
and the required inner work could also guide principled choices of stopping criteria. 
Finally, incorporating computable a posteriori
diagnostics into the adaptive projection procedure may help balance inner accuracy and outer progress more reliably.

Overall, our results support adaptive, spectrally informed projection as a practical alternative to uniformly high-accuracy inner solves in large-scale first-order semidefinite optimization. 
The key is not to solve every projection as accurately as possible, but to allocate computational effort in a way that supports effective progress of the outer algorithm. \section*{Ethics, Reproducibility \& Acknowledgments}
\paragraph{Ethics statement}
This work adheres to the Umeå University Code of Ethics, ensuring responsible use of publicly available datasets and maintaining full transparency in experimental design and reporting. The study complies with ethical guidelines on research integrity and considers potential societal impacts. Since the contributions are primarily theoretical, the focus is on advancing our understanding of the theory of mathematical optimization rather than targeting specific applications that may have direct societal impacts. 

\paragraph{Reproducibility statement}
All datasets used in this work are referenced and publicly accessible.
The experimental configurations and computational environment are outlined in detail within the main text and the appendix.
The original code, accompanied by instructions, is included in the supplementary material.
We make our implementation available under {\href{https://github.com/KarthikPrakhya/InexactPDHG-SDP}{https://github.com/KarthikPrakhya/InexactPDHG-SDP}}.
Furthermore, all random seeds are provided to facilitate precise reproducibility.

\paragraph{Acknowledgments}
The authors would like to thank their advisor, Alp Yurtsever, for his guidance and support during the preparation of this manuscript. This work was supported by the Wallenberg AI, Autonomous Systems and Software Program (WASP) funded by the Knut and Alice Wallenberg Foundation. Computational resources were provided by the National Academic Infrastructure for Supercomputing in Sweden (NAISS), funded by the Swedish Research Council. This research was conducted using the resources of High Performance Computing Center North (HPC2N). 

\paragraph{Use of Generative AI} The authors used OpenAI’s ChatGPT and Anthropic's Claude for language editing and improving the clarity of the manuscript. All scientific content, mathematical results, and conclusions were developed and verified by the authors, who take full responsibility for the final manuscript.

\setlength{\bibsep}{6pt}
\bibliography{bibliography}
\bibliographystyle{bibliography}
\appendix
\section{Implementation Details for Inexact Projection onto a Spectrahedron}\label{app:sfw-impl-details}

In this section, we explain the procedure to scale and shift the projection subproblem prior to applying Spectrahedron Frank-Wolfe (see Algorithm  \ref{alg:spectrahedron-fw}).
In practice, we find that this helps speed up the convergence of Spectrahedron Frank-Wolfe by inducing a more favorable geometry for the projection subproblem. 
Consider the problem
\begin{equation}
\label{eq:weighted_projection_original}
\begin{aligned}
\min_{\mathbf{H} \in \set{S}_+^n}\quad
& \frac{1}{2}\|\mathbf{H}-\mathbf{G}\|_F^2\\
\text{s.t.}\quad
& \operatorname{Tr}(\mathbf{H})=R.
\end{aligned}
\end{equation}
This process of scaling and shifting generates a provably equivalent projection subproblem but this remains just a practical heuristic to improve the convergence of Spectrahedron Frank-Wolfe as we do not offer proof that it speeds up convergence and subsequent guarantees in \Cref{app:approx-eigvec-comp} only work with the formulation in \eqref{eq:weighted_projection_original}.

Let $\mathbf{H} = \mathbf{D}^{-1}\mathbf{Y} \mathbf{D}^{-1}$, where $\mathbf{D}$ is a positive diagonal matrix that acts as a diagonal preconditioner. Then
\begin{equation*}
    \operatorname{Tr}(\mathbf{H})
    =\operatorname{Tr}(\mathbf{D}^{-1}\mathbf{Y}\mathbf{D}^{-1})
    =\operatorname{Tr}(\mathbf{D}^{-2} \mathbf{Y})
    =R.
\end{equation*}
Prior to projection, a trace-based shift is applied to $\mathbf{G}$:
\begin{equation*}
    \mathbf{G} \rightarrow \mathbf{G} - \frac{\mathrm{Tr}(\mathbf{G})}{n}.
\end{equation*}
The trace-based shift above doesn't change the optimal projection but it has to be accounted for when computing the objective. 
As the details of reconstruction from this shift are straightforward, we omit the discussion for simplicity.
Then a diagonal preconditioner $\mathbf{D}$ is applied.
A reasonable choice that was used in the experiments is 
\begin{equation*}
    \mathbf{D} = \mathrm{Diag}\left(\frac{1}{\sqrt{\mathrm{diag}(\mathbf{G})} + \epsilon\mathbf{1}_n}\right ),
\end{equation*}
where $\epsilon > 0$ is a small additive constant added for stability, both the square root and reciprocal are applied element-wise, $\mathrm{diag}(\cdot)$ indicates diagonal of the input square matrix and $\mathrm{Diag}(\cdot)$ is a diagonal matrix whose diagonal is the input.

\subsection{Rescaled Objective for Projection Subproblem}

The objective then becomes
\begin{align*}
    \frac{1}{2}\|\mathbf{H}-\mathbf{G}\|_F^2
    &=\frac{1}{2}\|\mathbf{D}^{-1} \mathbf{Y}\mathbf{D}^{-1} - \mathbf{G}\|_F^2\\
    &=\frac{1}{2}\|\mathbf{D}^{-1}(\mathbf{Y} - \mathbf{D} \mathbf{G} \mathbf{D})\mathbf{D}^{-1}\|_F^2.
\end{align*}
Define $\widetilde{\mathbf{G}} = \mathbf{D}\mathbf{G}\mathbf{D}$.
Then the objective can be expressed as
\begin{align*}
    \phi(\mathbf{Y})
    &=\frac{1}{2}\|\mathbf{D}^{-1}(\mathbf{Y} - \widetilde{\mathbf{G}})\mathbf{D}^{-1}\|_F^2\\
    &=\frac{1}{2}
    \operatorname{Tr}\left[
    \mathbf{D}^{-1}(\mathbf{Y}-\widetilde{\mathbf{G}})^\top
    \mathbf{D}^{-2}(\mathbf{Y}-\widetilde{\mathbf{G}})\mathbf{D}^{-1}
    \right]\\
    &=\frac{1}{2}
    \operatorname{Tr}\left[
    (\mathbf{Y}-\widetilde{\mathbf{G}})^\top
    \mathbf{D}^{-2}(\mathbf{Y}-\widetilde{\mathbf{G}})\mathbf{D}^{-2}
    \right].
\end{align*}
When applying Spectrahedron Frank-Wolfe and the current rank of the projection iterate is $r$, we express the rescaled iterate in factorized form $\mathbf{Y} := R\mathbf{U}\mathrm{diag}(\bm{\alpha})\mathbf{U}^\top$ with $\mathbf{V} := \mathbf{U}\mathrm{diag}(\sqrt{\bm{\alpha}})$, where $\mathbf{U} \in \mathbb{R}^{n \times r}$, $\bm{\alpha} = (\alpha_1, \ldots, \alpha_r)$ and square-root is applied element-wise.
Assuming we consider only the terms that depend on $\mathbf{Y}$, this leads to an easy evaluation and update of the current objective if we rewrite in terms of $\mathbf{V}$, $\mathbf{U}$ and $\boldsymbol{\alpha}$ using the cyclical property of trace after expanding:
\begin{align*}
    \frac{1}{2}\left (\operatorname{Tr}\left [ (\mathbf{D}^{-1} \mathbf{Y} \mathbf{D}^{-1})^2\right ] - 2 \operatorname{Tr}\left[ \mathbf{D}^{-2} \widetilde{\mathbf{G}} \mathbf{D}^{-2}\mathbf{Y}\right ]\right ) = \frac{1}{2}\left (R^2 \operatorname{Tr}\left [ ((\mathbf{D}^{-1}\mathbf{V})^\top (\mathbf{D}^{-1}\mathbf{V}))^2 \right ] - 2 R \operatorname{Tr}\left[ \mathbf{D}^{-2} \widetilde{\mathbf{G}} \mathbf{D}^{-2} \mathbf{U}\mathrm{diag}(\alpha)\mathbf{U}^\top \right ]\right )
\end{align*}
We store and update $\mathbf{U}$, $\bm{\alpha}$ and $\widetilde{\mathbf{G}}\mathbf{U}$ with rank-1 updates for fast evaluation of the objective for each iteration of Spectrahedron Frank-Wolfe.

\subsection{Rescaled Gradient of Projection Subproblem}

Let $\mathbf{A} = \mathbf{Y} - \widetilde{\mathbf{G}}$, so that $d\mathbf{A} = d\mathbf{Y}$.
Then 
\begin{equation*}
    \phi(A)=\frac{1}{2}\operatorname{Tr}
    \left[\mathbf{A}^\top \mathbf{D}^{-2}\mathbf{A}\mathbf{D}^{-2}\right].
\end{equation*}
Differentiating gives
\begin{align*}
    d\phi
    =&\frac{1}{2}\operatorname{Tr}
    \left[(d\mathbf{A})^\top \mathbf{D}^{-2}\mathbf{A}\mathbf{D}^{-2}\right]\\
    &+\frac{1}{2}\operatorname{Tr}
    \left[\mathbf{A}^\top \mathbf{D}^{-2}(d\mathbf{A})\mathbf{D}^{-2}\right].
\end{align*}
Using cyclic invariance of the trace, the symmetry of $\mathbf{D}$,
and the Frobenius inner product, we get
\begin{align*}
d\phi
&=\frac{1}{2}\operatorname{Tr}
\left[(\mathbf{D}^{-2}\mathbf{A}\mathbf{D}^{-2})^\top d\mathbf{A}\right]\\
&\quad+\frac{1}{2}\operatorname{Tr}
\left[\mathbf{D}^{-2}\mathbf{A}^\top \mathbf{D}^{-2}d\mathbf{A}\right]\\
&=\operatorname{tr}
\left[(\mathbf{D}^{-2}\mathbf{A}\mathbf{D}^{-2})^\top d\mathbf{A}\right].
\end{align*}
Therefore, $\nabla \phi(\mathbf{A})=\mathbf{D}^{-2}\mathbf{A}\mathbf{D}^{-2}$ and hence
\begin{equation*}
\nabla \phi(\mathbf{Y})
=\mathbf{D}^{-2}(\mathbf{Y}-\widetilde{\mathbf{G}}) \mathbf{D}^{-2}.
\end{equation*}
If $\mathbf{Y} = R\mathbf{V}\mathbf{V}^\top$, this can also be written as
\begin{equation*}
    \nabla \phi(\mathbf{Y}) = R\mathbf{D}^{-2}\mathbf{V}\mathbf{V}^\top \mathbf{D}^{-2} -\mathbf{D}^{-2}\widetilde{\mathbf{G}} \mathbf{D}^{-2}.
\end{equation*}
With stored $\mathbf{U}$, one can easily construct an matrix-vector product operator for the gradient that can be used to determine the step direction for the forward step in Spectrahedron Frank-Wolfe using matrix-free iterative solvers like Lanczos.

\subsection{Linear minimization oracle}

The linear minimization oracle (LMO) for the rescaled projection
subproblem is
\begin{equation}
\label{eq:lmo}
\operatorname{LMO}(\mathbf{Y})
=
\argmax_{\substack{\mathbf{S} \in \set{S}_+^n\\
\operatorname{Tr}(\mathbf{D}^{-2}\mathbf{S})=R}}
\langle \nabla\phi(\mathbf{Y}),\mathbf{Y}-\mathbf{S}\rangle.
\end{equation}
Since the term involving $\mathbf{Y}$ is constant with respect
to $\mathbf{S}$, this is equivalent to minimizing
\begin{equation*}
\argmin_{\substack{\mathbf{S}\in \set{S}_+^n\\
\operatorname{Tr}(\mathbf{D}^{-2}\mathbf{S})=R}}
\langle \nabla\phi(\mathbf{Y}),\mathbf{S}\rangle.
\end{equation*}
Introduce the change of variables
$\mathbf{Z}=\mathbf{D}^{-1}\mathbf{S}\mathbf{D}^{-1}$.
Then
\[
\mathbf{S}=\mathbf{D}\mathbf{Z}\mathbf{D},
\qquad
\operatorname{Tr}(\mathbf{Z})=R,
\qquad
\mathbf{Z} \in \set{S}_+^n.
\]
Consequently, the minimization problem becomes
\begin{equation}
\label{eq:lmo_transformed}
    \argmin_{\substack{\mathbf{Z} \in \set{S}_+^n \\
    \operatorname{Tr}(\mathbf{Z})=R}}
    \langle \mathbf{D}\nabla\phi(\mathbf{Y})\mathbf{D},
    \mathbf{Z}\rangle.
\end{equation}
Let $\mathbf{s}$ be a unit eigenvector associated with the
smallest eigenvalue of
$\mathbf{D}\nabla\phi(\mathbf{Y})\mathbf{D}$.
An optimal solution of \eqref{eq:lmo_transformed} is $\mathbf{Z}^\star=R\mathbf{s}\mathbf{s}^\top$.
Transforming back gives
\begin{equation*}
\mathbf{S}^\star
=R(\mathbf{D}\mathbf{s})(\mathbf{D}\mathbf{s})^\top.
\end{equation*}
Thus, the LMO reduces to computing a smallest-eigenvalue
eigenvector of $\mathbf{D}\nabla\phi(\mathbf{Y})\mathbf{D}$.
Using the factorization
$\mathbf{Y}= R \mathbf{V}\mathbf{V}^\top$, the matrix-vector
product required by an iterative eigensolver can be evaluated
without explicitly forming the full gradient matrix.

\subsection{Expansion of the linear minimization oracle}

Using
\begin{equation*}
    \nabla\phi(\mathbf{Y})
    =\mathbf{D}^{-2}
    (\mathbf{Y}-\widetilde{\mathbf{G}})
    \mathbf{D}^{-2},
    \qquad
    \mathbf{Y}=R\mathbf{V}\mathbf{V}^\top,
\end{equation*}
and
\begin{equation*}
    \mathbf{S}^\star
    =R(\mathbf{D}\mathbf{s})(\mathbf{D}\mathbf{s})^\top,
\end{equation*}
we obtain
\begin{align*}
    &\langle\nabla\phi(\mathbf{Y}),
    \mathbf{Y}-\mathbf{S}^\star\rangle\\
    &=
    \langle
    \mathbf{D}^{-2}\mathbf{Y}\mathbf{D}^{-2},
    \mathbf{Y}\rangle
    -
    \langle
    \mathbf{D}^{-2}\widetilde{\mathbf{G}}\mathbf{D}^{-2},
    \mathbf{Y}\rangle\\
    &\quad
    -\langle
    \mathbf{D}^{-2}\mathbf{Y}\mathbf{D}^{-2},
    \mathbf{S}^\star
    \rangle + \langle
    \mathbf{D}^{-2}\widetilde{\mathbf{G}}\mathbf{D}^{-2},
    \mathbf{S}^\star
    \rangle.
\end{align*}
The first term satisfies
\begin{align*}
    &\langle
    \mathbf{D}^{-2}\mathbf{Y}\mathbf{D}^{-2},
    \mathbf{Y}\rangle \\
    &\quad= R^2 \langle
    \mathbf{D}^{-2}\mathbf{V}\mathbf{V}^\top\mathbf{D}^{-2},
    \mathbf{V}\mathbf{V}^\top\rangle\\
    &\quad=
    R^2 \operatorname{Tr}\left[
    (\mathbf{V}^\top\mathbf{D}^{-2}\mathbf{V})^2
    \right]\\
    &\quad=
    R^2 \operatorname{Tr}\left [ ((\mathbf{D}^{-1}\mathbf{V})^\top (\mathbf{D}^{-1}\mathbf{V}))^2 \right ].
\end{align*}
The terms involving the oracle atom can be expressed as
\begin{align*}
    &\langle
    \mathbf{D}^{-2}\mathbf{Y}\mathbf{D}^{-2},
    \mathbf{S}^\star
    \rangle \\
    &\quad= R^2\langle
    \mathbf{D}^{-2}\mathbf{V}\mathbf{V}^\top \mathbf{D}^{-2},
    (\mathbf{D}\mathbf{s})(\mathbf{D}\mathbf{s})^\top
    \rangle\\
    &\quad=
    R^2\mathbf{s}^\top
    \mathbf{D}^{-1}\mathbf{V}\mathbf{V}^\top
    \mathbf{D}^{-1}\mathbf{s}\\
    &\quad=
    R^2 \|\mathbf{V}^\top\mathbf{D}^{-1}\mathbf{s}\|_2^2,
\end{align*}
and
\begin{align*}
    &\langle
    \mathbf{D}^{-2}\widetilde{\mathbf{G}}\mathbf{D}^{-2},
    \mathbf{S}^\star
    \rangle \\
    &\quad=R \langle
    \mathbf{D}^{-2}\widetilde{\mathbf{G}}\mathbf{D}^{-2},
    (\mathbf{D}\mathbf{s})(\mathbf{D}\mathbf{s})^\top
    \rangle\\
    &\quad=
    R\mathbf{s}^\top
    \mathbf{D}^{-1}\widetilde{\mathbf{G}}\mathbf{D}^{-1}
    \mathbf{s}.
\end{align*}
Finally,
\begin{equation*}
    \langle
    \mathbf{D}^{-2}\widetilde{\mathbf{G}}\mathbf{D}^{-2},
    \mathbf{Y}\rangle
    =
    R\operatorname{Tr}\left[ \mathbf{D}^{-2} \widetilde{\mathbf{G}} \mathbf{D}^{-2} \mathbf{U}\mathrm{diag}(\alpha)\mathbf{U}^\top \right ].
\end{equation*}
Combining these expressions gives
\begin{align*}
    &\langle\nabla\phi(\mathbf{Y}),
    \mathbf{Y}-\mathbf{S}^\star\rangle\\
    &=
    R^2 \operatorname{Tr}\left [ ((\mathbf{D}^{-1}\mathbf{V})^\top (\mathbf{D}^{-1}\mathbf{V}))^2 \right ]\\
    &\quad-
    R\operatorname{Tr}\left[ \mathbf{D}^{-2} \widetilde{\mathbf{G}} \mathbf{D}^{-2} \mathbf{U}\mathrm{diag}(\alpha)\mathbf{U}^\top \right ]\\
    &\quad-
    R^2\|\mathbf{V}^\top\mathbf{D}^{-1}\mathbf{s}\|_2^2\\
    &\quad+
    R\mathbf{s}^\top\mathbf{D}^{-1}
    \widetilde{\mathbf{G}}\mathbf{D}^{-1}\mathbf{s}.
\end{align*}
The calculation of each of these four terms again only require the storage and update of $\mathbf{U}$, $\bm{\alpha}$ and $\widetilde{G}\mathbf{U}$ for every iteration of the Spectrahedron Frank-Wolfe algorithm.

\subsection{Away-Step Atom Calculation}\label{app:away-step-atom-calc}

Suppose the current iterate admits the factorization
\begin{equation*}
    \mathbf{Y}=R\mathbf{V}\mathbf{V}^\top,
    \qquad
    \mathbf{V}=\mathbf{U}\operatorname{diag}(\sqrt{\boldsymbol\alpha}).
\end{equation*}
For an away step, we seek the active rank-one atom that maximizes
the inner product with the gradient with respect to the image of the current projection iterate, which is spanned by columns of $\mathbf{V}$.
Writing an active atom as
$\mathbf{v}\mathbf{v}^\top$, its score is
\begin{equation*}
    \langle\nabla\phi(\mathbf{Y}),
    \mathbf{v}\mathbf{v}^\top\rangle
    =\mathbf{v}^\top\nabla\phi(\mathbf{Y})\mathbf{v}.
\end{equation*}
To express this selection in factor coordinates, let
$\mathbf{v}=\mathbf{V}\mathbf{c}$. The corresponding score is
\begin{equation*}
    \mathbf{c}^\top
    \mathbf{V}^\top\nabla\phi(\mathbf{Y})\mathbf{V}
    \mathbf{c}.
\end{equation*}
If the atom is normalized by
$\operatorname{Tr}(\mathbf{D}^{-2}\mathbf{v}\mathbf{v}^\top)=R$,
then the constraint becomes
\begin{equation*}
    \mathbf{c}^\top
    \mathbf{V}^\top\mathbf{D}^{-2}\mathbf{V}
    \mathbf{c}=R.
\end{equation*}
The resulting generalized Rayleigh quotient is
\begin{equation*}
    \max_{\mathbf{c}}
    \quad
    \mathbf{c}^\top
    \mathbf{V}^\top\nabla\phi(\mathbf{Y})\mathbf{V}
    \mathbf{c}
    \quad
    \text{s.t.}\quad
    \mathbf{c}^\top
    \mathbf{V}^\top\mathbf{D}^{-2}\mathbf{V}
    \mathbf{c}=R.
\end{equation*}
Introducing a Lagrange multiplier $\lambda$, the stationarity
condition is
\begin{equation*}
    \mathbf{V}^\top\nabla\phi(\mathbf{Y})\mathbf{V}\mathbf{c}
    =
    \lambda\,
    \mathbf{V}^\top\mathbf{D}^{-2}\mathbf{V}\mathbf{c}.
\end{equation*}
Substituting the gradient expression yields
\begin{align*}
    &\left[
    R\mathbf{V}^\top\mathbf{D}^{-2}
    \mathbf{V}\mathbf{V}^\top\mathbf{D}^{-2}\mathbf{V}
    -
    \mathbf{V}^\top\mathbf{D}^{-2}
    \widetilde{\mathbf{G}}\mathbf{D}^{-2}\mathbf{V}
    \right]\mathbf{c}\\
    &\qquad=
    \lambda\,
    \mathbf{V}^\top\mathbf{D}^{-2}\mathbf{V}\mathbf{c}.
\end{align*}
This gives a $r \times r$ generalized eigenvalue formulation for the
away-atom selection in factor coordinates.

\section{Randomized Lanczos and Approximate Eigenvector Computation} \label{app:approx-eigvec-comp}
For all the experiments, we use the thick-restart Lanczos method provided as part of ARPACK \citep{lehoucq1998arpack} as the default eigensolver.
Here, we develop careful analysis of the effect of errors in the estimation of the dominant eigenpair on the convergence of Spectrahedron Frank-Wolfe but we make a simplifying assumption that a simple Randomized Lanczos method \citep{kuczynski1992estimating} is used to compute an approximation of the eigenpair of the gradient $\nabla \phi(\mathbf{H}_k)$ corresponding to the smallest eigenvalue of $\nabla \phi(\mathbf{H}_k)$ (see \Cref{alg:approx_min_evec}).
Not having to compute an exact eigenpair reduces the computational overhead incurred by the projection subproblem in \approxpdhg.
\citep{jaggi2013revisiting} shows that the LMO can be solved inexactly and therefore the Frank-Wolfe algorithm allows for an inexact eigenvector computation with some upper bound on the error in the eigenvector computation. 
Here, we will show that the Frank-Wolfe algorithm and the Spectrahedron Frank-Wolfle algorithm bot converge to an inexact projection onto $\set{H}$ with desired level of precision with sufficiently high probability assuming that the number of iterations of the randomized Lanczos algorithm exceeds a certain threshold.
\begin{algorithm} 
\setstretch{1.25}
\caption{ApproxMinEvec via randomized storage-optimal Lanczos \citep{yurtsever2017sketchy}}
\label{alg:approx_min_evec}
\begin{algorithmic}[1]
\State \textbf{Input: } Input matrix $\nabla \phi(\mathbf{H}_k) \in \set{S}_n$ and maximum number $q$ of iterations
\State \textbf{Output: } Approximate minimum eigenpair $(\lambda, \vec{v}) \in \mathbb{R} \times \mathbb{R}^n$ of the matrix $\nabla \phi(\mathbf{H}_k)$
\Function{ApproxMinEvec}{$\nabla \phi(\mathbf{H}_k)$, $q$}
    \State Initialize random vector: $\vec{v}_1 = \operatorname{randn}(n, 1)$ %
    \State Normalize: $\vec{v}_1 = \vec{v}_1/\|\vec{v}_1\|$
    \For{$i = 1, 2, 3, \ldots, \min\{q, n-1\}$}
        \State Compute Rayleigh quotient: $\omega_i = \operatorname{Re}(\vec{v}_i^*(\nabla \phi(\mathbf{H}_k)\vec{v}_i))$
        \State Lanczos recurrence: $\vec{v}_{i+1} = \nabla \phi(\mathbf{H}_k)\vec{v}_i - \omega_i \vec{v}_i - \rho_{i-1} \vec{v}_{i-1}$
        \State Compute norm: $\rho_i = \|\vec{v}_i\|$
        \If{$\rho_i = 0$} 
            \State \textbf{break}
        \EndIf
        \State Normalize: $\vec{v}_{i+1} = \vec{v}_{i+1}/\rho_i$
    \EndFor
    \State Form tridiagonal matrix: $\mathbf{T} = \operatorname{tridiag}(\rho_{1:(i-1)}, \omega_{1:i}, \rho_{1:(i-1)})$
    \State Solve eigenproblem: $[\lambda, \vec{u}] = \operatorname{MinEvec}(\mathbf{T})$
    \State Reconstruct eigenvector: $\vec{v} = \sum_{j=1}^i u_j \vec{v}_j$
\EndFunction
\end{algorithmic}
\end{algorithm}

The following lemma from \citep{kuczynski1992estimating, yurtsever2017sketchy} shows that the Randomized Lanczos method achieve a certain level of precision $\frac{\xi}{8}\|\nabla \phi(\mathbf{H}_k)\|_F$ for some $\xi \in (0, 1]$  in estimating the smallest eigenvector with probability of $1 - 2\zeta$ after certain number of iterations that depends on both dimension $n$, $\xi$ and $\zeta$. 
\begin{lemma}[Randomized Lanczos method \citep{kuczynski1992estimating, yurtsever2017sketchy}]
\label{lemma:randomized-lanczos}
Let $\nabla \phi(\mathbf{H}_k) \in \set{S}_n$. For $\xi \in (0, 1]$ and $\zeta \in (0, 0.5]$, the randomized Lanczos method given by \Cref{alg:approx_min_evec} computes a unit vector $\vec{v} \in \mathbb{R}^n$ that satisfies
\begin{equation}\label{eqn:lanczos-accuracy-bound}
\vec{v}^* \nabla \phi(\mathbf{H}_k) \vec{v} \leq \lambda_{\min}(\nabla \phi(\mathbf{H}_k)) + \frac{\xi}{8}\|\nabla \phi(\mathbf{H}_k)\|_F \quad \text{with probability at least } 1 - 2\zeta
\end{equation}
after $q \geq \frac{1}{2} + \xi^{-1/2} \log(n/\zeta^2)$ iterations. The arithmetic cost is at most $q$ matrix-vector multiplies with $\nabla \phi(\mathbf{H}_k)$ and $\mathcal{O}(qn)$ extra operations. The working storage is $\mathcal{O}(n + q)$.
\end{lemma}
\begin{proof}
We refer to the proof of Theorem 4.2(a) in \citep{kuczynski1992estimating}.
\end{proof}
Given \Cref{lemma:randomized-lanczos}, the problem reduces to finding a suitable schedule $\{\xi_k\}_{k \geq 0}$ and suitable $\zeta$ such that we find the eigenvector with a controlled level of precision with constant probability such that the Frank-Wolfe algorithm or its variant converges to the approximate projection onto $\set{H}$. 
In order to do this, we need to first quantify what inexact computation of the LMO means and it can be done as follows.
Suppose that $\hat{\mathbf{S}}_k$ is the inexact LMO computed for the $k$th iteration in place of the exact LMO given by \eqref{eqn:lmo-step}:
\begin{equation}\label{eqn:lmo-step-inexact}
    \langle \nabla \phi(\mathbf{H}_k), \hat{\mathbf{S}}_k\rangle  = \min_{\mathbf{U} \in \set{H}}\langle \nabla \phi(\mathbf{H}_k), \mathbf{U}\rangle + \theta_k, 
\end{equation}
where $\theta_k > 0$ is the error in the LMO step for the $k$th iteration. 
Given that \eqref{eqn:lmo-step-inexact} is equivalent to the definition of inexact LMO call presented in \citep{jaggi2013revisiting} up to a scaling constant, we can now present and prove the following result that presents a suitable schedule $\{\xi_k\}_{k \geq 0}$:
\begin{theorem}\label{thm:lanczos-iter-fw} 
Suppose that the vanilla Frank-Wolfe algorithm (see \Cref{alg:approx_proj_fw}) is applied to the projection subproblem given in \eqref{eqn:proj-subproblem-H} and that $\set{H} = \set{S}_+^n \cap \{\mathbf{X}~:~\operatorname{Tr}(\mathbf{X}) = R \}$. 
For any prescribed overall failure probability $p \in (0,1)$ and any horizon $K \geq 1$, choose $\xi_k= \frac{c}{(k+1)^2}$ for some $c \in (0, 1]$ and run the randomized Lanczos method with per-call failure probability at most $p/K$ for each of the first $K$ iterations. Then, with probability of at least $1-p$,
the vanilla Frank-Wolfe algorithm will converge to the projection $\mathbf{H}^\star$ of $\mathbf{G}$ onto $\set{H}$ at rate $\mathcal{O}(1/k)$ for $1 \leq k \leq K$ after at least $\mathcal{O}\left((K + 1)\log\left(\frac{nK}{p}\right)\right )$ Lanczos iterations per Frank-Wolfe iteration.
\end{theorem}
\begin{proof} Let $h_k := \phi(\vec{H}_k) - \phi(\vec{H}^\star)$. Then following the proof of Theorem 2 in \citep{jaggi2013revisiting} and assuming an inexact LMO step as in \eqref{eqn:lmo-step-inexact}, we obtain the following recursive relationship:
\begin{equation*}
    h_{k + 1} \leq (1 - \eta_k)h_k + \eta_k \theta_k + \frac{\eta_k^2 C_\phi}{2}
\end{equation*}
where $C_\phi$ is the curvature constant corresponding to the objective in \eqref{eqn:proj-subproblem-H} and is equal to the squared Euclidean diameter of the set $\set{H}$, which is given by $2R^2$.
Unrolling this inequality with step size $\eta_k = \frac{2}{k + 2}$ and using Theorem 1 from \citep{jaggi2013revisiting} gives us that
\begin{equation*}
    h_k \leq \frac{4R^2}{k + 2} + \frac{2}{k(k + 1)} \sum_{i = 0}^{k-1} (i + 1)\theta_i.
\end{equation*}
Now suppose that an approximate FW step of $\hat{\mathbf{S}}_k := R\vec{v}\vec{v}^\top$ is taken with $\vec{v}$ being an approximate eigenvector.
Then we can derive an upper bound for $\theta_k$ as follows:

\begin{equation} \label{eqn:lanczos-max-iter-lem-proof-1}
    \langle \nabla \phi(\mathbf{H}_k), \hat{\mathbf{S}}_k\rangle  = \min_{\mathbf{U} \in \set{H}}\langle \nabla \phi(\mathbf{H}_k), \mathbf{U}\rangle + \theta_k \iff \theta_k = R(\omega - \lambda_{min}(\nabla \phi(\mathbf{H}_k))),
\end{equation}
where $\omega := \vec{v}^\top\nabla \phi(\mathbf{H}_k)\vec{v}$.
Given \eqref{eqn:lanczos-accuracy-bound} and \eqref{eqn:lanczos-max-iter-lem-proof-1}, we need to have that
\begin{equation}\label{eqn:lanczos-max-iter-lem-proof-2}
    \theta_k \leq \frac{R\xi_k \norm{\nabla \phi(\mathbf{H}_k)}_F}{8},
\end{equation}
where $\xi_k \in (0, 1]$.
Suppose we choose $\xi_k \sim \frac{c}{(k + 1)^2}$ with $c \in (0, 1]$, then $\sum_{k = 0}^\infty \xi_k < \infty$.
Next, we note that 
\begin{equation*}
    \norm{\nabla \phi(\mathbf{H}_k)}_F \leq \norm{\mathbf{H}_k - \mathbf{G}}_F \leq \norm{\mathbf{H}_k}_F + \norm{\mathbf{G}}_F \leq R + \norm{\mathbf{G}}_F,
\end{equation*}
so $\norm{\nabla \phi(\mathbf{H}_k)}_F$ is bounded.
Given \eqref{eqn:lanczos-max-iter-lem-proof-2} and the fact that $\norm{\nabla \phi(\mathbf{H}_k)}_F$ is bounded, we must have that $\theta_k \sim \mathcal{O}((k + 1)^{-2})$ and this in turn implies that $\frac{2}{k(k + 1)} \sum_{i = 0}^{k-1} (i + 1) \theta_i \sim \mathcal{O}\left (\frac{\log(k + 1)}{k(k + 1)} \right )$ since $\sum_{i = 0}^{k - 1} (i + 1)\theta_i \leq C\sum_{i = 0}^{k - 1}\frac{1}{i + 1}  = \mathcal{O}(\log(k + 1))$ for some $C > 0$ that is independent of $k$.

Now to prove the lower bound on the maximum number of Lanczos iterations, we simply invoke \Cref{lemma:randomized-lanczos} as follows.
For the $k$th iteration out of $K$ iterations, assume that the probability of failing to find a suitable approximation of the eigenvector corresponding to the smallest eigenvalue $\lambda_{min}(\nabla \phi(\mathbf{H}_k))$ is $p_k$.
Then the probability that any one of the $K$ iterations fails to find the minimum eigenvector is given by $\sum_{k = 1}^K p_k$.
To ensure that the overall probability of failure is $p$, we can enforce that $p_k = p/K$.
Given the objective of estimating the eigenvector with constant probability of $1 - 2\zeta$, \eqref{eqn:lanczos-accuracy-bound} entails that $p_k = 2\zeta$. 
Therefore, we must have that $\zeta = \frac{p}{2K}$.
Given this and the fact that $\xi_k = \frac{c}{(k + 1)^2}$, we can say that 
\begin{equation*}
    q_k \geq \mathcal{O}\left( \left ( \frac{c}{(k + 1)^2} \right )^{-1/2}\log\left ( \frac{n}{(p/(2K))^2}\right ) \right ) = \mathcal{O}\left ((k + 1)\log\left(\frac{nK}{p}\right ) \right ).
\end{equation*}
Consequently, the worst-case total number of Lanczos iterations per Frank-Wolfe iteration is bounded below by $\mathcal{O}\left ((K + 1)\log\left(\frac{nK}{p}\right ) \right )$. This concludes the proof.
\end{proof}
Now we try to detail the effect of an inexact LMO in the sense of \eqref{eqn:lmo-step-inexact} for Spectrahedron FW (see \Cref{alg:spectrahedron-fw}).
For sake of brevity in the proofs, we will follow the same notation as \citep{garber2026randomized} and refer to objective residual at iteration $k$ of \Cref{alg:spectrahedron-fw} as $h_k := \phi(\mathbf{H}_k) - \phi(\mathbf{H}^\star)$.
We also define $\delta_k := \lambda_{n - r^\star}(\nabla \phi(\mathbf{H}_k)) - \lambda_{n - r^\star + 1}(\nabla \phi(\mathbf{H}_k))$ as the eigengap at $\nabla \phi(\mathbf{H}_k)$ for $r^* < n$ and $\delta_{k, r^\star} := \lambda_{n - r^\star + 1}(\nabla \phi(\mathbf{H}_k)) - \lambda_{n}(\nabla \phi(\mathbf{H}_k))$ as the in-face eigengap at $\nabla\phi(\mathbf{H}_k)$.
We denote the projection operator onto span of the eigenvectors corresponding to the smallest $r^\star$ eigenvalues of $\nabla \phi(\mathbf{H}_k)$ as $\mathbf{P}_k$ and  $\mathbf{P}_k^\bot$ as the projection onto the span of the eigenvectors corresponding to $n - r^\star$ largest eigenvalues of $\nabla \phi(\mathbf{H}_k)$.
Finally, we define $r_k := \operatorname{rank}(\mathbf{H}_k)$.
First, we present the Poincar\'{e} separation theorem and the Weyl's inequality which are useful in proving the desired results:
\begin{lemma}[Poincar\'{e} Separation Theorem \citep{magnus2019matrix}] \label{lem:poincare-sep-thm} Let $\mathbf{A} \in \set{S}_n$ with eigenvalues ordered as follows: $\lambda_1(\mathbf{A}) \geq \ldots \geq \lambda_n(\mathbf{A})$.
Let $\mathbf{U} \in \mathbb{R}^{n \times r}$ be a semi-orthogonal matrix (i.e. $\mathbf{U}^\top \mathbf{U} = \mathbf{I}_{r \times r}$.
Similarly, denote the eigenvalues of $\mathbf{U}^\top \mathbf{A}\mathbf{U}$ by $\mu_1 \geq \ldots \geq \mu_r$.
Then the following holds:
\begin{equation*}
    \lambda_i(\mathbf{A}) \geq \mu_i \geq \lambda_{n - r + i}(\mathbf{A})
\end{equation*}
for all $i = 1, \ldots, r$.
\end{lemma}
\begin{lemma}[Weyl's Inequality \citep{franklin2000matrix}] \label{lem:weyl-thm} Let $\mathbf{A}, \mathbf{B} \in \set{S}_n$, then
\begin{equation*}
    |\lambda_k(\mathbf{A}) - \lambda_k(\mathbf{B})| \leq \|\mathbf{A} - \mathbf{B}\|_{op} \leq \|\mathbf{A} - \mathbf{B}\|_F,
\end{equation*}
where $\|\cdot\|_{op}$ denotes the operator norm and $k \in \{1, \ldots, n\}$.
\end{lemma}
We will now state the unchanged and relevant results from \citep{garber2026randomized} without proof.
The calculation of $\mathbf{v}_{t,+}$ and $\mathbf{u}_{k,+}$ involves inexact computation of the leading eigenvector of $-\nabla \phi(\mathbf{H}_k)$ and therefore only the forward and pairwise steps are affected by this inexact evaluation.
In contrast, the drop and away steps are computed by solving a small dense linear system as detailed in \Cref{app:away-step-atom-calc}, so we simply cite the relevant results from \citep{garber2026randomized} for these steps.
\begin{lemma}[Lemma 2, \citep{garber2026randomized}]\label{lem:permissible-step-lem-garber} Let $\mathbf{H} \in \mathbb{S}_+^n$ and $\mathbf{v} \in \mathrm{Im}(\mathbf{H})$ such that $\mathbf{v} \neq \mathbf{0}$.
Let $\mathbf{Y} = \mathbf{H} - \gamma \mathbf{v}\mathbf{v}^\top$ for some $\gamma \geq 0$. 
If $\gamma \leq (\mathbf{v}^\top \mathbf{H}^\dagger \mathbf{v})^{-1}$, then $\mathbf{Y} \in \mathbb{S}_+^n$.
If $\gamma = (\mathbf{v}^\top \mathbf{H}^\dagger \mathbf{v})^{-1}$, then $\mathrm{rank}(\mathbf{Y}) = \mathrm{rank}(\mathbf{H}) - 1$ and $\mathbf{H}^\dagger \mathbf{v} \in \mathrm{ker}(\mathbf{Y})$.
\end{lemma}

\begin{lemma}[Drop Step of SFW, Lemma 3, \cite{garber2026randomized}] \label{lem:sfw-lemma-drop-step} Consider \Cref{alg:spectrahedron-fw} applied to the projection subproblem defined in \eqref{eqn:proj-subproblem-H} with $\set{H} = \set{S}_+^n \cap \{\mathbf{X}~:~\operatorname{Tr}(\mathbf{X}) = R \}$.
For any iteration $k$ of \Cref{alg:spectrahedron-fw} such that $h_k <  \min \left \{ \frac{\delta_k^2}{18}, \frac{R\delta_k}{6}\right \}$ and $\operatorname{rank}(\mathbf{H}_k) > r^\star$, we have that $\phi(\mathbf{H}_{k+1}^{\text{drop}}) < \phi(\mathbf{H}_k)$.
\end{lemma}
\begin{proof}
The proof is similar to Lemma 3 in \citep{garber2026randomized} except that $\mathrm{Tr}(\mathbf{H}_k) = R$, which leads to the condition $h_k <  \min \left \{ \frac{\delta_k^2}{18}, \frac{R\delta_k}{6}\right \}$ instead.
\end{proof}
\begin{lemma}[Pairwise step of SFW, Lemma 4, \cite{garber2026randomized}] \label{lem:sfw-lemma-pfw-step} 
Consider \Cref{alg:spectrahedron-fw} applied to the projection subproblem defined in \eqref{eqn:proj-subproblem-H} with $\set{H} = \set{S}_+^n \cap \{\mathbf{X}~:~\operatorname{Tr}(\mathbf{X}) = R \}$.
For any iteration $k$ of \Cref{alg:spectrahedron-fw} such that $h_k < \min\left \{\frac{\delta_k^3}{24R}, \frac{\delta_k^2}{12} \right \}$, we have that
\begin{equation*}
    \phi(\mathbf{H}_{k+1}^{PFW}) \leq \phi(\mathbf{H}_k) - \min\left \{\frac{\delta_k^2}{9R}, \frac{\gamma_k\delta_k}{3}\right \}\|\mathbf{P}_k^\perp \vec{u}_{k,-}\|_2^2.
\end{equation*}
\end{lemma}
\begin{proof}
First decompose  $\mathbf{u}_{k,-}$ as
\begin{equation*}
    \mathbf{u}_{k,-} = \frac{\tau}{\sqrt{R}} \mathbf{z} + \sqrt{\frac{R-\tau^2}{R}}\mathbf{w},
\end{equation*}
where $\tau = \|\mathbf{P}_k^\perp \mathbf{u}_{k,-}\|$, $\mathbf{z} = \sqrt{R}\frac{\mathbf{P}_k^\perp \mathbf{u}_{k,-}}{\|\mathbf{P}_t^\perp \mathbf{u}_{k,-}\|}$ and $\mathbf{w} = \sqrt{R}\frac{\mathbf{P}_k \mathbf{u}_{k,-}}{\|\mathbf{P}_t \mathbf{u}_{k,-}\|}$.
Using the definition of $\mathbf{P}_k$ and $\mathbf{P}_k^\perp$, it follows from the Poincar\'{e} theorem (\Cref{lem:poincare-sep-thm}) that $\mathbf{z}^\top \nabla \phi(\mathbf{H}_k) \mathbf{z} \geq R\lambda_{n-r^\star}(\nabla \phi(\mathbf{H}_k))$ and $R\lambda_{n-r^\star+1}(\nabla \phi(\mathbf{H}_k)) \geq \mathbf{w}^\top \nabla \phi(\mathbf{H}_k) \mathbf{w}$ so that
\begin{equation}\label{eqn:pfw-proof-poincare-ineq}
    \mathbf{z}^\top \nabla \phi(\mathbf{H}_k) \mathbf{z}
    - \mathbf{w}^\top \nabla \phi(\mathbf{H}_k) \mathbf{w}
    \geq
    R\lambda_{n-r^\star}(\nabla \phi(\mathbf{H}_k))
    - R\lambda_{n-r^\star+1}(\nabla \phi(\mathbf{H}_k))
    = R\delta_k.
\end{equation}
Now define another unit vector scaled by $R$:
\begin{equation*}
    \mathbf{u}
    = \frac{\tau'}{\sqrt{R}} \mathbf{z} + \sqrt{\frac{R -\tau'^2}{R}}\mathbf{w}
\end{equation*}
with $0 \leq \tau' \leq \tau$. Fix some step-size $\eta \in [0,1]$
and define a new iterate $\mathbf{Y}_\eta$ as follows:
\begin{equation*}
    \mathbf{Y}_{\eta}
    = \mathbf{H}_k
    + \eta\left(
        \mathbf{u}\mathbf{u}^\top
        - \mathbf{u}_{k,-}\mathbf{u}_{k,-}^\top
    \right).
\end{equation*}
Using the $1$-smoothness of the projection objective $\phi$, it follows that
\begin{align}\label{eqn:pfw-proof-smoothness}
    \phi(\mathbf{Y}_{\eta})
    &\leq \phi(\mathbf{H}_k)
    + \eta
    \left\langle
        \mathbf{u}\mathbf{u}^\top
        - \mathbf{u}_{k,-}\mathbf{u}_{k,-}^\top,
        \nabla_t
    \right\rangle
    + \frac{\eta^2}{2}
    \left\|
        \mathbf{u}\mathbf{u}^\top
        - \mathbf{u}_{k,-}\mathbf{u}_{k,-}^\top
    \right\|_F^2
    \nonumber \\
    &= \phi(\mathbf{H}_k)
    + \eta\left(
        (\tau'^2-\tau^2)\mathbf{z}^\top\nabla \phi(\mathbf{H}_k)\mathbf{z}
        -(\tau'^2-\tau^2)\mathbf{w}^\top\nabla \phi(\mathbf{H}_k)\mathbf{w}
    \right)
    + \eta^2
    \left(R^2 -(\mathbf{u}^\top\mathbf{u}_{k,-})^2\right)
    \nonumber \\
    &\overset{(a)}\leq
    \phi(\mathbf{H}_k)
    - \eta (\tau^2-\tau'^2)\delta_k
    + \eta^2
    \left(
        R^2 -
        \left(
            \tau\tau'
            +\sqrt{R-\tau^2}\sqrt{R-\tau'^2}
        \right)^2
    \right)
    \nonumber \\
    &= \phi(\mathbf{H}_k)
    - \eta (\tau^2-\tau'^2)\delta_k \nonumber \\
    &\quad + \eta^2 
    \left(
        R^2 -
        \left(
            \tau^2 \tau'^2
            + (R^2 - R\tau^2-R\tau'^2+\tau^2\tau'^2)
            + 2\tau\tau'\sqrt{R-\tau^2}\sqrt{R-\tau'^2}
        \right)
    \right) \nonumber \\
    &= \phi(\mathbf{H}_k)
    - \eta (\tau^2-\tau'^2)\delta_k
    + \eta^2
    \left(
        R\tau^2+R\tau'^2-2\tau^2\tau'^2
        -2\tau\tau'\sqrt{R-\tau^2}\sqrt{R-\tau'^2}
    \right) \nonumber \\
    &\overset{(b)}{\leq}
    \phi(\mathbf{H}_k)
    - \eta (\tau^2-\tau'^2)\delta_k
    + \eta^2
    \left(
        R\tau^2+R\tau'^2-2\tau^2\tau'^2
        -2\tau\tau'(R-\tau^2)
    \right) \nonumber \\
    &= \phi(\mathbf{H}_k)
    - \eta (\tau^2-\tau'^2)\delta_k
    + \eta^2 R (\tau-\tau')^2
    + 2\eta^2\tau^2\tau'(\tau-\tau') \nonumber\\
    &= \phi(\mathbf{H}_k)
    - \eta (\tau-\tau')
    \left(
        (\tau+\tau')\delta_k
        -\eta\left(R(\tau-\tau')+ 2\tau^2\tau' \right)
    \right),
\end{align}
where (a) is due to \eqref{eqn:pfw-proof-poincare-ineq}
and (b) follows from the fact that $0 \leq \tau'\leq\tau$.
First note that setting $\tau'=0$ in (14) yields
\begin{equation*}
    \phi(\mathbf{Y}_{\eta})
    \leq \phi(\mathbf{H}_k)
    -\eta \tau^2(\delta_k-R\eta).
\end{equation*}
Using the step-size
\begin{equation*}
    \eta^\star=\min\left\{
        \frac{\delta_k}{2R},\,\gamma_k
    \right\},
\end{equation*}
that satisfies that $\mathbf{Y}_{\eta^\star} \in \set{S}_+^n$ as per \Cref{lem:permissible-step-lem-garber} since $\gamma_k = (\mathbf{u}_{k,-}^\top \mathbf{H}_k^\dagger \mathbf{u}_{k,-})^{-1}$, then gives
\begin{equation*}
    \phi(\mathbf{Y}_{\eta^\star})-\phi(\mathbf{H}^\star)
    \leq h_k
    -\frac{
        \min\left\{
            \frac{\delta_k}{2R},\,\gamma_k
        \right\}
        R\delta_k\tau^2
    }{2},
\end{equation*}
which implies that
\begin{equation}\label{eqn:pfw-proof-tau-sq-bound}
    \tau^2
    \leq
    \frac{2h_k}{
        \min\left\{
            \frac{\delta_k}{2},\,\gamma_k
        \right\}\delta_k
    }.
\end{equation}
Consider now the case $\tau'=(1-c)\tau$ for some scalar
$c\in[0,1]$ to be specified in the sequel. 
Using the $1$-smoothness of $\phi$ again, the definition of $\mathbf{u}_{k,+}$ in \Cref{alg:spectrahedron-fw}, and applying \eqref{eqn:pfw-proof-smoothness} with $\eta=\gamma_k$
(note that by definition $\gamma_k\in(0,1]$), we have that
\begin{align*}
    \phi(\mathbf{H}_{t+1}^{\mathrm{PFW}})
    &\leq \phi(\mathbf{H}_k)
    +\gamma_k
    \left\langle
        \mathbf{u}_{k,+}\mathbf{u}_{k,+}^{\top}
        -\mathbf{u}_{k,-}\mathbf{u}_{k,-}^{\top},
        \nabla \phi (\mathbf{H}_k)
    \right\rangle
    +\frac{\gamma_k^2}{2}
    \left\|
        \mathbf{u}_{k,+}\mathbf{u}_{k,+}^{\top}
        -\mathbf{u}_{k,-}\mathbf{u}_{k,-}^{\top}
    \right\|_F^2\\
    &\leq \phi(\mathbf{H}_k)
    -\gamma_k c\tau^2
    \left(
        \delta_k-\gamma_k \left (Rc+ 2\tau^2 \right )
    \right).
\end{align*}
Since $\mathbf{u}$ and $c \in [0, 1]$ are both arbitrary at this stage as long as $\mathbf{u}$ is chosen so that $\tau^\prime \leq \tau$, this bound is not affected by inexact computation of $\mathbf{u}_{k,+}$ as long as $\eta = \gamma_k$. 
Now set
\begin{equation*}
    c=\min\left\{\frac{\delta_k}{3\gamma_k R},1\right\}
\end{equation*}
in to obtain
\begin{align*}
    \phi(\mathbf{X}_{k+1}^{\mathrm{PFW}})
    &\leq \phi(\mathbf{X}_k)
    -\gamma_k
    \min\left\{\frac{\delta_k}{3\gamma_k},1\right\}
    \tau^2
    \left(\frac{2\delta_k}{3}-\frac{2\gamma_k\tau^2}{R}\right)\\
    &\leq \phi(\mathbf{H}_k)
    -\gamma_k
    \min\left\{\frac{\delta_k}{3\gamma_k},1\right\}
    \tau^2
    \left(
        \frac{2\delta_k}{3}
        -\frac{4\gamma_k h_k}{
            \min\left\{\frac{\delta_k}{2},\gamma_k\right\}
            \delta_k R^2
        }
    \right),
\end{align*}
where the last inequality follows from substituting \eqref{eqn:pfw-proof-tau-sq-bound} into the term
$2\gamma_k\tau^2$.
Given that $\tau^2=\|\mathbf{P}_k^\perp\mathbf{u}_{k,-}\|^2$
and that $\gamma_k\in(0,1]$, it holds that
\begin{equation*}
    h_k
    \leq \frac{\delta_k^2}{12}
    \min\left\{\frac{\delta_k}{2R},1\right\}
    \leq \frac{\delta_k^2}{12\gamma_k}
    \min\left\{\frac{\delta_k}{2R},\gamma_k\right\},
\end{equation*}
implies that
\begin{equation*}
    \phi(\mathbf{H}_{t+1}^{\mathrm{PFW}})
    \leq \phi(\mathbf{H}_k)
    - \frac{\gamma_k\delta_k}{3}
    \min\left\{\frac{\delta_k}{3\gamma_k R},1\right\}
    \|\mathbf{P}_k^\perp\mathbf{u}_{k,-}\|^2.
\end{equation*}
\end{proof}
\begin{lemma}[Lemma 5, \cite{garber2026randomized}] \label{lem:sfw-trace-bound} 
Consider \Cref{alg:spectrahedron-fw} applied to the projection subproblem defined in \eqref{eqn:proj-subproblem-H} with $\set{H} = \set{S}_+^n \cap \{\mathbf{X}~:~\operatorname{Tr}(\mathbf{X}) = R \}$.
For any iteration $k$ of \Cref{alg:spectrahedron-fw}, we have that $\mathbb{E}_k\left [ \|\mathbf{P}_k^\perp \vec{u}_{k,-}\|_2^2 \right ] \geq \frac{1}{r_k} \operatorname{Tr}(\mathbf{P}_k^\bot \mathbf{H}_k)$.
\end{lemma}
\begin{lemma}[Eigengap, Lemma 6, \cite{garber2026randomized}] \label{lem:sfw-eigengap} Assume that $r^\star > 2$.
For any iteration $k$ such that $\operatorname{Tr}(\mathbf{P}_k^\bot \mathbf{H}_k) \leq \frac{h_k}{4\|\nabla \phi(\mathbf{H}_k)\|_2}$, then it holds that 
\begin{equation*}
    \lambda_{n - r^\star + 1}(\nabla \phi(\mathbf{H}_k)) - \lambda_n(\nabla \phi(\mathbf{H}_k)) \geq \sqrt{\frac{h_k}{8r^\star}}.
\end{equation*}
\end{lemma}
We now use Weyl's inequality (\Cref{lem:weyl-thm} to find an upper bound for  $\delta_{k,r^\star}$ as follows:

\begin{lemma}[Vanishing in-face eigengap at optimality]
\label{lem:in-face-gap-optimality}
Suppose
$\operatorname{rank}(\mathbf{H}^\star)=r^\star$.
Then
\begin{equation*}
\lambda_{n-r^\star+1}(\nabla \phi(\mathbf{H}^\star))
=
\lambda_n(\nabla \phi(\mathbf{H}^\star)).
\end{equation*}
Consequently, the in-face eigengap at the optimum is zero.
\end{lemma}

\begin{proof}
By the first-order optimality condition, $\mathbf{H}^\star$
minimizes the linear function
\begin{equation*}
    \min_{\substack{\mathbf{X} \in \set{S}_+^n \\ \mathrm{Tr}(\mathbf{X}) = R}}
    \langle \nabla \phi(\mathbf{H}^\star),\mathbf{X}\rangle.
\end{equation*}
By the Rayleigh--Ritz principle, the optimal value of this
problem is $\lambda_n(\nabla \phi(\mathbf{H}^\star))$. Moreover, every
optimal solution is supported on the eigenspace associated
with $\lambda_n(\mathbf{H}^\star)$. Therefore,
\begin{equation*}
    \operatorname{Range}(\mathbf{H}^\star)
    \subseteq
    \operatorname{Eig}_{\lambda_n(\mathbf{H}^\star)}
    (\nabla \phi(\mathbf{H}^\star)).
\end{equation*}
Since $\operatorname{rank}(\mathbf{H}^\star)=r^\star$,
the eigenspace associated with the smallest eigenvalue has
dimension at least $r^\star$. Hence,
\[
\lambda_{n-r^\star+1}(\nabla\phi(\mathbf{H}^\star))
=
\lambda_n(\nabla\phi(\mathbf{H}^\star)),
\]
as claimed.
\end{proof}
\begin{lemma}[Upper bound on the in-face eigengap]
\label{lem:in-face-gap-upper-bound}
Under the assumptions of
Lemma~\ref{lem:in-face-gap-optimality}, it holds that
\begin{equation*}
\delta_{k, r^\star}
\leq
2\|\mathbf{H}_k-\mathbf{H}^\star\|_F \leq 2\sqrt{2h_k},
\end{equation*}
\end{lemma}

\begin{proof}
By Weyl's inequality (\Cref{lem:weyl-thm}) and $1$-smoothness of $\phi$,
\begin{equation*}
    \lambda_{n-r^\star+1}(\nabla(\phi(\mathbf{H}_k))
    \leq
    \lambda_{n-r^\star+1}(\nabla\phi(\mathbf{H}^\star))
    +\|\mathbf{H}_k-\mathbf{H}^\star\|_2
\end{equation*}
and
\begin{equation*}
\lambda_n(\nabla\phi(\mathbf{H}_k))
\geq
\lambda_n(\nabla\phi(\mathbf{H}^\star))
-\|\mathbf{H}_k-\mathbf{H}^\star\|_2.
\end{equation*}
Combining these inequalities and applying
Lemma~\ref{lem:in-face-gap-optimality} yields
\begin{equation*}
\begin{aligned}
\delta_{k, r^\star}
&\leq
\lambda_{n-r^\star+1}(\nabla\phi(\mathbf{H}^\star))
-\lambda_n(\nabla\phi(\mathbf{H}^\star))
+2\|\mathbf{H}_k-\mathbf{H}^\star\|_2\\
&=2\|\mathbf{H}_k-\mathbf{H}^\star\|_2.
\end{aligned}
\end{equation*}
Since the spectral norm is bounded by the Frobenius norm,
the strong-convexity of $\phi$ gives
\begin{equation*}
\begin{aligned}
\delta_{k, r^\star}
&\leq
2\|\mathbf{H}_k-\mathbf{H}^\star\|_F\\
&\leq 2\sqrt{2h_k}.
\end{aligned}
\end{equation*}
This completes the proof.
\end{proof}
The following lemma captures the effect of an inexact LMO or inexact eigenvector on the reduction in the objective due to traditional FW update or an Away step under an  eigengap and is a slight modification of Lemma 7 from \citep{garber2026randomized}:
\begin{lemma}[FW/Away Steps under an  Eigengap, Lemma 7, \citep{garber2023linear}]\label{lem:sfw-afw-step-eigengap} 
Consider \Cref{alg:spectrahedron-fw} applied to the projection subproblem defined in \eqref{eqn:proj-subproblem-H} with $\set{H} = \set{S}_+^n \cap \{\mathbf{X}~:~\operatorname{Tr}(\mathbf{X}) = R \}$.
Assume $r^\star \geq 2$.
For any iteration $k$ of \Cref{alg:spectrahedron-fw} and suppose that $r_k \geq r^\star$ and $\lambda_{r_k}(\mathbf{H}_k) \geq \frac{\delta_{k, r^\star}}{16}$.
Assume the eigenvector of $\nabla \phi(\mathbf{H}_k)$ corresponding to the smallest eigenvalue of $\nabla \phi(\mathbf{H}_k)$ is computed inexactly in the sense of \eqref{eqn:lmo-step-inexact}.
Then we must have that 
\begin{equation*}
    \min \{\phi(\mathbf{H}_{t+1}^{\text{FW}}), \phi(\mathbf{H}_{t+1}^{\text{AFW}})\} \leq \phi(\mathbf{H}_k) - \frac{\delta_{k, r^\star}^2}{64} + \frac{\delta_{k, r^\star}\theta_k}{32R},
\end{equation*}
where $\theta_k \leq \frac{R\xi_k \norm{\nabla \phi(\mathbf{H}_k)}_F}{8}$ and $\xi_k \in (0, 1]$.
\end{lemma}
\begin{proof} 
Suppose $r_k := \operatorname{rank}(\mathbf{H}_k) \geq r^\star$. 
Construct $\mathbf{U} \in \mathbb{R}^{n \times r_k}$ as the orthonormal basis of $\operatorname{Im}(\mathbf{H}_k)$. 
Let $(\mu_1, \vec{w})$ be the eigenpair of $\mathbf{U}^\top \nabla \phi(\mathbf{H}_k)\mathbf{U}$  such that $\mu_1$ is largest eigenvalue of $\mathbf{U}^\top \nabla \phi(\mathbf{H}_k)\mathbf{U}$, so that $\vec{v} = \mathbf{U}\vec{w} \in \operatorname{Im}(\mathbf{H}_k)$.
Then the Poincar\'{e} separation theorem (\Cref{lem:poincare-sep-thm}) implies that 
\begin{equation*}
    \mathbf{v}^\top \nabla \phi(\mathbf{H}_k) \mathbf{v} = \mathbf{w}^\top \mathbf{U}^\top\nabla \phi(\mathbf{H}_k) \mathbf{U}\mathbf{w} = \mu_1 \geq \lambda_{n - r^\star + 1}(\nabla \phi(\mathbf{H}_k)),
\end{equation*}
which in turn implies that 
\begin{equation*}
    \vec{v}_{t,-}^\top \nabla \phi(\mathbf{H}_k) \vec{v}_{t, -} \geq R\lambda_{n - r^\star + 1}(\nabla \phi(\mathbf{H}_k))
\end{equation*}
based on definition of $\vec{v}_{t,-}$.
Given that an inexact LMO or eigenvector is computed in the sense of \eqref{eqn:lmo-step-inexact}, we can use similar reasoning as in the proof of \Cref{thm:lanczos-iter-fw} to say that
\begin{equation*}
    R\lambda_n(\nabla \phi(\mathbf{H}_k)) = \vec{v}_{t,+}^\top \nabla \phi(\mathbf{H}_k)\vec{v}_{t,+} - \theta_k
\end{equation*}
such that $\theta_k \leq \frac{R\xi_k \norm{\nabla \phi(\mathbf{H}_k)}_F}{8}$ (see \eqref{eqn:lanczos-max-iter-lem-proof-1} and \eqref{eqn:lanczos-max-iter-lem-proof-2}).
Together and after using the fact that $\max\{a, b\} \geq \frac{a+b}{2}$ for any real $a$ and $b$, this implies that 
\begin{align}\label{eqn:sfw-lemma7-proof-1}
    \max \left \{ \langle \mathbf{H}_k - \vec{v}_{k,+}\vec{v}_{k,+}^\top, \nabla \phi(\mathbf{H}_k) \rangle, \langle \vec{v}_{k,-}\vec{v}_{k,-}^\top - \mathbf{H}_k, \nabla \phi(\mathbf{H}_k) \rangle \right \} &\geq \frac{1}{2}\langle \vec{v}_{t,-}\vec{v}_{t,-}^\top - \vec{v}_{t,+}\vec{v}_{t,+}^\top, \nabla \phi(\mathbf{H}_k) \rangle \nonumber \\
    &\geq \frac{1}{2}\left (R\lambda_{n - r^\star + 1}(\nabla \phi(\mathbf{H}_k)) - R\lambda_n(\nabla \phi(\mathbf{H}_k)) - \theta_k \right ) \nonumber \\
    &\geq \frac{R\delta_{k, r^\star}}{2} - \frac{\theta_k}{2}.
\end{align}
Since $\phi$ is 1-smooth, we have that 
\begin{align}\label{eqn:sfw-lemma7-proof-2}
    \phi(\mathbf{H}_{t+1}^{\text{FW}}) &\leq \phi(\mathbf{H}_k + \eta(\vec{v}_{k,+}\vec{v}_{k,+}^\top - \mathbf{H}_k)) \nonumber \\
    &\leq \phi(\mathbf{H}_k) + \eta \langle \mathbf{v}_{k, +}\mathbf{v}_{k, +}^\top - \mathbf{H}_k, \nabla \phi(\mathbf{H}_k) \rangle + \frac{\eta^2}{2}\|\mathbf{v}_{k, +}\mathbf{v}_{k, +}^\top - \mathbf{H}_k\|_F^2 \nonumber \\
    &\leq \phi(\mathbf{H}_k) + \eta \langle \mathbf{v}_{k, +}\mathbf{v}_{k, +}^\top - \mathbf{H}_k, \nabla \phi(\mathbf{H}_k) \rangle + \eta^2 R^2 \nonumber \\
    &\leq \phi(\mathbf{H}_k) + \eta \langle \mathbf{v}_{k, +}\mathbf{v}_{k, +}^\top - \mathbf{H}_k, \nabla \phi(\mathbf{H}_k) \rangle + \frac{\eta^2 R^2}{(1 - \eta)^2}
\end{align}
where the first line follows from the fact that $\eta^{\text{FW}}$ is set by line-search and the second-to-last line follows from the fact that
\begin{align*}
    \|\mathbf{v}_{k, +}\mathbf{v}_{k, +}^\top - \mathbf{H}_k\|_F^2 &\leq \|\mathbf{v}_{k, +}\mathbf{v}_{k, +}^\top\|_F^2 + \|\mathbf{H}_k\|_F^2 - 2\ip{\vec{v}_{k,+}\vec{v}_{k,+}^\top}{\mathbf{H}_k} \\
    &\leq \mbox{Tr}((\mathbf{v}_{k, +}\mathbf{v}_{k, +}^\top)^\top \mathbf{v}_{k, +}\mathbf{v}_{k, +}^\top) + \sum_{i = 1}^n \lambda_i(\mathbf{H}_k)^2 \\
    &\leq R^2 + \left (\sum_{i = 1}^n \lambda_i(\mathbf{H}_k) \right )^2 = 2R^2,
\end{align*}
where we used the fact that the third term is an inner product of two positive semidefinite matrices and is there non-negative in the second step.
Similarly, for any step-size $\eta \in \left [0, \frac{\lambda_{r_k}(\mathbf{H}_k)}{R} \right] \subset [0, (\mathbf{v}_{k, -}^\top \mathbf{H}_k^\dagger \mathbf{v}_{k, -})^{-1}]$, it holds that 
\begin{align}\label{eqn:sfw-lemma7-proof-3}
    \phi(\mathbf{H}_{t+1}^{\text{AFW}}) &\leq \phi\left ( \frac{1}{1 - \eta} (\mathbf{H}_k - \eta \vec{v}_{k,-}\vec{v}_{k,-}^\top) \right ) \nonumber \\ 
    &= \phi \left ( \mathbf{H}_k + \frac{\eta}{1 - \eta}(\mathbf{H}_k - \vec{v}_{k,-}\vec{v}_{k,-}^\top)\right ) \nonumber \\
    &\leq \phi(\mathbf{H}_k) + \frac{\eta}{1 - \eta}\langle \mathbf{H}_k - \vec{v}_{k,-}\vec{v}_{k,-}^\top, \nabla \phi (\mathbf{H}_k) \rangle + \frac{\eta^2 R^2}{(1 - \eta)^2} \nonumber \\
    &\leq \phi(\mathbf{H}_k) + \eta \langle \mathbf{H}_k - \vec{v}_{k,-}\vec{v}_{k,-}^\top, \nabla \phi (\mathbf{H}_k) \rangle + \frac{\eta^2R^2}{(1 - \eta)^2},
\end{align}
where the first line follows from the fact that $\eta^{\text{AFW}}$ is set by line-search and the last line follows from the fact that $\eta \in [0, 1)$ and from the fact that $ \langle \mathbf{H}_k - \vec{v}_{k,-}\vec{v}_{k,-}^\top, \nabla \phi (\mathbf{H}_k) \rangle \leq 0$.
Combining \eqref{eqn:sfw-lemma7-proof-1}, \eqref{eqn:sfw-lemma7-proof-2} and \eqref{eqn:sfw-lemma7-proof-3}, it follows that 
\begin{align*}
    \min\{\phi(\mathbf{H}_{t+1}^{\text{FW}}), \phi(\mathbf{H}_{t+1}^{\text{AFW}})\} &\leq \phi(\mathbf{H}_k) - \eta \max \left \{ \langle \mathbf{H}_k - \vec{v}_{t,+}\vec{v}_{t,+}^\top), \nabla \phi(\mathbf{H}_k) \rangle, \langle \vec{v}_{t,-}\vec{v}_{t,-}^\top - \mathbf{H}_k, \nabla \phi(\mathbf{H}_k) \rangle \right \} + \frac{\eta^2 R^2}{(1 - \eta)^2} \nonumber \\
    &\leq \phi(\mathbf{H}_k) - \frac{\eta R \delta_{k, r^\star}}{2} + \frac{\eta \theta_k}{2} + \frac{\eta^2 R^2}{(1 - \eta)^2}
\end{align*}
for all $\eta \in [0, \frac{\lambda_{r_k}(\mathbf{H}_k)}{R}]$.
Since $r_k \geq r^\star \geq 2$ and given the constraint $\operatorname{Tr}(\mathbf{H}_k) = R$, we must have that $\lambda_{r_k}(\mathbf{H}_k) \leq \frac{R}{2}$.
Given this and the assumption that $\lambda_{r_k}(\mathbf{H}_k) \geq \frac{\delta_{k, r^\star}}{16}$, we can set $\eta = \frac{\delta_{k, r^\star}}{16R}$ to get
\begin{align*}
    \min\{\phi(\mathbf{H}_{t+1}^{\text{FW}}), \phi(\mathbf{H}_{t+1}^{\text{AFW}})\} &\leq \phi(\mathbf{H}_k) - \frac{\delta_{k, r^\star}^2}{32} + \frac{\delta_{k, r^\star}\theta_k}{32R} + \frac{\delta_{k, r^\star}^2}{64} \\  
    &= \phi(\mathbf{H}_k) - \frac{\delta_{k, r^\star}^2}{64} + \frac{\delta_{k, r^\star}\theta_k}{32R},
\end{align*}
where we used the fact that $\lambda_{r_k}(\mathbf{H}_k) \leq \frac{R}{2}$ implies that $\eta \leq \frac{1}{2}$ to bound the last term in the first step as follows: $\frac{\eta^2}{(1 - \eta)^2} \leq \frac{\eta^2}{\left (1 - \frac{1}{2} \right )^2} = \frac{\delta_{k, r^\star}^2}{64}$.
\end{proof}
We now present a useful lemma from \citep{garber2023linear} that helps prove the bound for the case where the rank $r^\star$ of the optimal projection is $1$:
\begin{lemma}[Lemma 4, \citep{garber2023linear}]\label{lem:sfw-key-ineq-quad-growth-rank-1}
Let $\mathbf{H} \in \set{S}_+^n \cap \{\mathbf{X}~:~\operatorname{Tr}(\mathbf{X}) = R \}$ and suppose that $\lambda_{n - 1}(\nabla \phi (\mathbf{H})) - \lambda_n(\nabla \phi (\mathbf{H})) \geq \delta_{\mathbf{H}}$ for some $\delta_{\mathbf{H}} > 0$.
Let $\vec{u}_n$ be an eigenvector of $\nabla \phi(\mathbf{H})$ associated with the eigenvalue $\lambda_n(\nabla \phi(\mathbf{H}))$.
Then the following must hold:
\begin{equation}\label{eqn:sfw-rank-1-key-ineq}
    \langle \mathbf{X} - R\vec{u}_n \vec{u}_n^\top, \nabla \phi (\mathbf{H}) \rangle \geq \delta_{\mathbf{H}}(R - \vec{u}_n^\top \mathbf{X} \mathbf{u}_n) \geq \frac{\delta_{\mathbf{H}}}{2R} \|\mathbf{X} - R\vec{u}_n \vec{u}_n^\top\|_F^2
\end{equation}
for all $\mathbf{X} \in \set{S}_+^n \cap \{\mathbf{X}~:~\operatorname{Tr}(\mathbf{X}) = R \}$.
Also, this implies that if $\mathbf{H}^\star$ is an optimal solution to the projection problem in \eqref{eqn:proj-subproblem-H} that satisfies the strict complementarity condition in \eqref{eqn:strict-complementarity-eigengap} with $r^\star = 1$, then the following quadratic growth property holds for the projection objective $\phi$:
\begin{equation}\label{eqn:sfw-rank-1-quad-growth-condition}
    \|\mathbf{X} - \mathbf{H}^\star\|_F^2 \leq \frac{2R}{\delta}(\phi(\mathbf{X}) - \phi(\mathbf{H}^\star))     
\end{equation}
for all $\mathbf{X} \in \set{S}_+^n \cap \{\mathbf{X}~:~\operatorname{Tr}(\mathbf{X}) = R \}$.
\end{lemma}
\begin{proof} Fix some $\mathbf{X} \in \set{S}_+^n \cap \{\mathbf{X}~:~\operatorname{Tr}(\mathbf{X}) = R \}$ and let us write the eigendecomposition of $\nabla \phi (\mathbf{H})$ as $\nabla \phi (\mathbf{H}) = \sum_{i = 1}^n \lambda_i \mathbf{u}_i \mathbf{u}_i^\top$, where eigenvalues are ordered in decreasing order.
Given that $\lambda_{n - 1} - \lambda_n = \delta_{\mathbf{H}}$, it follows that 
\begin{align*}
    \langle \mathbf{X} - R\mathbf{u}_n \mathbf{u}_n^\top, \nabla \phi (\mathbf{H}) \rangle &= \left \langle \mathbf{X} - R\vec{u}_n \vec{u}_n^\top,  \sum_{i = 1}^n \lambda_i \vec{u}_i \vec{u}_i^\top \right \rangle \nonumber\\
    &= \sum_{i = 1}^n \lambda_i \vec{u}_i^\top \mathbf{X} \mathbf{u}_i - R\lambda_n \nonumber \\
    &\geq (\lambda_n + \delta_{\mathbf{H}})\sum_{i = 1}^{n-1} \vec{u}_i^\top \mathbf{X}\vec{u}_i + \lambda_n \vec{u}_n^\top \mathbf{X}\vec{u}_n - R\lambda_n \nonumber \\
    &= \delta_{\mathbf{H}}\sum_{i = 1}^{n - 1} \vec{u}_i^\top \mathbf{X}\vec{u}_i + \lambda_n \left ( \sum_{i = 1}^n \vec{u}_n^\top \mathbf{X}\vec{u}_n - R\right )\\
    &= \delta_{\mathbf{H}}\sum_{i = 1}^{n - 1} \vec{u}_i^\top \mathbf{X}\vec{u}_i = \delta_{\mathbf{H}}(R - \mathbf{u}_n^\top \mathbf{X}\mathbf{u}_n).
\end{align*}
where the last two equalities follows from the fact that $\sum_{i = 1}^n \vec{u}_i^\top \mathbf{X} \vec{u}_i = R$.
Since $\|\mathbf{X}\|_F \leq R$ and $\|\mathbf{u}_n \mathbf{u}_n^\top \|^2 = 1$, we can write 
\begin{align*}
    \langle \mathbf{X} - R\vec{u}_n \vec{u}_n^\top, \nabla \phi(\mathbf{H}) \rangle &\geq \delta_{\mathbf{H}}(R - \vec{u}_n^\top \mathbf{X} \vec{u}_n) \\
    &\geq \frac{\delta_{\mathbf{H}}}{2R}(R^2 \|\vec{u}_n \vec{u}_n^\top\|^2 + \|\vec{X}\|_F^2 - 2R\vec{u}_n^\top \mathbf{X} \vec{u}_n) \\
    &= \frac{\delta_{\mathbf{H}}}{2R}\|\mathbf{X} - R\vec{u}_n \vec{u}_n^\top\|_F^2.
\end{align*}
This proves the first inequality described in \eqref{eqn:sfw-rank-1-key-ineq}.
If we set $\mathbf{H} = \mathbf{H}^\star = R\vec{u}_n \vec{u}_n^\top$ since the optimal projection $\mathbf{H}^\star$ is assumed to be rank 1 and $\delta_{\mathbf{H}} = \delta$ to reflect the strict complementarity condition in \eqref{eqn:strict-complementarity-eigengap}, then we have that 
\begin{equation*}
    \frac{\delta}{2R}\|\mathbf{X} - \mathbf{H}^\star\|_F^2 \leq \langle \mathbf{X} - \mathbf{H}^\star, \nabla \phi(\mathbf{H}^\star) \rangle \leq \phi(\mathbf{X}) - \phi(\mathbf{H}^\star),
\end{equation*}
where we used the convexity of $\phi$ in the last step above.
This concludes the proof of the quadratic growth property of $\phi$ in \eqref{eqn:sfw-rank-1-quad-growth-condition}.
\end{proof}
The following lemma is based on Theorem 2 and its proof from \citep{garber2023linear} and shows that the objective residual decreases given a traditional FW step using LMO for the special case where the projection is of rank 1 and quantifies this decrease given the case when eigenvector of $\nabla \phi(\mathbf{H}_k)$ corresponding to the smallest eigenvalue of $\nabla \phi(\mathbf{H}_k)$ is computed inexactly in the sense of \eqref{eqn:lmo-step-inexact}.
\begin{lemma}[Objective Residual Decrease from FW Step for Rank-1 Optimum]\label{lem:obj-residual-sfw-rank-1} Consider \Cref{alg:spectrahedron-fw} applied to the projection subproblem defined in \eqref{eqn:proj-subproblem-H} with $\set{H} = \set{S}_+^n \cap \{\mathbf{X}~:~\operatorname{Tr}(\mathbf{X}) = R \}$. 
Suppose $r^\star = 1$.
Assume the eigenvector of $\nabla \phi(\mathbf{H}_k)$ corresponding to the smallest eigenvalue of $\nabla \phi(\mathbf{H}_k)$ is computed inexactly in the sense of \eqref{eqn:lmo-step-inexact}.
For each iteration $k$ of \Cref{alg:spectrahedron-fw} such that $\|\nabla \phi(\mathbf{H}_k) - \nabla \phi(\mathbf{H}^\star)\|_F \leq \frac{\delta}{3}$, we must have that 
\begin{equation*}
    \phi(\mathbf{H}_{k + 1}^{FW}) - \phi(\mathbf{H}^\star) \leq \left (1 - \min\left \{ \frac{\delta}{12R}, \frac{1}{2}\right \} \right )(\phi(\mathbf{H}_k) - \phi(\mathbf{H}^\star)) +  
    \begin{cases}
        \frac{\delta^2 \|\mathbf{G}\|_F}{18R} + \left (\frac{\delta}{6R} - \frac{\delta^2}{36R^2} \right )\theta_k, & \delta \leq 6R \\
        2R\|G\|_F, & \mbox{otherwise}
    \end{cases}
\end{equation*}
where $\theta_k \leq \frac{R\xi_k \norm{\nabla \phi(\mathbf{H}_k)}_F}{8}$ and $\xi_k \in (0, 1]$.
\end{lemma}
\begin{proof}
Let $\nabla \phi (\mathbf{H}_k) = \sum_{i = 1}^n \lambda_i \vec{u}_i \vec{u}_i^\top$ with $\mathbf{u}_i$ being the corresponding eigenvector for the $i$th eigenvalue.
We can lower-bound the distance between the smallest and second-largest eigenvalue of $\nabla \phi(\mathbf{H}_k)$ as follows:
\begin{align}
    \lambda_{n - 1} - \lambda_n &= \lambda_{n - 1}(\nabla \phi(\mathbf{H}^\star)) - \lambda_n(\nabla \phi(\mathbf{H}^\star)) + (\lambda_{n-1} - \lambda_{n - 1}(\nabla \phi(\mathbf{H}^\star))) + (\lambda_n(\nabla \phi(\mathbf{H}^\star)) - \lambda_n) \nonumber \\
    &\geq \lambda_{n - 1}(\nabla \phi(\mathbf{H}^\star)) - \lambda_n(\nabla \phi(\mathbf{H}^\star)) - 2\|\nabla \phi(\mathbf{H}_k) - \nabla \phi(\mathbf{H}^\star)\|_F \nonumber \\
    &\geq \delta - \frac{2\delta}{3} = \frac{\delta}{3}.
\end{align}
where we used Weyl's inequality (\Cref{lem:weyl-thm}) twice in the second-from-last step to say that
\begin{equation*}
        |\lambda_{n - 1} - \lambda_{n - 1}(\nabla \phi (\mathbf{H}^\star)| \leq \|\nabla \phi(\mathbf{H}_k) - \nabla \phi(\mathbf{H}^\star)\|_F \quad\mbox{and}\quad
        |\lambda_n - \lambda_n(\nabla \phi (\mathbf{H}^\star)| \leq \|\nabla \phi(\mathbf{H}_k) - \nabla \phi(\mathbf{H}^\star)\|_F
\end{equation*}
implies that 
\begin{equation*}
    \lambda_{n - 1} - \lambda_{n - 1}(\nabla \phi (\mathbf{H}^\star)) + \lambda_n(\nabla \phi (\mathbf{H}^\star)) - \lambda_n \geq -2\|\nabla \phi(\mathbf{H}_k) - \nabla \phi(\mathbf{H}^\star)\|_F.
\end{equation*}
Since $\lambda_n < \lambda_{n - 1}$, the LMO step admits a unique optimal solution $\vec{u}_n$ that is the eigenvector corresponding to $\lambda_n$.
Given that the LMO step and eigenvector of $-\nabla \phi(\mathbf{H}_k)$ is computed inexactly in the sense of \eqref{eqn:lmo-step-inexact}, we have that 
\begin{equation*}
    \langle R\vec{u}_n \vec{u}_n^\top, \nabla \phi(\mathbf{H}_k) \rangle =  \langle \vec{v}_{k,+} \vec{v}_{k,+}^\top, \nabla \phi(\mathbf{H}_k) \rangle - \theta_k 
\end{equation*}
such that $\theta_k \leq \frac{R\xi_k \norm{\nabla \phi(\mathbf{H}_k)}_F}{8}$ and $\xi_k \in (0, 1]$ (see \eqref{eqn:lanczos-max-iter-lem-proof-1} and \eqref{eqn:lanczos-max-iter-lem-proof-2}).
Given this and the fact that $\phi$ is 1-smooth, it holds that 
\begin{align} \label{eqn:fw-proof-rank-1-ineq}
    \phi(\mathbf{H}_{k + 1}^{\text{FW}}) &\leq \phi(\mathbf{H}_k) + \eta \langle \vec{v}_{k,+}\vec{v}_{k,+}^\top - \mathbf{H}_k, \nabla \phi (\mathbf{H}_k) \rangle + \frac{\eta^2}{2}\|\vec{v}_{k,+}\vec{v}_{k,+}^\top - \mathbf{H}_k\|_F^2 \nonumber \\
    &\leq \phi(\mathbf{H}_k) + \eta \langle R\vec{u}_n\vec{u}_n^\top - \mathbf{H}_k, \nabla \phi (\mathbf{H}_k) \rangle + \eta \theta_k + \frac{\eta^2}{2}\|R\vec{u}_n\vec{u}_n^\top - \mathbf{H}_k\|_F^2\nonumber \\
    &\qquad + \frac{\eta^2}{2}(\|\vec{v}_{k,+}\vec{v}_{k,+}^\top - \mathbf{H}_k\|_F^2 - \|R\vec{u}_n\vec{u}_n^\top - \mathbf{H}_k\|_F^2).
\end{align}
We now bound the last term as follows.
Using \(\mathbf{H}_k=\mathbf{G}+\nabla\phi(\mathbf{H}_k)\) and the inexact LMO condition, we obtain
\begin{align} \label{eqn:fw-proof-rank-1-ineq-last-term}
    \frac{\eta^2}{2}\left(
    \|\mathbf{v}_{k,+}\mathbf{v}_{k,+}^\top-\mathbf{H}_k\|_F^2
    -\|R\mathbf{u}_n\mathbf{u}_n^\top-\mathbf{H}_k\|_F^2
    \right)
    &=\eta^2\langle \mathbf{H}_k,
    R\mathbf{u}_n\mathbf{u}_n^\top-\mathbf{v}_{k,+}\mathbf{v}_{k,+}^\top\rangle \nonumber \\
    &=\eta^2\langle \mathbf{G},
    R\mathbf{u}_n\mathbf{u}_n^\top-\mathbf{v}_{k,+}\mathbf{v}_{k,+}^\top\rangle
    - \eta^2\theta_k \nonumber \\
    &\leq 2\eta^2 R \|\mathbf{G}\|_F-\eta^2\theta_k,
\end{align}
where we used the fact that both rank-one matrices have the same squared Frobenius norm of $R^2$.
We then substitute both \eqref{eqn:fw-proof-rank-1-ineq-last-term} and \eqref{eqn:sfw-rank-1-key-ineq} with $\delta_{\mathbf{H}} = \delta/3$ into \eqref{eqn:fw-proof-rank-1-ineq} and then subtract $\phi(\mathbf{H}^\star)$ from both sides of \eqref{eqn:fw-proof-rank-1-ineq} to get 
\begin{multline*}
    h_{k + 1} \leq h_k + \left ( \eta - \frac{3R\eta^2}{\delta} \right )\langle R\vec{u}_n\vec{u}_n^\top - \mathbf{H}_k, \nabla \phi (\mathbf{H}_k) \rangle + 2\eta^2 R \|\mathbf{G}\|_F + (\eta - \eta^2)\theta_k \\
    \leq \left (1 - \left ( \eta - \frac{3R\eta^2}{\delta}\right ) \right )h_k + 2\eta^2 R \|\mathbf{G}\|_F + (\eta - \eta^2)\theta_k
\end{multline*}
for all $0 \leq \eta \leq \frac{\delta}{3R}$, where we used \eqref{eqn:sfw-rank-1-key-ineq} in the first inequality and convexity of $\phi$ for the second inequality above.
We can then consider the following two cases: $\frac{\delta}{6R} \leq 1$ and $\frac{\delta}{6R} > 1$. 
In the first case, we can set $\eta = \frac{\delta}{6R}$ to get 
$h_{k + 1} \leq \left (1 - \frac{\delta}{12R} \right ) h_k  + \frac{\delta^2 \|\mathbf{G}\|_F}{18R} + \left (\frac{\delta}{6R} - \frac{\delta^2}{36R^2} \right )\theta_k$. 
For the second case, we can choose $\eta = 1$ and use the fact that $\delta > 6R$ to get $h_{t + 1} \leq \frac{1}{2}h_k + 2R \|\mathbf{G}\|_F$.
\end{proof}
Usually, $\delta < 6R$, so the more practically relevant error bound is $\frac{\delta^2 \|\mathbf{G}\|_F}{18R} + \left (\frac{\delta}{6R} - \frac{\delta^2}{36R^2} \right )\theta_k$, which is small since $\frac{\delta}{R}$ is usually quite small in practice and $\mathbf{G}$ is scaled and shifted prior to projection.
Therefore, the above result shows a linear convergence to a small neighborhood around the optimal projection in the case $r^\star = 1$ given inexact eigenvector computation.
Now given all the previously proved results from \citep{garber2023linear} and \citep{garber2026randomized}, we can prove the following equivalent result to Theorem 2 in \citep{garber2026randomized} that handles the scenario when the eigenvector of $-\nabla \phi(\mathbf{H}_k)$ is computed inexactly in the sense of \eqref{eqn:lmo-step-inexact}:  
\begin{theorem} \label{thm:sfw-convergence-inexact-evec}
Consider \Cref{alg:spectrahedron-fw} applied to the projection subproblem defined in \eqref{eqn:proj-subproblem-H} with $\set{H} = \set{S}_+^n \cap \{\mathbf{X}~:~\operatorname{Tr}(\mathbf{X}) = R \}$.
Assume the eigenvector of $\nabla \phi(\mathbf{H}_k)$ corresponding to the smallest eigenvalue of $\nabla \phi(\mathbf{H}_k)$ is computed inexactly in the sense of \eqref{eqn:lmo-step-inexact}.
For the inexact eigenvector computation, assume that $\theta_k \leq \frac{R\xi_k \norm{\nabla \phi(\mathbf{H}_k)}_F}{8}$ for some $\xi_k \in (0, 1]$ (see \eqref{eqn:lanczos-max-iter-lem-proof-1} and \eqref{eqn:lanczos-max-iter-lem-proof-2}). 
The sequence of iterates $\{\mathbf{H}_k\}_{k \geq 0}$ generated by \Cref{alg:spectrahedron-fw} are non-increasing ($h_{k + 1} \leq h_k$ for all $k \geq 0$) and for all $k \geq \operatorname{rank}(\mathbf{H}_1) + 2$, we have that 
\begin{equation} \label{eqn:upper-bound-obj-drop-steps}
    h_k \leq \frac{8R^2}{k - \operatorname{rank}(\mathbf{H}_1) + 4}.
\end{equation}
The overall number of drop steps $t_{\text{drop}}$ cannot exceed $(k + \operatorname{rank}(\mathbf{H}_1) - 2)/2$.
For each iteration $k$ of \Cref{alg:spectrahedron-fw} where a drop step isn't taken and $h_k \leq \frac{\delta^2}{18}$ and $r^\star = 1$, we have that 
\begin{equation}\label{eqn:sfw-rank-1-fw-rate}
    h_{k+1} \leq h_k\left ( 1 - \min \left \{ \frac{\delta}{12R}, \frac{1}{2}\right \} \right ) +  
    \begin{cases}
        \frac{\delta^2 \|\mathbf{G}\|_F}{18R} + \left (\frac{\delta}{6R} - \frac{\delta^2}{36R^2} \right )\theta_k, & \delta \leq 6R \\
        2R\|G\|_F, & \mbox{otherwise}
    \end{cases}.
\end{equation}
If $r^\star = n$ and $h_k <  \frac{\lambda_{r^\star}^2}{8} $ and a drop step isn't taken, then we have that 
\begin{equation}\label{eqn:sfw-rank-n-fw-rate}
    h_{k+1} \leq h_k\left ( 1 - \frac{1}{512n}\right ) + \frac{\sqrt{2h_k}\theta_k}{16R}.
\end{equation}
Finally, if $n > r^\star \geq 2$ and $h_k < \min \left \{ \epsilon_1, \epsilon_2, \epsilon_3, \epsilon_4 \right \}$ with $\epsilon_1 := \frac{\delta^2}{32}$, $\epsilon_2 := \left \{\frac{\delta^2}{72}, \frac{R\delta}{12}\right \}$, $\epsilon_3 := \frac{\lambda_{r^\star}^2}{8}$ and $\epsilon_4 := \frac{\delta^3}{192R}$ and if a drop step isn't taken, then we get
\begin{equation}\label{eqn:sfw-rank-r-star-l-n-fw-rate}
    \mathbb{E}[h_{k+1}|\mathbf{H}_k] \leq h_k\left (1 - \frac{1}{r^\star}\min\left \{ \frac{\delta}{48GR}\min\left \{\frac{\delta}{3},\lambda_{r^\star} \right \}, \frac{1}{512}\right \} \right ) + \frac{\sqrt{2h_k}\theta_k}{16R},
\end{equation}
where the expectation is with respect to the random choice of $\vec{u}_{k, -}$ and $G := \sup_{\mathbf{H} \in \set{S}_+^n \cap \{\mathbf{X}~:~\operatorname{Tr}(\mathbf{X}) = R \}} \|\nabla \phi(\mathbf{H})\|_F$.
\end{theorem}
\begin{proof}
The non-increasing nature of the sequence $\{\mathbf{h}_k\}_{k \geq 0}$ is clear since (1) a drop step is taken only if it doesn't increase the objective and (2) if a drop step is not taken, then line search for $\eta_k^{FW}$ necessitates that $\phi(\mathbf{H}_{k + 1}) \leq \phi(\mathbf{H}_{k + 1}^{FW}) \leq \phi(\mathbf{H}_k)$ and it holds that $\mathbf{H}_{k+1} = \arg\min_{\mathbf{H} \in \{\mathbf{H}^{\text{FW}}_{k+1}, \mathbf{H}^{\text{AFW}}_{k+1}, \mathbf{H}^{\text{PFW}}_{k+1}\}} \phi(\mathbf{H})$.
The upper bound on the number of drop steps can be shown as follows. 
Each drop step reduces the rank of resulting matrix $\mathbf{H}_k$ by one and an LMO update, away step or pairwise step increase the rank by at most one.
Therefore, if denote the number of drop steps up to the beginning of some iteration $k$ by $k_{\text{drop}}$, it holds that
\begin{equation*}
    1 \leq \operatorname{rank}(\mathbf{H}_1) \leq \operatorname{rank}(\mathbf{H}_1) - k_{\text{drop}} + (k - 1 - k_{\text{drop}}).
\end{equation*}
Rearranging the above inequality gives the desired upper bound on $k_{\text{drop}}$.
To get \eqref{eqn:upper-bound-obj-drop-steps}, we use can use this upper bound and the well-known result that $h^{\text{FW}}_k \leq \frac{4R^2}{k + 2}$ for standard FW over a spectrahedron (see \citep{jaggi2013revisiting} for an example).

To prove \eqref{eqn:sfw-rank-1-fw-rate}, we can use the 1-smoothness and strong convexity of $\phi$ to say that 
\begin{equation*}
    \|\nabla \phi (\mathbf{H}_k) - \nabla \phi (\mathbf{H}^\star)\|_F \leq \|\mathbf{H}_k - \mathbf{H}^\star\|_F \leq \sqrt{2h_k}.
\end{equation*}
It then follows that if $h_k <  \frac{\delta^2}{18}$ then $\|\nabla \phi (\mathbf{H}_k) - \nabla \phi (\mathbf{H}^\star)\|_F < \frac{\delta}{3}$, so \eqref{eqn:sfw-rank-1-fw-rate} follows directly from  \Cref{lem:obj-residual-sfw-rank-1}.

To prove \eqref{eqn:sfw-rank-r-star-l-n-fw-rate}, we proceed as follows. 
Using Weyl's theorem (\Cref{lem:weyl-thm}), we get
\begin{equation} \label{eqn:sfw-main-thm-proof-1}
    |\lambda_i(\nabla \phi(\mathbf{H}_k)) - \lambda_i(\nabla \phi(\mathbf{H}^\star))| \leq \|\nabla \phi(\mathbf{H}_k) - \nabla \phi(\mathbf{H}^\star)\|_F \leq \|\mathbf{H}_k - \mathbf{H}^\star\|_F \leq \sqrt{2h_k},
\end{equation}
where the second-from-last inequality follows from $1$-smoothness of $\phi$ and the second from the strong convexity of $\phi$.
From this, it follows that $h_k \leq \epsilon_1 := \frac{\delta^2}{32}$ implies that 
\begin{align}\label{eqn:sfw-main-thm-proof-2}
    \delta_k &= \lambda_{n - r^\star}(\nabla \phi (\mathbf{H}_k)) - \lambda_{n - r^\star + 1}(\nabla \phi (\mathbf{H}_k)) \nonumber \\
    &= \lambda_{n - r^\star}(\nabla \phi (\mathbf{H}^\star)) - \lambda_{n - r^\star + 1}(\nabla \phi (\mathbf{H}^\star)) - (\lambda_{n - r^\star + 1}(\nabla \phi (\mathbf{H}_k)) - \lambda_{n - r^\star + 1}(\nabla \phi (\mathbf{H}^\star))) \nonumber \\
    &\quad - (\lambda_{n - r^\star}(\nabla \phi (\mathbf{H}^\star)) - \lambda_{n - r^\star}(\nabla \phi (\mathbf{H}_k))) \nonumber \\
    &\geq \lambda_{n - r^\star}(\nabla \phi (\mathbf{H}^\star)) - \lambda_{n - r^\star + 1}(\nabla \phi (\mathbf{H}^\star)) - |\lambda_{n - r^\star + 1}(\nabla \phi (\mathbf{H}_k)) - \lambda_{n - r^\star + 1}(\nabla \phi (\mathbf{H}^\star))| \nonumber \\
    &\quad - |\lambda_{n - r^\star}(\nabla \phi (\mathbf{H}^\star)) - \lambda_{n - r^\star}(\nabla \phi (\mathbf{H}_k))| \nonumber \\
    &\geq \delta - 2\sqrt{2h_k} = \frac{\delta}{2},
\end{align}
where we used \eqref{eqn:strict-complementarity-eigengap} and \eqref{eqn:sfw-main-thm-proof-1} in the second-to-last step.
Define $\epsilon_2 := \min\left \{\frac{(\delta/2)^2}{18}, \frac{R\delta/2}{6}\right \} = \left \{\frac{\delta^2}{72}, \frac{R\delta}{12}\right \}$.
Given \Cref{lem:sfw-lemma-drop-step}, if $h_k < \min\{ \epsilon_1, \epsilon_2\}$ and $\operatorname{rank}(\mathbf{H}_k) > r^\star$, then a drop step is taken.
Otherwise, if a drop step is not taken, it must be the case that $\operatorname{rank}(\mathbf{H}_k) \leq r^\star$.
If $h_k < \epsilon_3 := \frac{\lambda_{r^\star}^2}{8}$ with $\lambda_{r^\star} := \lambda_{r^\star}(\mathbf{H}^\star)$ and using the strong convexity of $\phi$ and Weyl's inequality (\Cref{lem:weyl-thm}), we have that 
\begin{equation*}
    |\lambda_{r^\star}(\mathbf{H}_k) - \lambda_{r^\star}(\mathbf{H}^\star)| \leq \|\mathbf{H}_k - \mathbf{H}^\star\|_F \leq \sqrt{2 h_k} \leq \frac{\lambda_{r^\star}}{2}.
\end{equation*}
Based on this, we must have that 
\begin{equation}\label{eqn:sfw-main-thm-proof-3}
    \lambda_{r^\star}(\mathbf{H}_k) > \lambda_{r^\star}(\mathbf{H}^\star) - \frac{\lambda_{r^\star}}{2} = \lambda_{r^\star} -  \frac{\lambda_{r^\star}}{2} = \frac{\lambda_{r^\star}}{2} 
\end{equation}
so it holds that $\operatorname{rank}(\mathbf{H}_k) \geq r^\star$.
Based on this, if $h_k < \min\{\epsilon_1, \epsilon_2, \epsilon_3\}$ and a drop step hasn't been performed, it must hold that $\operatorname{rank}(\mathbf{H}_k) = r^\star$.
Given this, we can express $\mathbf{H}_k$ as follows:
\begin{equation*}
    \mathbf{H}_k = \mathbf{U}\bm{\Lambda}\mathbf{U}^\top,
\end{equation*}
where $\bm{\Lambda} = \operatorname{Diag}(\lambda_1(\mathbf{H}_k), \ldots, \lambda_{r^\star}(\mathbf{H}_k))$ and $\mathbf{U} \in \mathbb{R}^{n \times r^\star}$.
Since $\gamma_k := (\vec{u}_{k, -}^\top \mathbf{H}_k^\dagger \vec{u}_{k,-})^{-1}$.
Since $\vec{u}_{k, -} = \mathbf{U}\vec{v}$ for some $\vec{v} = (v_1, \ldots, v_{r^\star}) \in \mathbb{R}^{r^\star}$, we have that 
\begin{equation*}
    \vec{u}_{t, -}^\top\mathbf{H}_k^\dagger \vec{u}_{t,-} = \sum_{i = 1}^{r^\star} \frac{v_i^2}{\lambda_i(\mathbf{H}_k)} \leq \sum_{i = 1}^{r^\star} \frac{v_i^2}{\lambda_{r^\star}(\mathbf{H}_k)} = \frac{\|\vec{v}\|^2}{\lambda_{r^\star}(\mathbf{H}_k)} = \frac{R}{\lambda_{r^\star}(\mathbf{H}_k)},
\end{equation*}
where the last step follows from the fact that $\|\vec{v}\| = \|\mathbf{U}\vec{v}\| = \|\vec{u}_{t,-}\|$.
Therefore, it also holds that $\gamma_k \geq \frac{\lambda_{r^\star}(\mathbf{H}_k)}{R} \geq \frac{\lambda_{r^\star}}{2R}$.
Define also $\epsilon_4 := \frac{(\delta/2)^3}{24R} = \frac{\delta^3}{192R}$. Then suppose that $h_k < \min \{\epsilon_1, \epsilon_2, \epsilon_3, \epsilon_4\}$, $\operatorname{Tr}(\mathbf{P}_k^\bot \mathbf{H}_k) > \frac{h_k}{4G}$ and a drop step hasn't been taken, then we can apply \Cref{lem:sfw-lemma-pfw-step} and \Cref{lem:sfw-trace-bound} to say that 
\begin{align*}
    \mathbb{E}_k [\phi(\mathbf{H}_{k+1}^{\text{PFW}})] &\leq \phi(\mathbf{H}_k) - \mathbb{E}_k\left [ \min \left \{\frac{\delta_k^2}{9R}, \frac{\gamma_k \delta_k}{3} \right \} \|\mathbf{P}_k^\bot \mathbf{u}_{k,-}\|^2\right ] \nonumber \\
    &\leq \phi(\mathbf{H}_k) - \min \left \{\frac{\delta^2}{36R}, \frac{\lambda_{r^\star}\delta}{12R} \right \}\mathbb{E}_k\left [\|\mathbf{P}_k^\bot \mathbf{u}_{k,-}\|^2\right ] \nonumber \\
    &\leq \phi(\mathbf{H}_k) - \frac{\delta}{48Gr^\star R}\min \left \{\frac{\delta}{3}, \lambda_{r^\star} \right \}h_k,
\end{align*}
where the second-to-last inequality follows from the fact that $\gamma_k \geq \frac{\lambda_{r^\star}}{2R}$ and \eqref{eqn:sfw-main-thm-proof-2}.
Therefore, we have that
\begin{align}\label{eqn:sfw-main-thm-proof-4}
    \mathbb{E}_k[\phi(\mathbf{H}_{k+1})] - \phi(\mathbf{H}^\star) &\leq \mathbb{E}_k\left [ \min \{\phi(\mathbf{H}_{k+1}^{\text{FW}}), \phi(\mathbf{H}_{k+1}^{\text{AFW}}), \phi(\mathbf{H}_{k+1}^{\text{PFW}})\}\right ] - \phi(\mathbf{H}^\star) \nonumber \\
    &\leq \min \left \{\phi(\mathbf{H}_{k+1}^{\text{FW}}), \phi(\mathbf{H}_{k+1}^{\text{AFW}}), \mathbb{E}_k\left [ \phi(\mathbf{H}_{k+1}^{\text{PFW}})\right ]\right \}  - \phi(\mathbf{H}^\star) \nonumber \\
    &\leq h_k \left (1 - \frac{\delta}{48Gr^\star R}\min \left \{ \frac{\delta}{3}, \lambda_{r^\star}\right \} \right ).
\end{align}
Here, $\mathbb{E}_k[\cdot]$ is with respect to random choice of $\vec{u}_{k, -}$.
In the case where a drop step hasn't been taken and both $\operatorname{Tr}\left( \mathbf{P}_k^\bot \mathbf{H}_k\right ) \leq \frac{h_k}{4G}$ and $h_k < \min \{\epsilon_1, \epsilon_2, \epsilon_3, \epsilon_4\}$ hold, \Cref{lem:sfw-eigengap} implies that 
\begin{equation}\label{eqn:sfw-main-thm-proof-5}
    \delta_{k, r^\star} := \lambda_{n - r^\star + 1}(\nabla \phi(\mathbf{H}_k)) - \lambda_n(\nabla \phi(\mathbf{H}_k)) \geq \sqrt{\frac{h_k}{8r^\star}}.
\end{equation}
\Cref{lem:in-face-gap-upper-bound} implies that 
\begin{equation}\label{eqn:sfw-main-thm-proof-6}
    \delta_{k, r^\star} \leq 2\sqrt{2h_k}.
\end{equation}
\eqref{eqn:sfw-main-thm-proof-6} and \eqref{eqn:sfw-main-thm-proof-3} together imply that
\begin{equation}\label{eqn:sfw-main-thm-proof-7}
    \delta_{k, r^\star}  \leq 2\sqrt{2h_k} \leq 2\sqrt{2\epsilon_3} \leq \lambda_{r^\star} \leq 2\lambda_{r^\star}(\mathbf{H}_k).
\end{equation}
Given that $\lambda_{r^\star}(\mathbf{H}_k) \geq \frac{\delta_{k, r^\star}}{2} \geq \frac{\delta_{k, r^\star}}{16}$ and the fact that $r_k := \operatorname{rank}(\mathbf{H}_k) \geq r^\star$, the assumptions of  \Cref{lem:sfw-afw-step-eigengap} are satisfied.
Combining \Cref{lem:sfw-afw-step-eigengap} with \eqref{eqn:sfw-main-thm-proof-5}, we get
\begin{align}\label{eqn:sfw-main-thm-proof-8}
    \min \{\phi(\mathbf{H}_{k+1}^{\text{FW}}), \phi(\mathbf{H}_{k+1}^{\text{AFW}})\} - \phi(\mathbf{H}^\star) &\leq h_k - \frac{\delta_{k, r^\star}^2}{64} + \frac{\delta_{k, r^\star}\theta_k}{32R} \nonumber \\
    &\leq h_k \left ( 1 - \frac{1}{512r^\star} \right ) + \frac{\delta_{k, r^\star}\theta_k}{32R}.
\end{align}
Combining \eqref{eqn:sfw-main-thm-proof-4}, \eqref{eqn:sfw-main-thm-proof-6} and \eqref{eqn:sfw-main-thm-proof-8}, we get \eqref{eqn:sfw-rank-r-star-l-n-fw-rate}.

Finally, we can obtain \eqref{eqn:sfw-rank-n-fw-rate} for the special case where $r^\star = n$ by noting that if $h_k < \min\{\epsilon_1, \epsilon_2, \epsilon_3\}$, then \eqref{eqn:sfw-main-thm-proof-3} and \eqref{eqn:sfw-main-thm-proof-6} both hold and \eqref{eqn:sfw-main-thm-proof-7} follows from both of them, so the assumptions of  \Cref{lem:sfw-afw-step-eigengap} are again satisfied.
\eqref{eqn:sfw-rank-n-fw-rate} then follows similarly from \Cref{lem:sfw-afw-step-eigengap} with \eqref{eqn:sfw-main-thm-proof-5}.
\end{proof}
Given that \Cref{thm:sfw-convergence-inexact-evec} provides a recurrence relationship for the objective residual in \Cref{alg:spectrahedron-fw} given the rank $r^\star$ of the optimal projection for \eqref{eqn:proj-subproblem-H}, it now remains to determine the error tolerance schedule for the inexact eigenvector computation or to determine the appropriate rate of decay for $\{\theta_k\}_{k \geq 0}$ in \eqref{eqn:lmo-step-inexact}, which is clearly given in the following lemma:
\begin{lemma}\label{lem:geometric-recurrence-inexact} Let $0 \leq r < 1$.
Suppose that there are sequences $\{h_k\}_{k \geq 0}$ and $\{\vartheta_k\}_{k \geq 0}$ satisfy the following recurrence relationship:
\begin{equation*}
    h_{k + 1} \leq rh_k + \vartheta_k. 
\end{equation*}
Now suppose that $\vartheta_k \leq M s^k$ with $0 \leq s < 1$ for all $k$.
Then we get the following closed-form upper bound for the sequence $\{h_k\}_{k \geq 0}$:
\begin{equation}\label{eqn:d}
    h_k \leq r^k h_0 + Mr^{k - 1} \frac{1 - \left ( \frac{s}{r}\right )^k}{1 - \left ( \frac{s}{r}\right )}.
\end{equation}
Based on the above bound, we get the following five regimes:
\begin{enumerate}
    \item If $0 < s < r$, then $h_k = \mathcal{O}(r^k)$.
    \item If $s > r > 0$, then $h_k = \mathcal{O}(s^k)$.
    \item If $s = r$, then $h_k = \mathcal{O}(kr^k)$.
    \item If $r = 0$, then $h_k \leq \vartheta_{k-1} \leq Ms^{k-1}$ for $k \geq 1$. If $s > 0$, then $h_k = \mathcal{O}(s^k)$. If $s = 0$ and $r = 0$, then $h_k = 0$ for $k \geq 1$.
    \item If $s = 0$ and $r > 0$, then $h_k \leq r^k h_0 + Mr^{k-1}$.
\end{enumerate}
\end{lemma}
\begin{proof} If we unroll the recurrence relationship for $\{h_k\}_{k \geq 0}$, we get 
\begin{equation*}
    h_k \leq r^k h_0 + \sum_{j = 0}^{t-1}r^{k - 1 - j}\vartheta_j.
\end{equation*}
Substituting $\vartheta_k \leq Ms^k$ into the above, we get
\begin{equation*}
    h_k \leq r^k h_0 + Mr^{k - 1}\sum_{k = 0}^{k - 1}\left ( \frac{s}{r}\right )^j = r^k h_0 + Mr^{k - 1} \frac{1 - \left ( \frac{s}{r}\right )^k}{1 - \left ( \frac{s}{r}\right )}.
\end{equation*}
If $s < r$, then the second term decays at the rate $\mathcal{O}(r^k)$ and if $s > r$, then the second term decays at the rate $\mathcal{O}(s^k)$.
Since the first term decays at $\mathcal{O}(r^k)$, the dominant term determines the decay rate for the sequence $\{h_k\}_{k \geq 0}$.
This proves the desired result.
\end{proof}

\begin{theorem}\label{thm:lanczos-iter-sfw} 
Suppose that the Spectrahedron Frank-Wolfe algorithm (see \Cref{alg:spectrahedron-fw}) is applied to the projection subproblem given in \eqref{eqn:proj-subproblem-H} and that $\set{H} = \set{S}_+^n \cap \{\mathbf{X}~:~\operatorname{Tr}(\mathbf{X}) = R \}$. 
Suppose that Spectrahedron Frank-Wolfe converges at rate $h_k \leq \mathcal{O}(r^k)$ for some $r \in (0, 1)$ based on whether $r^\star = n$ or $n > r^\star \geq 2$ (see \Cref{thm:sfw-convergence-inexact-evec} and our earlier note about $r^\star = 1$ after \Cref{lem:obj-residual-sfw-rank-1}) when $\theta_k = 0$ (i.e. in the exact case).
For any prescribed overall failure probability $p \in (0,1)$ and any horizon $K \geq 1$, choose $\xi_k= \mathcal{O}(s^k)$  for some $s \in (0, 1)$ and run the randomized Lanczos method with per-call failure probability at most $p/K$ for each of the first $K$ iterations. Then, with probability of at least $1-p$,
the Spectrahedron Frank-Wolfe algorithm will converge to the projection $\mathbf{H}^\star$ of $\mathbf{G}$ onto $\set{H}$ at the rate
\begin{equation*}
h_k=
\begin{cases}
\mathcal{O}(r^k),&s<\sqrt r,\\
\mathcal{O}(kr^k),&s=\sqrt r,\\
\mathcal{O}((s\sqrt r)^k),&s>\sqrt r.
\end{cases}
\end{equation*}
for $1 \leq k \leq K$ after at least $\mathcal{O}\left(s^{-K/2}\log\left(\frac{nK}{p}\right)\right )$ Lanczos iterations per Frank-Wolfe iteration.
\end{theorem}
\begin{proof}
Given \eqref{eqn:lanczos-accuracy-bound} and \eqref{eqn:lanczos-max-iter-lem-proof-1}, then \eqref{eqn:lanczos-max-iter-lem-proof-2} has to hold:
\begin{equation*}
    \theta_k \leq \frac{R\xi_k \norm{\nabla \phi(\mathbf{H}_k)}_F}{8},
\end{equation*}
where $\xi_k \in (0, 1]$.
Suppose we choose $\xi_k = \mathcal{O}(s^k)$ for some $0 \leq s < 1$.
If $h_k = \mathcal{O}(r^k)$ for $0 < r < 1$, then
\begin{equation*}
\sqrt{h_k} = \mathcal{O}(r^{k/2}).
\end{equation*}
Since $\theta_k = \mathcal{O}(s^k)$, we obtain the following bound
\begin{equation*}
    \sqrt{h_k}\theta_k = \mathcal{O}\left((s\sqrt{r})^k\right).
\end{equation*}
Consequently, in the rank-$n$ and intermediate-rank cases, the
additive error term in the recurrence satisfies
\begin{equation*}
    \frac{\sqrt{2h_k}}{16R}\theta_k
    = \mathcal{O}\left((s\sqrt{r})^k\right).
\end{equation*}
Applying \Cref{lem:geometric-recurrence-inexact} with contraction
factor $r$ and perturbation factor $s\sqrt{r}$, we obtain
\begin{equation*}
h_k =
\begin{cases}
\mathcal{O}(r^k), & s\sqrt{r} < r,\\
\mathcal{O}(kr^k), & s\sqrt{r} = r,\\
\mathcal{O}\left((s\sqrt{r})^k\right), & s\sqrt{r} > r.
\end{cases}
\end{equation*}
Next, we note that 
\begin{equation*}
    \norm{\nabla \phi(\mathbf{H}_k)}_F \leq \norm{\mathbf{H}_k - \mathbf{G}}_F \leq \norm{\mathbf{H}_k}_F + \norm{\mathbf{G}}_F \leq R + \norm{\mathbf{G}}_F,
\end{equation*}
so $\norm{\nabla \phi(\mathbf{H}_k)}_F$ is bounded.
Given \eqref{eqn:lanczos-max-iter-lem-proof-2} and the fact that $\norm{\nabla \phi(\mathbf{H}_k)}_F$ is bounded, we must have that $\theta_k = \mathcal{O}(s^k)$.
Applying \Cref{lem:geometric-recurrence-inexact}, we obtain the geometric convergence of Spectrahedron FW.

Now to prove the lower bound on the maximum number of Lanczos iterations, we simply invoke \Cref{lemma:randomized-lanczos} as follows.
For the $k$th iteration out of $K$ iterations, assume that the probability of failing to find a suitable approximation of the eigenvector corresponding to the smallest eigenvalue $\lambda_{min}(\nabla \phi(\mathbf{H}_k))$ is $p_k$.
Then the probability that any one of the $K$ iterations fails to find the minimum eigenvector is given by $\sum_{k = 1}^K p_k$.
To ensure that the overall probability of failure is $p$, we can enforce that $p_k = p/K$.
Given the objective of estimating the eigenvector with constant probability of $1 - 2\zeta$, \eqref{eqn:lanczos-accuracy-bound} entails that $p_k = 2\zeta$. 
Therefore, we must have that $\zeta = \frac{p}{2K}$.
Given this and the fact that $\xi_k = \mathcal{O}(s^k)$, we can say that 
\begin{equation*}
    q_k \geq \mathcal{O}\left( ( s^k )^{-1/2}\log\left ( \frac{n}{(p/(2K))^2}\right ) \right ) = \mathcal{O}\left (s^{-k/2}\log\left(\frac{nK}{p}\right ) \right ).
\end{equation*}
Consequently, the worst-case total number of Lanczos iterations per Frank-Wolfe iteration is bounded below by $\mathcal{O}\left(s^{-K/2}\log\left(\frac{nK}{p}\right)\right )$. 
This concludes the proof.
\end{proof}

\section{Numerical Tuning and Additional Results}
\label{app:numerics}

\subsection{Tuning the PDHG step sizes}

We first tune the primal--dual step sizes on G38, G55, G60 and G70. 
The plots compare pairs $(\tau,\sigma)$ with $\tau\sigma=0.5$. The behavior depends on the graph: several choices make progress on G38, while the contrast is clearer on G55, G60 and G70, where $(\tau,\sigma)=(1,0.5)$ produces the strongest decrease in both objective residual and feasibility error over the displayed iteration range. 
We therefore use $(\tau,\sigma)=(1,0.5)$ for the subsequent comparisons. 
This is a consistent choice across the four instances rather than a claim that it is the pointwise best pair for every metric on every
graph.

\begin{figure}[t]
\centering
\begin{subfigure}{0.48\textwidth}
    \centering
    \includegraphics[width=\linewidth]{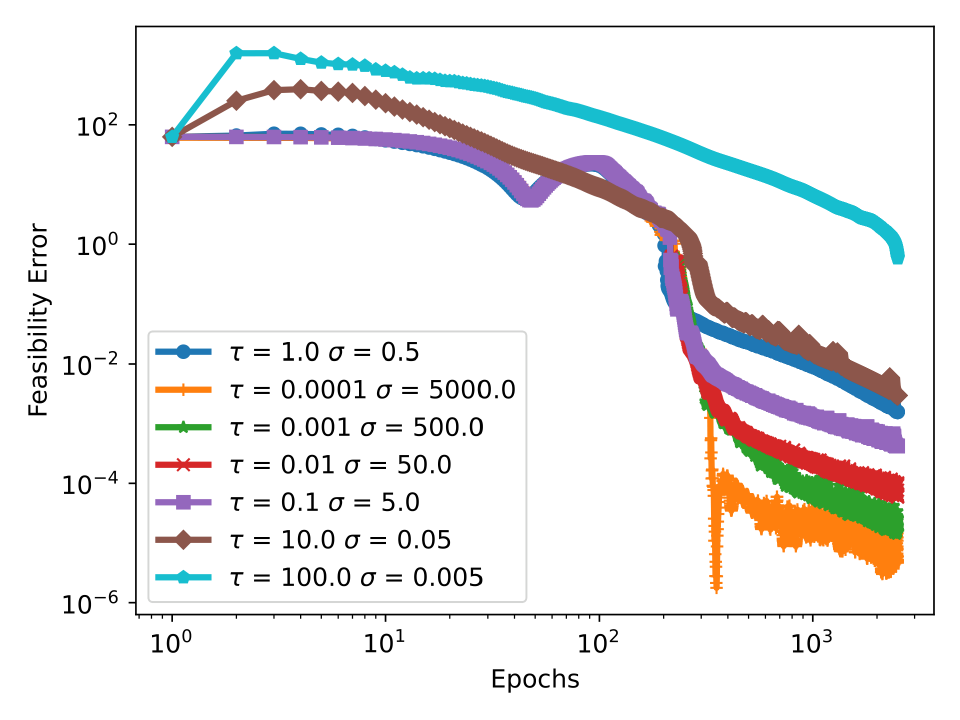}
    \caption{Feasibility error on G38.}
\end{subfigure}
\hfill
\begin{subfigure}{0.48\textwidth}
    \centering
    \includegraphics[width=\linewidth]{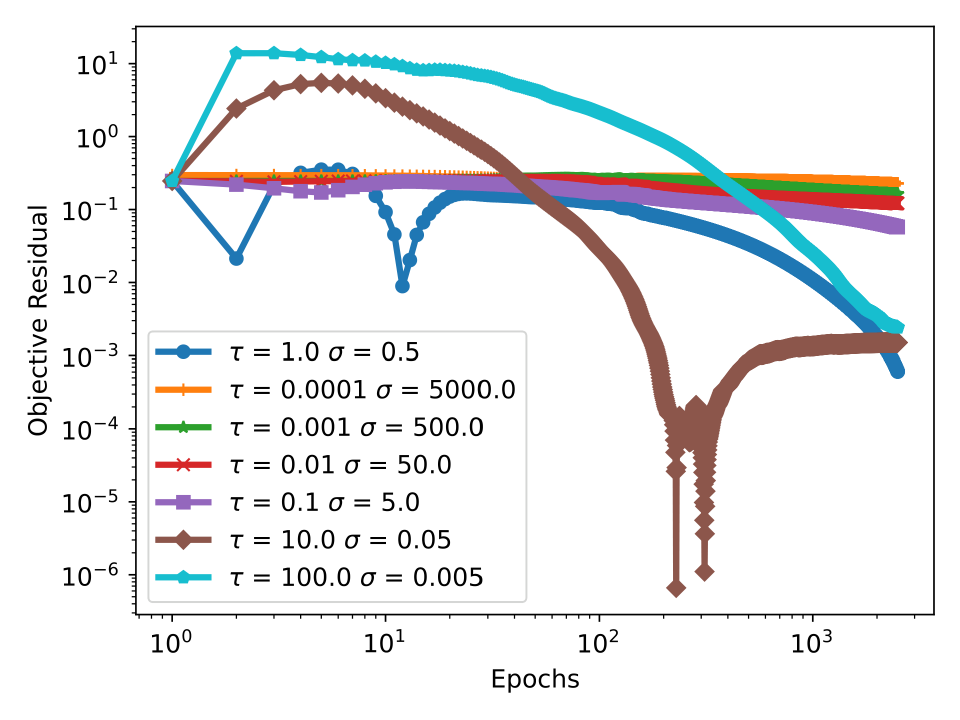}
    \caption{Objective residual on G38.}
\end{subfigure}
\caption{PDHG step-size tuning on G38. The tested pairs satisfy
$\tau\sigma=0.5$.}
\label{fig:pdhg_tuning_g38}
\end{figure}

\begin{figure}[t]
\centering
\begin{subfigure}{0.48\textwidth}
    \centering
    \includegraphics[width=\linewidth]{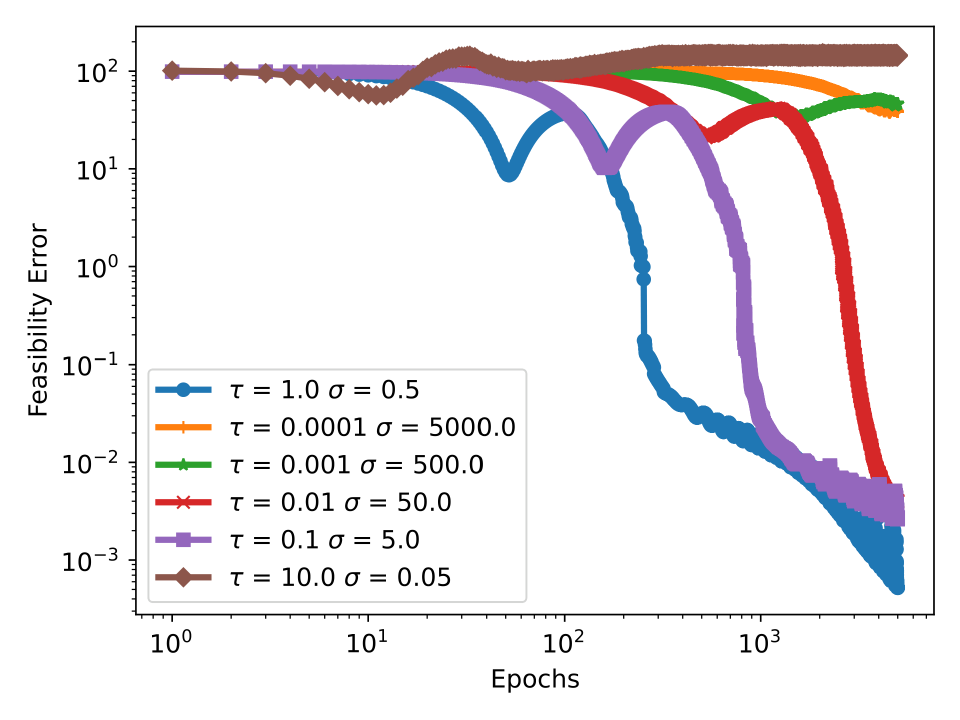}
    \caption{Feasibility error on G55.}
\end{subfigure}
\hfill
\begin{subfigure}{0.48\textwidth}
    \centering
    \includegraphics[width=\linewidth]{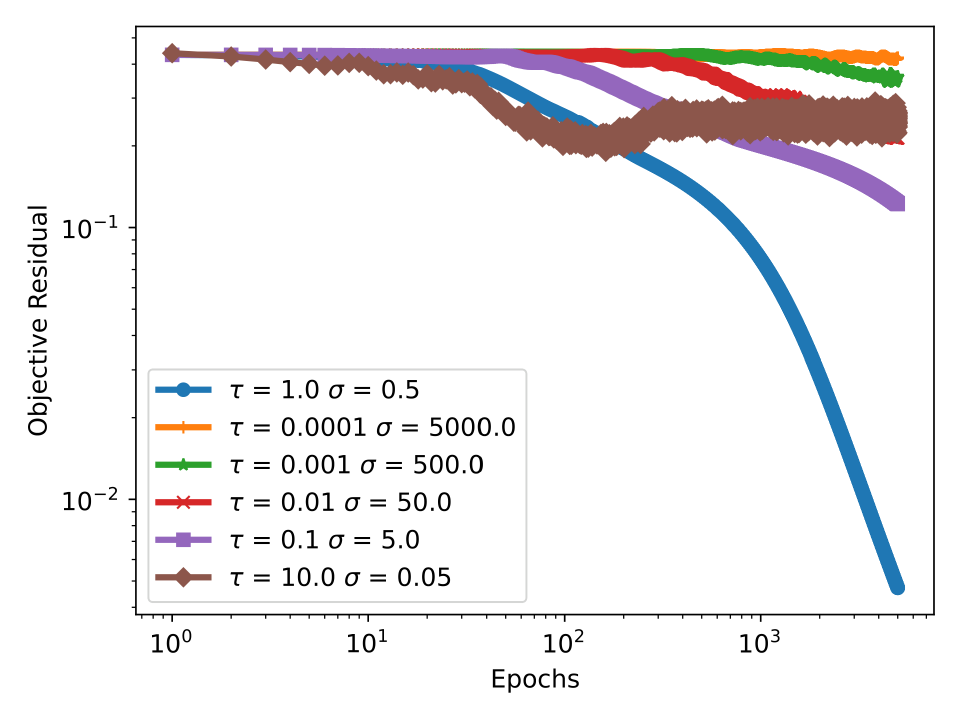}
    \caption{Objective residual on G55.}
\end{subfigure}
\caption{PDHG step-size tuning on G55. The pair $(\tau,\sigma)=(1,0.5)$
gives the clearest simultaneous reduction of objective residual and
feasibility error among the displayed choices.}
\label{fig:pdhg_tuning_g55}
\end{figure}

\begin{figure}[t]
\centering
\begin{subfigure}{0.48\textwidth}
    \centering
    \includegraphics[width=\linewidth]{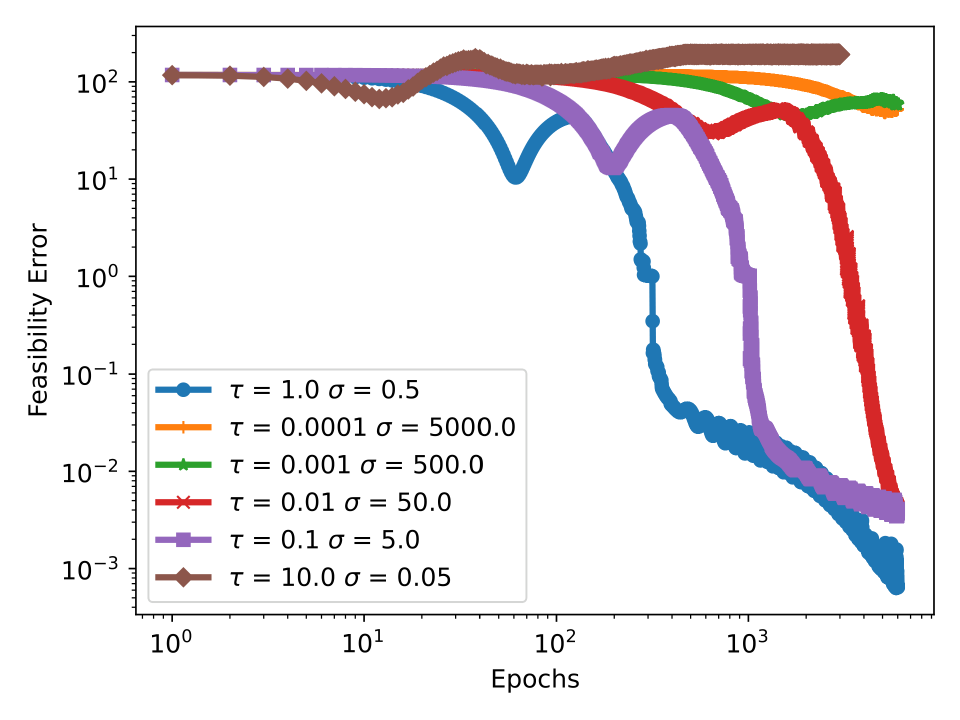}
    \caption{Feasibility error on G60.}
\end{subfigure}
\hfill
\begin{subfigure}{0.48\textwidth}
    \centering
    \includegraphics[width=\linewidth]{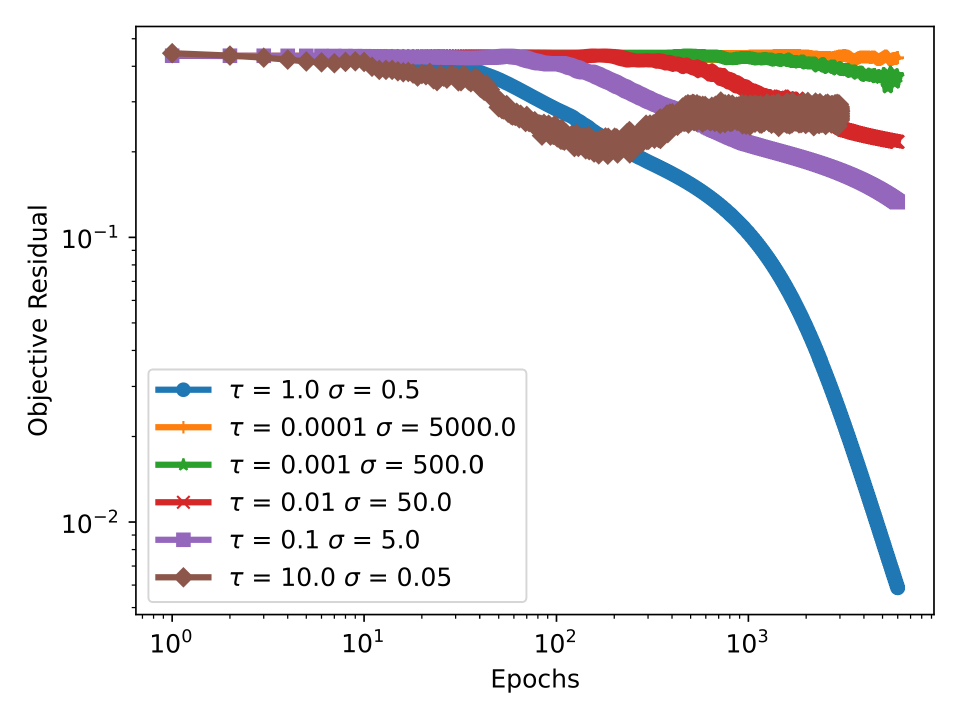}
    \caption{Objective residual on G60.}
\end{subfigure}
\caption{PDHG step-size tuning on G60. As on G55, $(\tau,\sigma)=(1,0.5)$
is the most effective of the plotted choices over the displayed horizon.}
\label{fig:pdhg_tuning_g60}
\end{figure}

\begin{figure}[t]
\centering
\begin{subfigure}{0.48\textwidth}
    \centering
    \includegraphics[width=\linewidth]{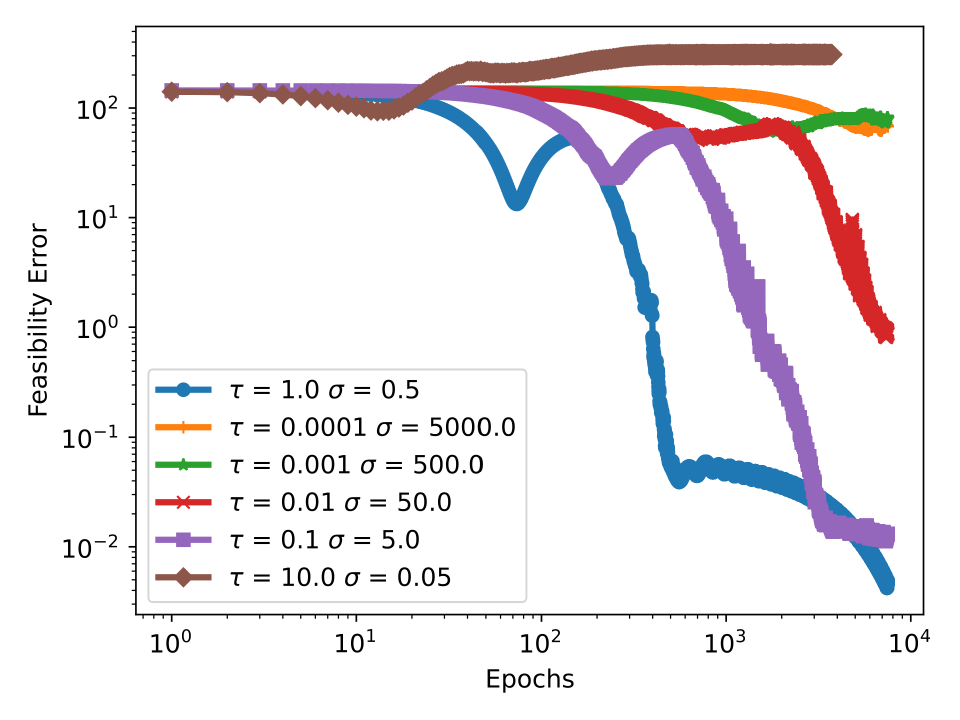}
    \caption{Feasibility error on G70.}
\end{subfigure}
\hfill
\begin{subfigure}{0.48\textwidth}
    \centering
    \includegraphics[width=\linewidth]{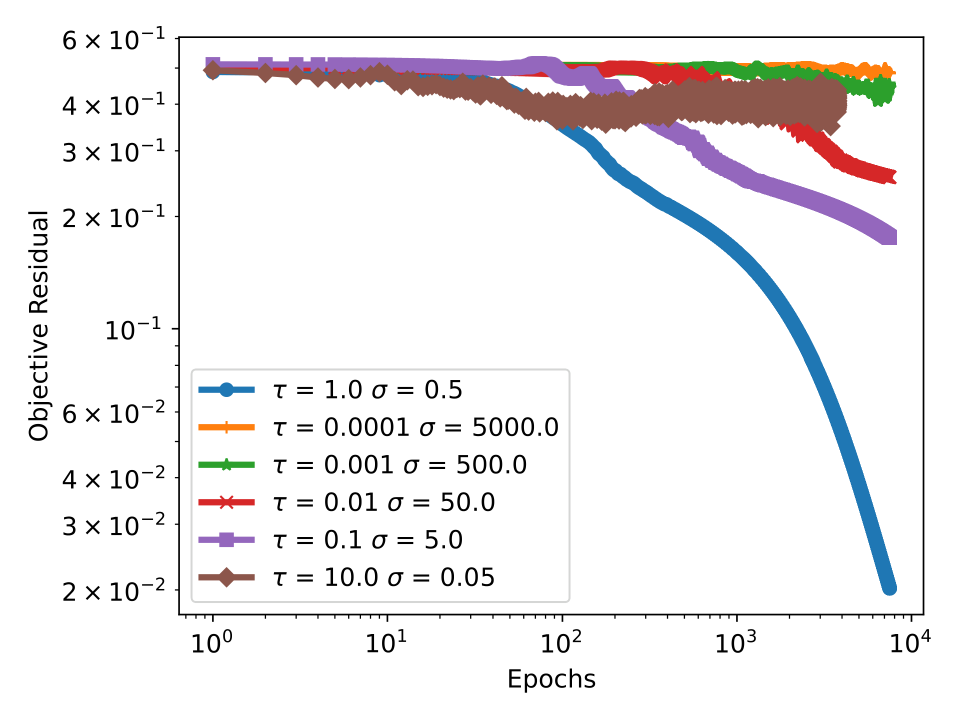}
    \caption{Objective residual on G70.}
\end{subfigure}
\caption{PDHG step-size tuning on G70. As on G60, $(\tau,\sigma)=(1,0.5)$
is the most effective of the plotted choices over the displayed horizon.}
\label{fig:pdhg_tuning_g70}
\end{figure}

\subsection{Sensitivity to the inexact-projection parameter
\texorpdfstring{$\rho$}{rho}}

We next vary $\rho$ in the adaptive error rule for Inexact PDHG with spectral warm starts and SFW, keeping $(\tau,\sigma)=(1,0.5)$ fixed. 
The objective residual and feasibility curves for a given instance are nearly indistinguishable over a broad range of $\rho$, both against outer iterations and against cumulative matvecs. 
This visual similarity is reflected in the
total-work counts in Tables~\ref{tab:rho_inner_work} and \ref{tab:rho_matvec_work}: the best value of $\rho$ differs by instance, but the total work changes only moderately over the tested range.
For this reason, $\rho$ was not tuned for G70.

\begin{figure}[t]
\centering
\begin{subfigure}{0.48\textwidth}
    \centering
    \includegraphics[width=\linewidth]{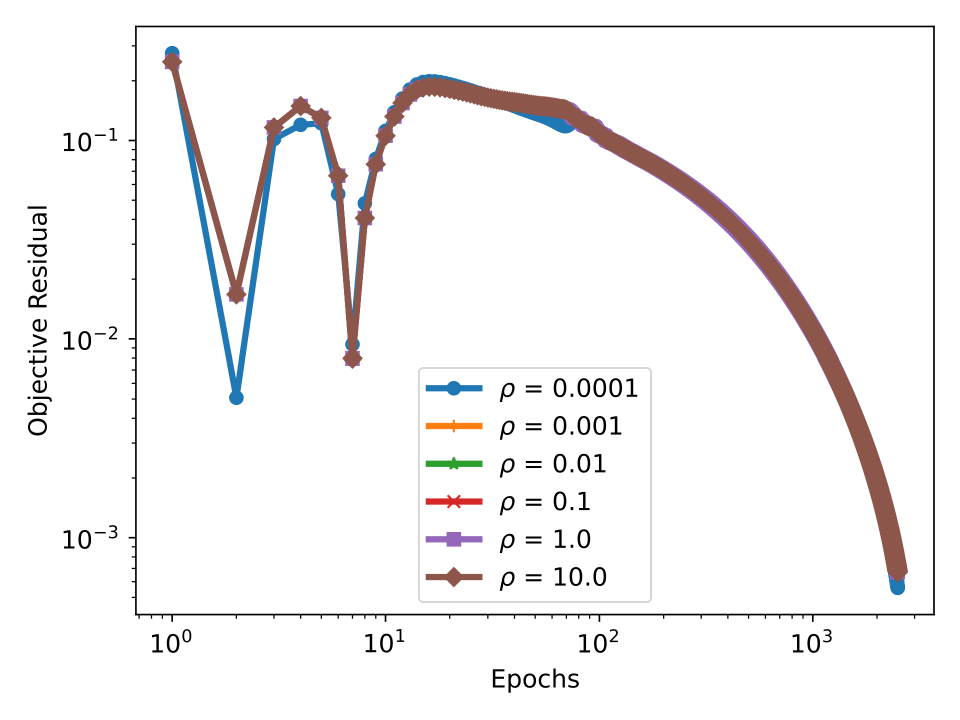}
    \caption{Objective residual vs. iterations.}
\end{subfigure}
\hfill
\begin{subfigure}{0.48\textwidth}
    \centering
    \includegraphics[width=\linewidth]{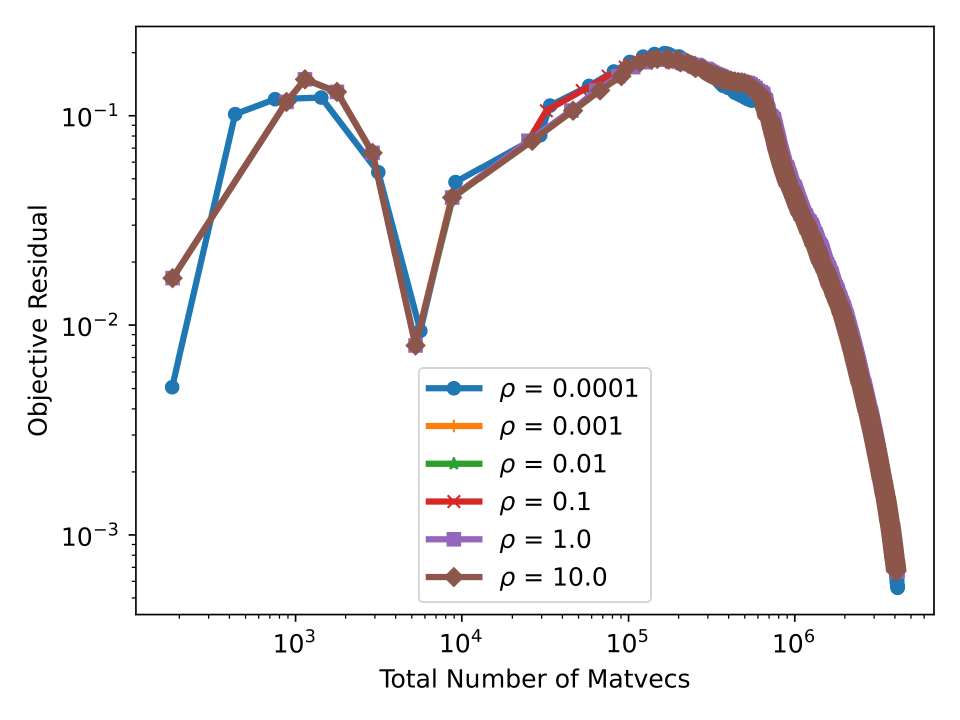}
    \caption{Objective residual vs. matvecs.}
\end{subfigure}

\vspace{0.2cm}
\begin{subfigure}{0.48\textwidth}
    \centering
    \includegraphics[width=\linewidth]{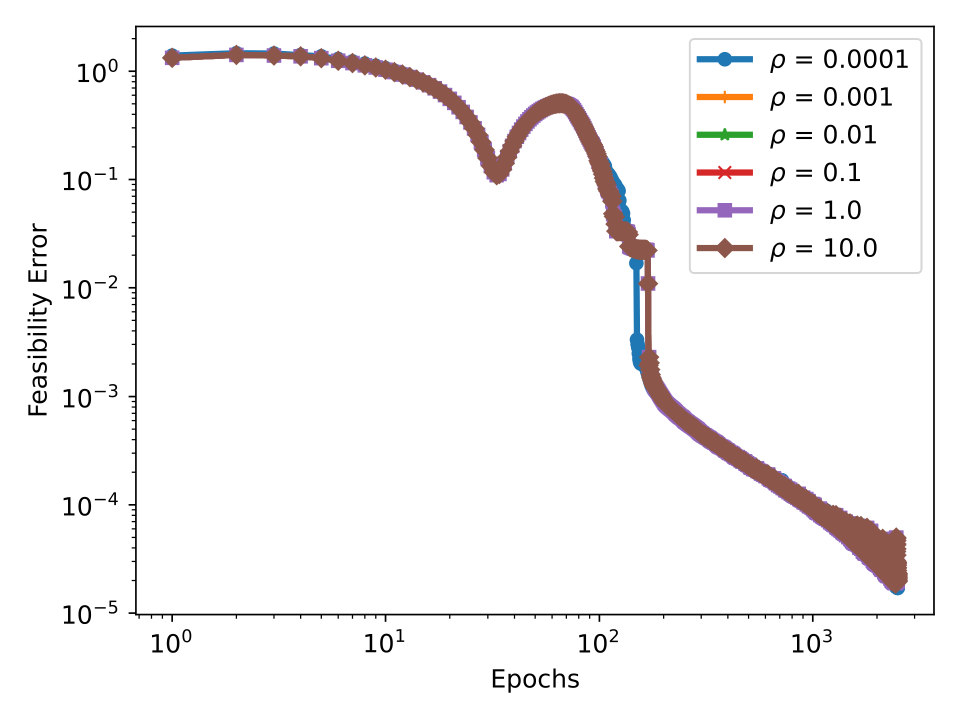}
    \caption{Feasibility error vs. iterations.}
\end{subfigure}
\hfill
\begin{subfigure}{0.48\textwidth}
    \centering
    \includegraphics[width=\linewidth]{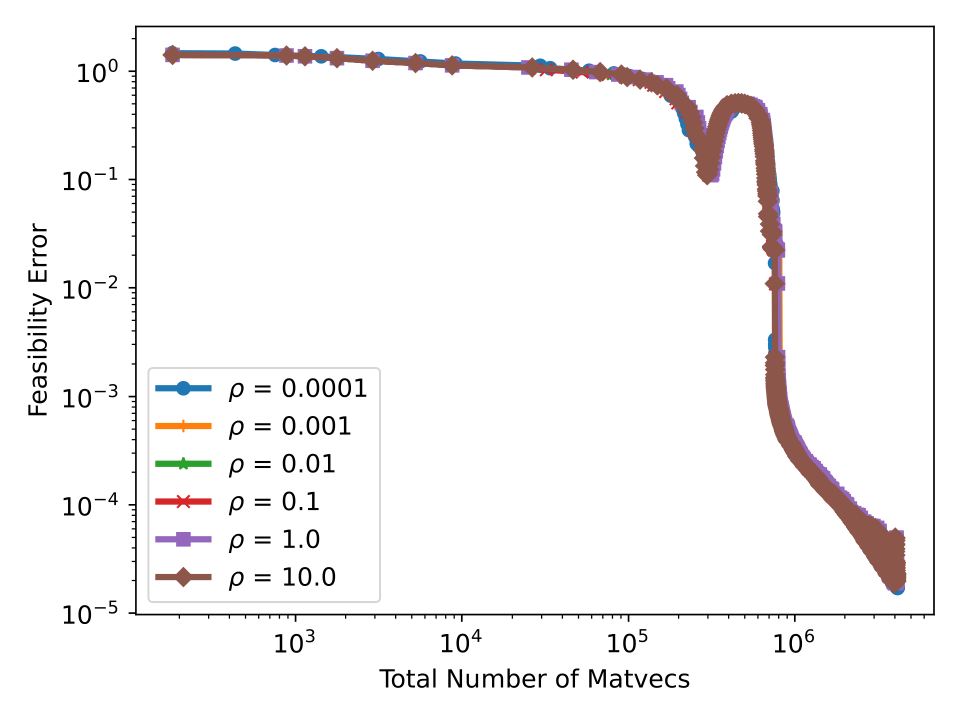}
    \caption{Feasibility error vs. matvecs.}
\end{subfigure}
\caption{Sensitivity to $\rho$ on G38. The convergence curves largely overlap
for the tested values of $\rho$.}
\label{fig:rho_tuning_g38}
\end{figure}

\begin{figure}[t]
\centering
\begin{subfigure}{0.48\textwidth}
    \centering
    \includegraphics[width=\linewidth]{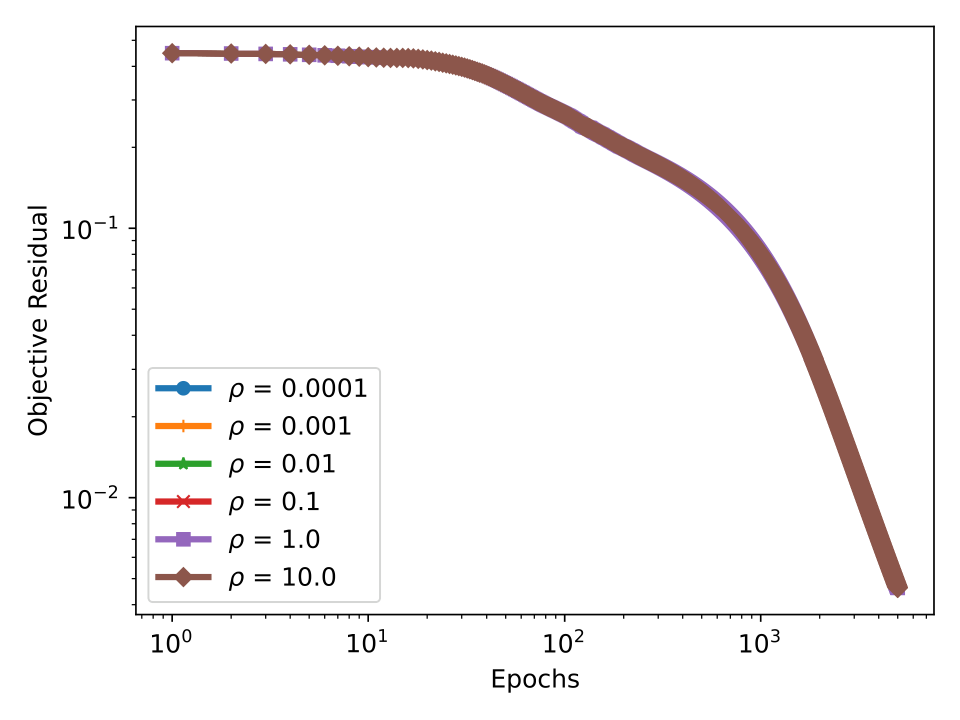}
    \caption{Objective residual vs. iterations.}
\end{subfigure}
\hfill
\begin{subfigure}{0.48\textwidth}
    \centering
    \includegraphics[width=\linewidth]{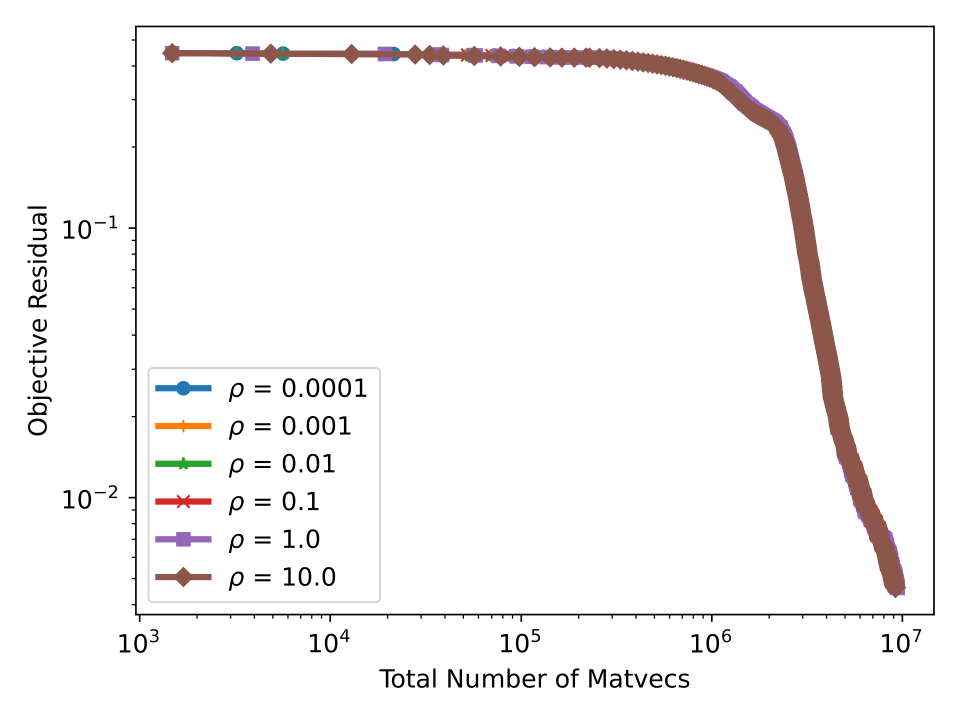}
    \caption{Objective residual vs. matvecs.}
\end{subfigure}

\vspace{0.2cm}
\begin{subfigure}{0.48\textwidth}
    \centering
    \includegraphics[width=\linewidth]{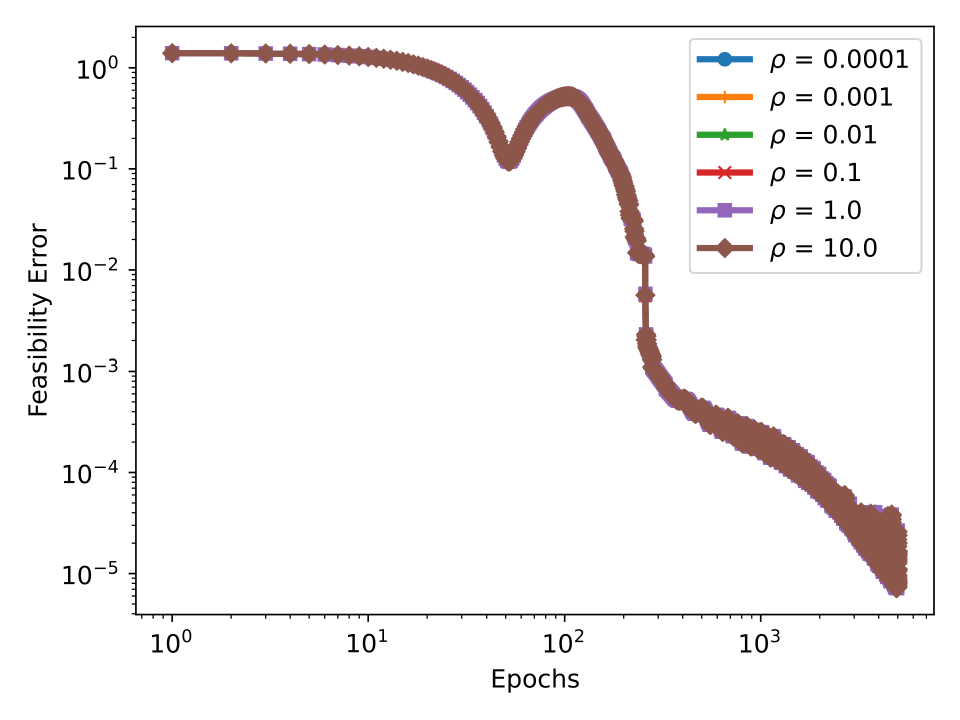}
    \caption{Feasibility error vs. iterations.}
\end{subfigure}
\hfill
\begin{subfigure}{0.48\textwidth}
    \centering
    \includegraphics[width=\linewidth]{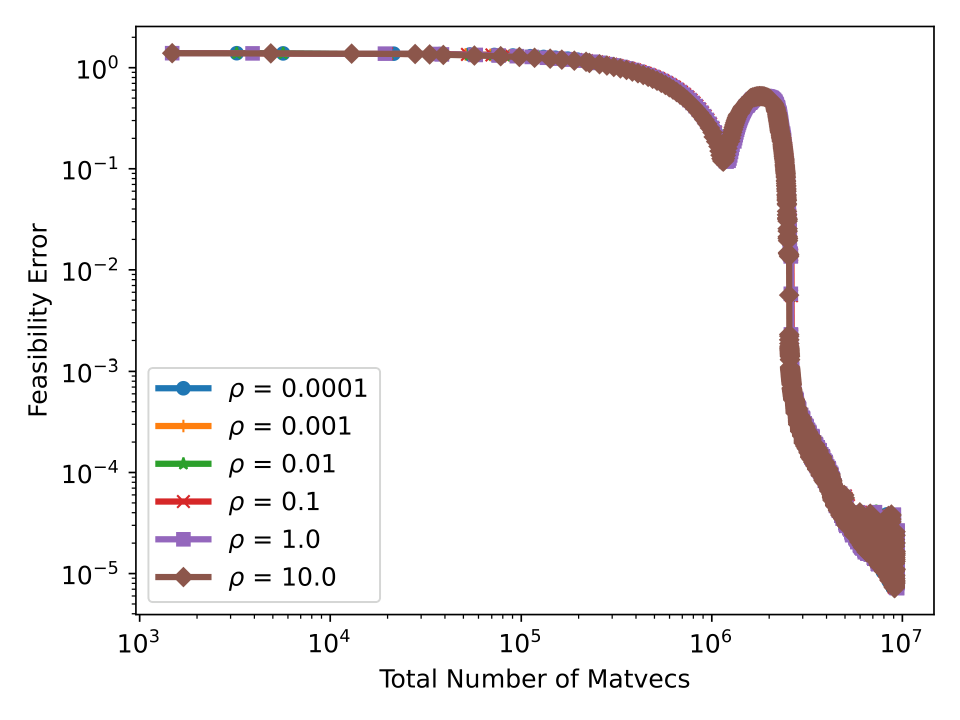}
    \caption{Feasibility error vs. matvecs.}
\end{subfigure}
\caption{Sensitivity to $\rho$ on G55. The objective and feasibility
trajectories remain very similar across the tested values.}
\label{fig:rho_tuning_g55}
\end{figure}

\begin{figure}[t]
\centering
\begin{subfigure}{0.48\textwidth}
    \centering
    \includegraphics[width=\linewidth]{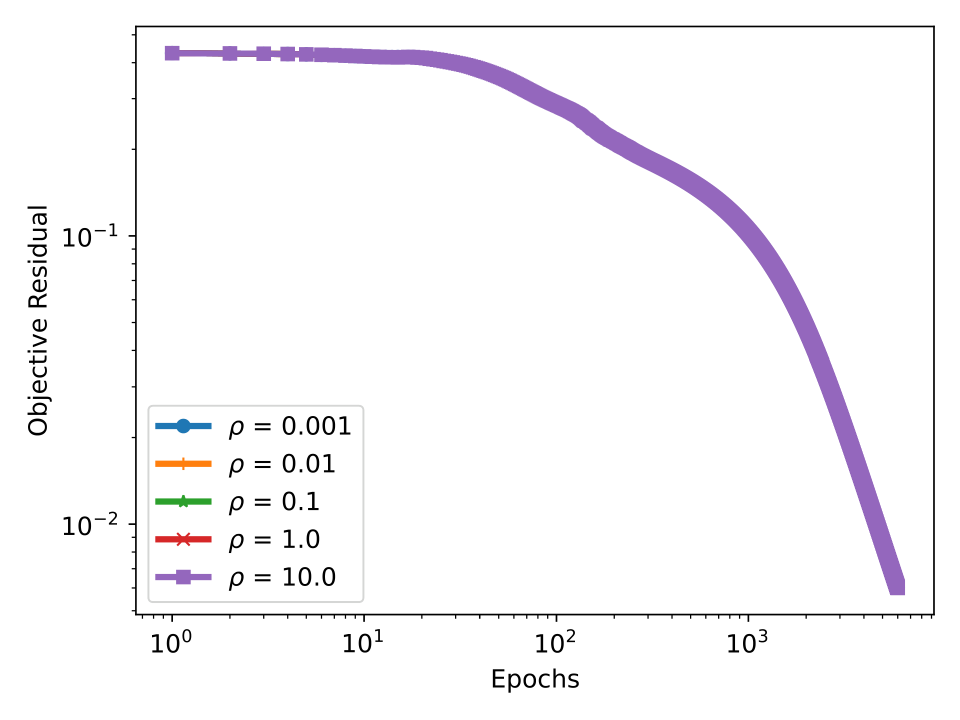}
    \caption{Objective residual vs. iterations.}
\end{subfigure}
\hfill
\begin{subfigure}{0.48\textwidth}
    \centering
    \includegraphics[width=\linewidth]{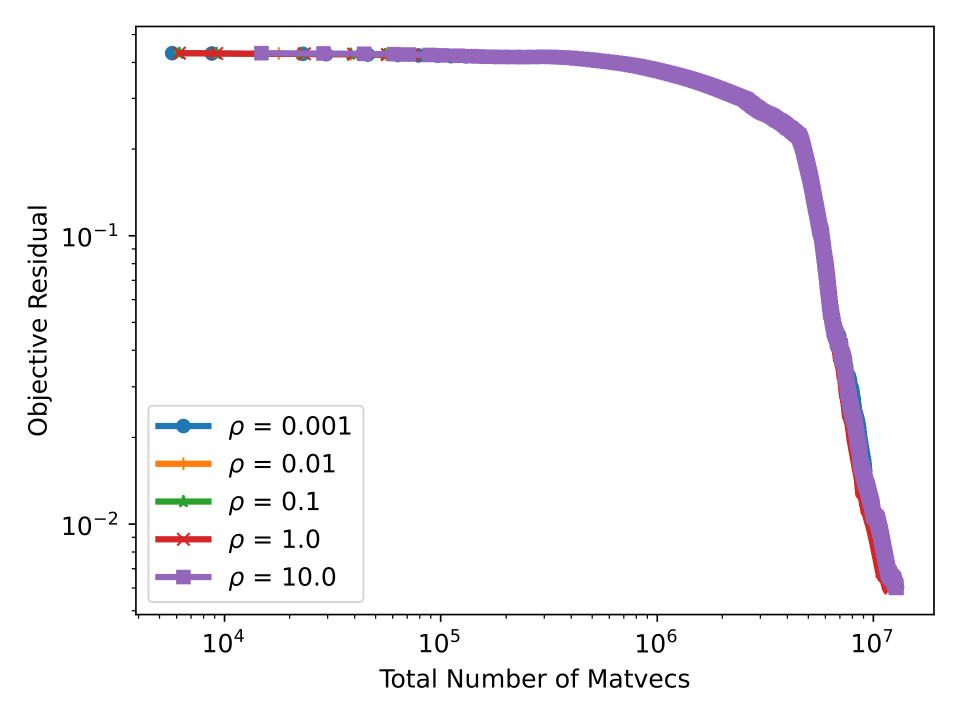}
    \caption{Objective residual vs. matvecs.}
\end{subfigure}

\vspace{0.2cm}
\begin{subfigure}{0.48\textwidth}
    \centering
    \includegraphics[width=\linewidth]{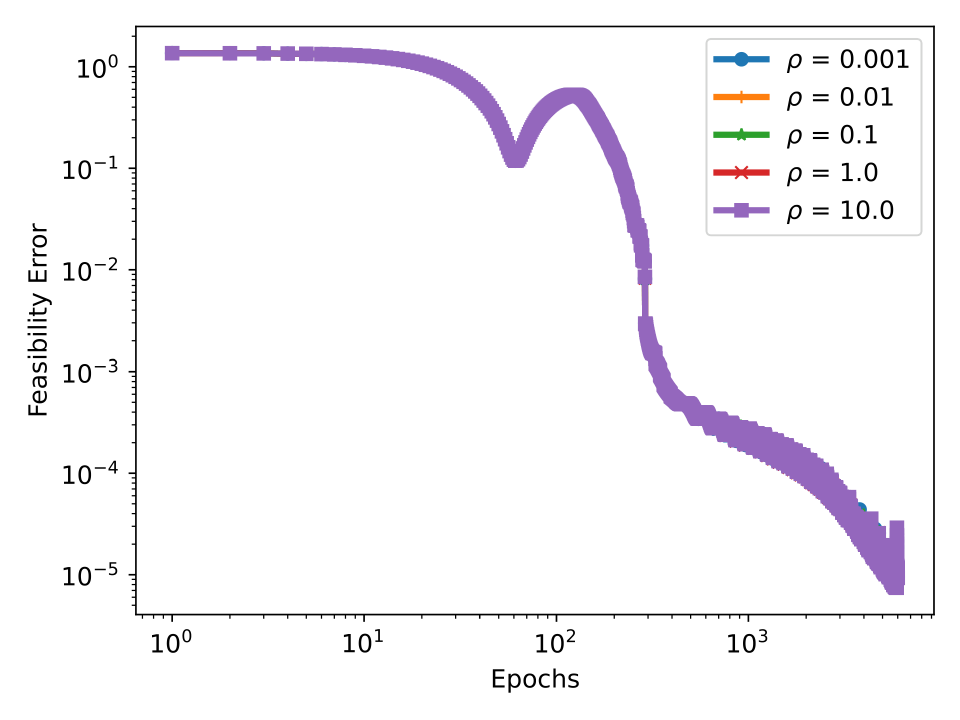}
    \caption{Feasibility error vs. iterations.}
\end{subfigure}
\hfill
\begin{subfigure}{0.48\textwidth}
    \centering
    \includegraphics[width=\linewidth]{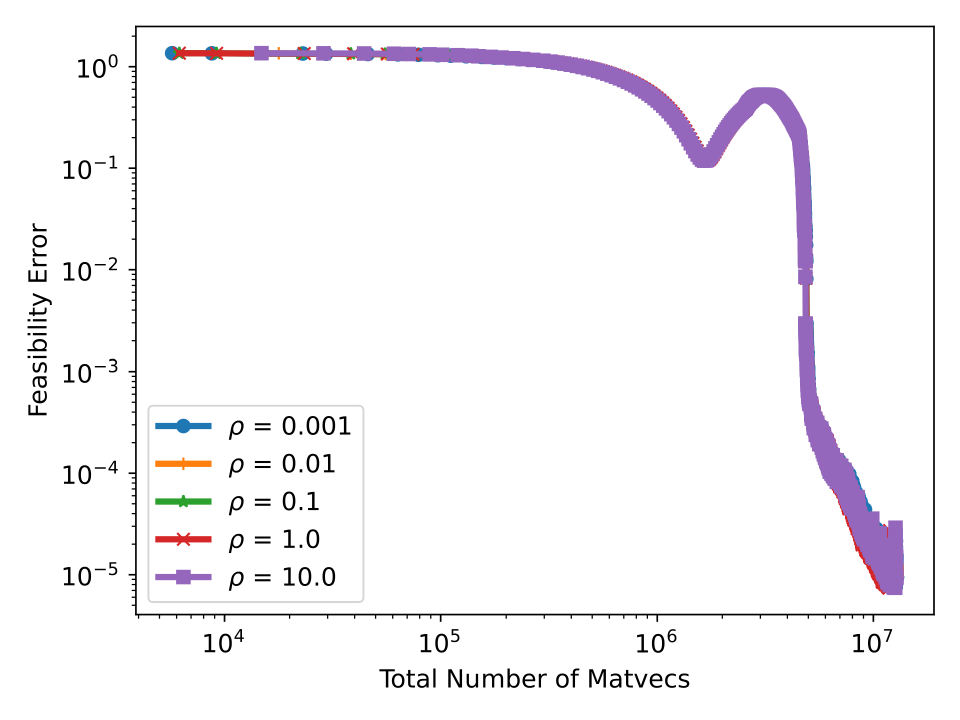}
    \caption{Feasibility error vs. matvecs.}
\end{subfigure}
\caption{Sensitivity to $\rho$ on G60. The displayed values of $\rho$ yield
nearly overlapping convergence curves. The $\rho=10^{-4}$ run was not
available for this instance.}
\label{fig:rho_tuning_g60}
\end{figure}

\begin{table}[t]
\centering
\caption{Total inner SFW iterations as $\rho$ varies. Each entry gives the
count and its percentage excess over the minimum for that instance. The
minimum in each row is boldfaced.}
\label{tab:rho_inner_work}
\scriptsize
\setlength{\tabcolsep}{3.5pt}
\begin{tabular}{lrrrrrr}
\toprule
Instance & $\rho=10^{-4}$ & $\rho=10^{-3}$ & $\rho=10^{-2}$
& $\rho=10^{-1}$ & $\rho=1$ & $\rho=10$ \\
\midrule
G38 &
\shortstack{3622\\(+1.1\%)} &
\shortstack{3721\\(+3.9\%)} &
\shortstack{3683\\(+2.8\%)} &
\shortstack{3638\\(+1.6\%)} &
\shortstack{3695\\(+3.2\%)} &
\shortstack{\textbf{3581}\\\textbf{(+0.0\%)}} \\
G55 &
\shortstack{\textbf{10059}\\\textbf{(+0.0\%)}} &
\shortstack{10197\\(+1.4\%)} &
\shortstack{10267\\(+2.1\%)} &
\shortstack{10425\\(+3.6\%)} &
\shortstack{10875\\(+8.1\%)} &
\shortstack{10356\\(+3.0\%)} \\
G60 &
-- &
\shortstack{13472\\(+9.3\%)} &
\shortstack{13860\\(+12.4\%)} &
\shortstack{13540\\(+9.8\%)} &
\shortstack{\textbf{12330}\\\textbf{(+0.0\%)}} &
\shortstack{13359\\(+8.3\%)} \\
\bottomrule
\end{tabular}
\end{table}

\begin{table}[t]
\centering
\caption{Total projection matrix--vector products as $\rho$ varies. Each
entry gives the count in millions and its percentage excess over the minimum
for that instance. The minimum in each row is boldfaced.}
\label{tab:rho_matvec_work}
\scriptsize
\setlength{\tabcolsep}{3.5pt}
\begin{tabular}{lrrrrrr}
\toprule
Instance & $\rho=10^{-4}$ & $\rho=10^{-3}$ & $\rho=10^{-2}$
& $\rho=10^{-1}$ & $\rho=1$ & $\rho=10$ \\
\midrule
G38 &
\shortstack{4.153\\(+1.9\%)} &
\shortstack{4.160\\(+2.1\%)} &
\shortstack{4.141\\(+1.6\%)} &
\shortstack{4.110\\(+0.9\%)} &
\shortstack{4.141\\(+1.6\%)} &
\shortstack{\textbf{4.074}\\\textbf{(+0.0\%)}} \\
G55 &
\shortstack{\textbf{9.047}\\\textbf{(+0.0\%)}} &
\shortstack{9.123\\(+0.8\%)} &
\shortstack{9.164\\(+1.3\%)} &
\shortstack{9.270\\(+2.5\%)} &
\shortstack{9.456\\(+4.5\%)} &
\shortstack{9.208\\(+1.8\%)} \\
G60 &
-- &
\shortstack{12.838\\(+6.5\%)} &
\shortstack{12.963\\(+7.5\%)} &
\shortstack{12.827\\(+6.4\%)} &
\shortstack{\textbf{12.055}\\\textbf{(+0.0\%)}} &
\shortstack{12.767\\(+5.9\%)} \\
\bottomrule
\end{tabular}
\end{table}

The total-work tables confirm that the tuning parameter does not materially
change the overall computational cost in most configurations. For G38, the
matvec count varies by at most $2.1\%$ from its minimum; for G55, by at most
$4.5\%$. G60 shows somewhat greater sensitivity, with a maximum deviation
of $7.5\%$ in matvecs and $12.4\%$ in inner iterations. Since the convergence
curves also nearly overlap, we use $\rho=0.01$ as a representative setting
rather than selecting a different value for each graph.
 
\end{document}